\documentclass[11pt]{article}
\usepackage{epigamath}

\usepackage[notext]{kpfonts}
\usepackage{baskervald}

\setpapertype{A4}

\usepackage[english]{babel}

\usepackage[dvips]{graphicx}     

\removelength{1cm}

\title{Isomorphisms between complements of \\ projective plane curves}
\titlemark{Isomorphisms between complements of projective plane curves}
\author{Mattias Hemmig}
\authoraddresses{
\authordata{Mattias Hemmig}{\firstname{Mattias} \lastname{Hemmig}\\
\institution{Universit\"{a}t Basel, Departement Mathematik und Informatik, Spiegelgasse $1$, CH-$4051$ Basel, Switzerland}\\
\email{mattias.hemmig@gmail.com}}
}
\authormark{M. Hemmig}
\date{\vspace{-5ex}} 
\journal{\'Epijournal de G\'eom\'etrie Alg\'ebrique} 
\acceptation{Received by the Editors on June 3, 2019. Accepted on September 24, 2019.}

\newcommand{\articlereference}{Volume 3 (2019), Article Nr. 16}

\acknowledgement{The author gratefully acknowledges support by the Swiss National Science Foundation Grant ``Curves in the spaces'' $200021\_169508$}

\usepackage[all]{xy}

\allowdisplaybreaks

\newdimen\origiwspc
\origiwspc=\fontdimen2\font

\usepackage{float}

\usepackage[all]{xy}
\usepackage{tikz-cd}
\usepackage{float}
\usetikzlibrary{decorations.pathreplacing}
\usepackage{setspace}

\usepackage{ulem} 
\renewcommand{\emph}[1]{{\it #1}}

\usepackage{enumerate}

\newcommand\N{{\mathbb N}}

\newcommand\A{{\mathbb A}}

\newcommand\Z{{\mathbb Z}}

\newcommand\Pic{\mathrm{ Pic}}

\newcommand\tr{\hbox to 1mm  {${}^t \!  $} }

\newcommand\p{{\mathbb P}}

\renewcommand\k{\mathrm{k}}

\DeclareMathOperator{\Aut}{Aut}
\DeclareMathOperator{\Aff}{Aff}

\DeclareMathOperator{\PGL}{PGL}

\DeclareMathOperator{\Loc}{Loc}
\DeclareMathOperator{\Jon}{Jon}

\DeclareMathOperator{\id}{id}
\DeclareMathOperator{\im}{im}

\DeclareMathOperator{\Sing}{Sing}

\def\dashmapsto{\mapstochar\dashrightarrow}

\newtheorem{theorem}{Theorem}
\newtheorem{lemma}{Lemma}[section]
\newtheorem{corollary}[lemma]{Corollary}
\newtheorem{proposition}[lemma]{Proposition}
\newtheorem{conjecture}[lemma]{Conjecture}

\newtheorem{thm}[lemma]{Theorem}

\newtheorem{definition}[lemma]{Definition}

\newtheorem{remark}[lemma]{Remark}

\newenvironment{claim*}[1][]{\paragraph{\bf Claim~$(#1)$.}}{}

\newenvironment{subproof}[1][\proofname]{%
 \renewcommand{\qedsymbol}{$\blacksquare$}%
  \begin{proof}[#1]%
}{%
  \end{proof}%
 \renewcommand{\qedsymbol}{$\Box$}%
}

\begin{document}


\maketitle



\begin{prelims}

\vspace{-0.55cm}

\def\abstractname{Abstract}
\abstract{In this article, we study isomorphisms between complements of irreducible curves in the projective plane $\p^2$, over an arbitrary algebraically closed field. Of particular interest are rational unicuspidal curves. We prove that if there exists a line that intersects a unicuspidal curve $C \subset \p^2$ only in its singular point, then any other curve whose complement is isomorphic to $\p^2 \setminus C$ must be projectively equivalent to $C$. This generalizes a result of H. Yoshihara who proved this result over the complex numbers. Moreover, we study properties of multiplicity sequences of irreducible curves that imply that any isomorphism between the complements of these curves extends to an automorphism of $\p^2$. Using these results, we show that two irreducible curves of degree $\leq 7$ have isomorphic complements if and only if they are projectively equivalent. Finally, we describe new examples of irreducible projectively non-equivalent curves of degree $8$ that have isomorphic complements.}

\keywords{Plane curves, curves of low degree, unicuspidal curves, complements of plane curves}

\MSCclass{14E07; 14H45; 14H50; 14J26}


\languagesection{Fran\c{c}ais}{%

\vspace{-0.05cm}
{\bf Titre. Isomorphismes entre compl\'ementaires de courbes projectives planes} \commentskip 
{\bf R\'esum\'e.} Dans cet article, nous \'etudions les isomorphismes entre compl\'ementaires de courbes irr\'eductibles dans le plan projectif $\p^2$ sur un corps alg\'ebriquement clos quelconque. Les courbes rationnelles unicuspidales sont d'un int\'er\^et tout particulier. Nous montrons que s'il existe une droite qui intersecte une courbe unicuspidale $C \subset \p^2$ seulement en ses points singuliers, alors toute autre courbe dont le compl\'ementaire est isomorphe \`a $\p^2 \setminus C$ doit \^etre projectivement \'equivalente \`a $C$. Il s'agit d'une g\'en\'eralisation d'un r\'esultat de H. Yoshihara qui l'a d\'emontr\'e sur les nombres complexes. De plus, nous \'etudions des propri\'et\'es des suites de multiplicit\'es des courbes irr\'eductibles qui impliquent que tout isomorphisme entre compl\'ementaires de ces courbes s'\'etend en un automorphisme de $\p^2$. Faisant usage de ces r\'esultats, nous montrons que deux courbes irr\'eductibles de degr\'e $\leq7$ ont des compl\'ementaires isomorphes si et seulement si elles sont projectivement \'equivalentes. Enfin, nous d\'ecrivons de nouveaux exemples de courbes irr\'eductibles de degr\'e $8$ non projectivement \'equivalentes qui ont des compl\'ementaires isomorphes.}

\end{prelims}


\newpage

\setcounter{tocdepth}{2} \tableofcontents

\section{Introduction}

Throughout this article, we fix an algebraically closed field $\k$ of arbitrary characteristic. Curves in $\p^2$ will always be assumed to be closed. Let $C, D \subset \p^2$ be two irreducible curves. We then call $C$ and $D$ \emph{projectively equivalent} if there exists an automorphism of $\p^2$ that sends $C$ to $D$. Our aim is to study isomorphisms $\p^2 \setminus C \to \p^2 \setminus D$ and properties of the curves $C$ and $D$, given such an isomorphism. In 1984, H.~Yoshihara stated the following conjecture.

\begin{conjecture}[\cite{Yos84}]\label{conjecture:yoshihara}
Let $C,D \subset \p^2$ be irreducible curves and $\varphi \colon \p^2 \setminus C \to \p^2 \setminus D$ an isomorphism between their complements. Then $C$ and $D$ are projectively equivalent.
\end{conjecture}

A counterexample to Conjecture~$\ref{conjecture:yoshihara}$ was given in \cite{Bla09}. The construction given there yields non-isomorphic (and hence projectively non-equivalent) rational curves $C_0$ and $D_0$ of degree $39$ that have isomorphic complements. Both curves have a unique singular point $p_0 \in C_0$ and $q_0 \in D_0$ respectively, such that $C_0 \setminus \{p_0\}$ and $D_0 \setminus \{q_0\}$ are isomorphic to open subsets of $\p^1$, each with $9$ complement points. To see that $C_0$ and $D_0$ are not isomorphic, it is shown that the two sets of $9$ complement points, corresponding to $C_0$ and $D_0$, are non-equivalent by the action of $\PGL_2 = \Aut(\p^1)$ on $\p^1$.

It is a general fact that if there exists an isomorphism $\varphi \colon \p^2 \setminus C \to \p^2 \setminus D$ that does not extend to an automorphism of $\p^2$, then $C$ and $D$ are of the same degree (Lemma~$\ref{Lem:equaldegree}$) and there exist points $p \in C$ and $q \in D$ such that each $C \setminus \{p\}$ and $D \setminus \{q\}$ are isomorphic to complements of $k \geq 1$ points in $\p^1$ (Proposition~$\ref{Prop:form}$). Moreover, when the number $k$ of complement points is $\geq 3$, the isomorphism $\varphi$ is uniquely determined, up to a left-composition with an automorphism of $\p^2$ (Proposition~$\ref{Prop:rigidity}$).

The case of unicuspidal rational curves (i.e.\ when the number $k$ of complement points is $1$) is of particular interest since the rigidity of Proposition~$\ref{Prop:rigidity}$ does not hold there. Indeed, by a result of P. Costa (\cite{Cos12}, \cite[Proposition~A.$3$.]{BFH16}), there exists a family of irreducible rational unicuspidal curves $(C_\lambda)_{\lambda \in \k^*}$ in $\p^2$ that are pairwise projectively non-equivalent, but all have isomorphic complements. The first main result of this article shows that a unicuspidal curve $C$ cannot be part of such a family if there exists a line $L$ that intersects $C$ only in its singular point.

\begin{theorem}\label{Thm:Yoshihara}
Let $C \subset \p^2$ be an irreducible curve and $L \subset \p^2$ a line such that $C \setminus L \simeq \A^1$. Let $\varphi \colon \p^2 \setminus C \to \p^2 \setminus D$ be an isomorphism, where $D \subset \p^2$ is some curve. Then $C$ and $D$ are projectively equivalent.
\end{theorem}

This theorem was already proven by H. Yoshihara~\cite{Yos84} over the field of complex numbers. His proof relies on the theorem of Abhyankar-Moh-Suzuki (\cite{AM75}, \cite{Suz74}) and also uses some analytic tools. We give a purely algebraic proof that works over arbitrary algebraically closed fields.

The counterexamples to Conjecture~$\ref{conjecture:yoshihara}$ given by P. Costa are of degree $9$ and it is thus natural to ask what happens in lower degrees. This is the second main result of this article. For the definition of multiplicity sequence used below, see Definition~$\ref{Def:multseq}$.

\begin{theorem}\label{Thm:degree8}
Let $C, D \subset \p^2$ be irreducible curves of degree $\leq 8$ and $\varphi \colon \p^2 \setminus C \to \p^2 \setminus D$ an isomorphism that does not extend to an automorphism of $\p^2$. Then $C$ and $D$ both are either:

\begin{enumerate}
\item[$(i)$] lines;
\item[$(ii)$] conics;
\item[$(iii)$] nodal cubics;
\item[$(iv)$] projectively equivalent rational unicuspidal curves;
\item[$(v)$] projectively equivalent curves of degree $6$ with multiplicity sequence $(3,2_{(7)})$;
\item[$(vi)$] curves of degree $8$ with multiplicity sequence $(3_{(7)})$ such that $$C \setminus \Sing(C) \simeq D \setminus \Sing(D) \simeq \A^1 \setminus \{0\}.$$
\end{enumerate}
\end{theorem}

In the proof, we study the diagrams of exceptional curves in the resolutions of the birational transformations of $\p^2$ that are induced by the isomorphisms between the complements, for all types of multiplicity sequences that can occur. We also use Theorem~$\ref{Thm:Yoshihara}$ as an important tool.

As an immediate consequence of Theorem~$\ref{Thm:degree8}$, we get the following corollary.
\begin{corollary}\label{cor:deg7}
Conjecture~$\ref{conjecture:yoshihara}$ holds for all irreducible curves of degree $\leq 7$.
\end{corollary}

Finally, we show that Corollary~$\ref{cor:deg7}$ is sharp by giving a counterexample of degree~$8$. The construction is based on a configuration of conics and is given in Section~$\ref{sec:deg8}$.

\begin{theorem}\label{theorem:counterexample}
There exist irreducible projectively non-equivalent curves $C, D \subset \p^2$ of degree $8$ with multiplicity sequence $(3_{(7)})$ that have isomorphic complements. 
\end{theorem}

\subsection{Acknowledgement}
This article is the second chapter of the author's PhD thesis. I am deeply grateful to my advisor J\'er\'emy Blanc for his excellent support and guidance throughout my PhD. I also thank Adrien Dubouloz and Jean-Philippe Furter for many helpful comments and interesting discussions.

\section{Preliminaries}

The following lemma is a well known fact, but included for the sake of completeness.

\begin{lemma}\label{Lem:equaldegree}
Let $C,D \subset \p^2$ be irreducible curves such that there exists $\varphi \colon \p^2 \setminus C \to \p^2 \setminus D$ an isomorphism. Then $\deg(C) = \deg(D)$.
\end{lemma}

\begin{proof}
Consider the following exact sequence of groups
$$0 \to \Z \xrightarrow{\alpha} \Pic(\p^2) \xrightarrow{\beta} \Pic(\p^2 \setminus C) \to 0$$
where $\alpha$ sends $1$ to the class of $C$ in $\Pic(\p^2)$ and $\beta$ is induced by the map that sends a curve $E \subset \p^2$ to the restriction $E \cap (\p^2 \setminus C)$. The exactness at $\Pic(\p^2)$ follows from the irreducibility of $C$. Since the class $[C]$ equals $\deg(C)[L]$, where $L$ is a line in $\p^2$, we obtain that $\Pic(\p^2 \setminus C) \simeq \Z / \deg(C)\Z$. The isomorphism $\varphi \colon \p^2 \setminus C \to \p^2 \setminus D$ induces an isomorphism on the corresponding Picard groups and hence the claim follows.
\end{proof}

\begin{remark}
{\rm
The claim of Lemma~$\ref{Lem:equaldegree}$ is false for reducible curves. As an example, consider the curves given by the equations $yz=0$ and $(x^2-yz)z=0$. They have isomorphic complements via the automorphism of $\p^2 \setminus \{z=0\}$ that sends $[x:y:z]$ to $[xz:x^2-yz:z^2]$ (which is an involution). This example also shows that it is easy to construct reducible counterexamples to Conjecture~$\ref{conjecture:yoshihara}$.}
\end{remark}

\begin{definition}
{\rm
Let $m \in \mathbb{Z}$. A birational morphism $\pi \colon X \to \p^2$ is called a \emph{$m$-tower resolution} of a curve $C \subset \p^2$ if
\begin{enumerate}
\item[$(i)$] there exists a decomposition
$$\pi \colon X = X_n \xrightarrow{\pi_n} \ldots \xrightarrow{\pi_2} X_1 \xrightarrow{\pi_1} X_0 = \p^2$$
where $\pi_{i}$ is the blow-up of a point $p_i$, for $i = 1,\ldots,n$, such that $\pi_{i}(p_{i+1}) = p_i$, for $i=1,\ldots,n-1$;
\item[$(ii)$] the strict transform of $C$ by $\pi$ in $X$ is isomorphic to $\p^1$ and has self-intersection~$m$.
\end{enumerate}}
\end{definition}

We use the following notational conventions throughout this article. Given a $m$-tower resolution of a curve $C \subset \p^2$ as above and $i \in \{1,\ldots,n\}$, we denote by $C_i$ the strict transform of $C$ by $\pi_{1} \circ \ldots \circ \pi_{i}$ in $X_i$. We usually denote by $E_i$ the exceptional curve of $\pi_i$, i.e.\ $\pi_i^{-1}(p_i) = E_i \subset X_i$. By abuse of notation, we also denote its strict transforms in $X_{i+1},\ldots,X_n$ by $E_i$.

\bigskip

We will frequently use the following fundamental lemma.

\begin{lemma}[\cite{Bla09}]\label{Lem:tower}
Let $C \subset \p^2$ be an irreducible curve and $\varphi \colon \p^2 \setminus C \to \p^2 \setminus D$ an isomorphism, for some curve $D \subset \p^2$. Then either $\varphi$ extends to an automorphism of $\p^2$ or the induced birational map $\varphi \colon \p^2 \dasharrow \p^2$ has a minimal resolution
$$\xymatrix{&X \ar[dr]^{\eta} \ar[dl]_{\pi} \\ \p^2 \ar@{-->}[rr]^{\varphi} && \p^2}$$
where $\pi$ and $\eta$ are $(-1)$-tower resolutions of $C$ and $D$ respectively.
\end{lemma}

Given a resolution as in Lemma~$\ref{Lem:tower}$, where $\pi$ has a decomposition $$\pi \colon X = X_n \xrightarrow{\pi_n} \ldots \xrightarrow{\pi_2} X_1 \xrightarrow{\pi_1} X_0 = \p^2$$ with base-points $p_1,\ldots,p_n$ and exceptional curves $E_1,\ldots,E_n$, we make the following observations that are used throughout this article.

\begin{enumerate}
\item[$(i)$] For any $i \in \{1,\ldots,n\}$, the curve $E_1\cup \ldots \cup E_i \subset X_i$ has simple normal crossings (SNC) and has a tree structure, i.e.\ for any two curves from $E_1,\ldots,E_i$ there exists a unique chain of curves from $E_1,\ldots,E_i$ connecting them.
\item[$(ii)$] For any $i \in \{1,\ldots,n\}$, the curves $E_1,\ldots,E_{i-1} \subset X_i$ have self-intersection $\leq -2$ and $E_i \subset X_i$ has self-intersection $-1$. 
\item[$(iii)$] The contracted locus of $\eta$ is $E_1 \cup \ldots E_{n-1} \cup C_n \subset X$ and is also a SNC-curve that has a tree structure. Moreover, $E_n$ is the strict transform of $D$ by $\eta$.
\end{enumerate}

\begin{remark}
{\rm We take the notations of Lemma~$\ref{Lem:tower}$ and suppose that $\varphi$ does not extend to an automorphism of $\p^2$. We then have a $(-1)$-tower resolution $\pi = \pi_1 \circ \ldots \circ \pi_n$ of $C$ with exceptional curves $E_1,\ldots,E_n$ and a $(-1)$-tower resolution $\eta = \eta_1 \circ \ldots \circ \eta_n$ of $D$ with exceptional curves $F_1,\ldots,F_n$. We then have the equality $\{E_1,\ldots,E_{n-1}\} = \{F_1,\ldots,F_{n-1}\}$ and $E_n$ is the strict transform of $D$ by $\eta$ and $F_n$ is the strict transform of $C$ by $\pi$. One may ask if such a resolution is always symmetric in the sense that 
$$E_i \cdot E_j = F_i \cdot F_j \quad \mbox{and} \quad E_i \cdot F_n = F_i \cdot E_n$$
for all $i,j = 1,\ldots,n$. This is in general not the case. For instance, there exists a non-symmetric resolution of an automorphism of the complement of a line with the following configuration of curves, where the unlabeled curves are $(-2)$-curves.
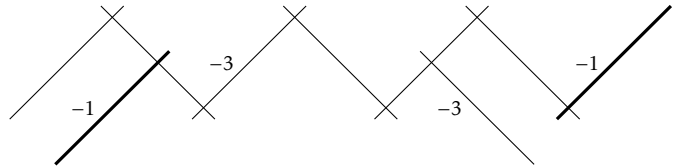
\begin{center}
$$
\begin{tikzpicture}[scale=1.5]
\draw (0,0)--(1,1);
\draw (0.8,1)--(1.8,0);
\draw[very thick] (1.4,0.6)--(0.4,-0.4);
\draw (1.6,0)--(2.6,1);
\draw (2.4,1)--(3.4,0);
\draw (3.2,0)--(4.2,1);
\draw (3.6,0.6)--(4.6,-0.4);
\draw (4,1)--(5,0);
\draw[very thick] (4.8,0)--(5.8,1);

\node at (1.86,0.5){\scriptsize $-3$};
\node at (3.86,0.1){\scriptsize $-3$};
\node at (0.65,0.1){\scriptsize $-1$};
\node at (5.07,0.5){\scriptsize $-1$};
\end{tikzpicture}
$$
\end{center}

\noindent Starting with either of the $(-1)$-curves in this configuration, one can successively contract all curves except the other $(-1)$-curve, whose image is a line in $\p^2$.
	
Similarly, one can find non-symmetric resolutions of automorphisms of the complement of a conic. However, no example of a non-symmetric resolution of an isomorphism between complements of irreducible singular curves is known to the author.}
\end{remark}

\begin{proposition}\label{Prop:form}
Let $\varphi \colon \p^2 \setminus C \hookrightarrow \p^2$ be an open embedding, where $C$ is an irreducible curve and let us consider $D = \p^2 \setminus \im(\varphi)$. If $\varphi$ does not extend to an automorphism of $\p^2$, then one of the following holds.
\begin{enumerate}
\item[$(i)$] $C$ and $D$ both are lines.
\item[$(ii)$] $C$ and $D$ both are conics.
\item[$(iii)$] $C$ and $D$ each have a unique proper singular point $p$ and $q$ respectively, such that $C \setminus \{p\}$ and $D \setminus \{q\}$ each are isomorphic to open subsets of $\p^1$, with the same number of complement points. 
\end{enumerate} 
\end{proposition}

\begin{proof}
By Lemma~$\ref{Lem:tower}$ the birational map $\varphi$ has a minimal resolution
$$
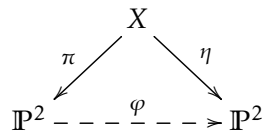
\xymatrix{&X \ar[dr]^{\eta} \ar[dl]_{\pi} \\ \p^2 \ar@{-->}[rr]^{\varphi} && \p^2}$$
where $\pi$ and $\eta$ are $(-1)$-tower resolutions of $C$ and $D$ respectively. Since $C$ and $D$ have the same degree the cases $(i)$ and $(ii)$ are clear and we assume that $C$ (and thus also $D$) has degree $\geq 3$. The curves $C$ and $D$ are both rational since they have a $(-1)$-tower resolution and hence they have a singular point $p$ and $q$ respectively, by the genus-degree formula for plane curves. Denote by $\hat{C}$ the strict transform of $C$ by $\pi$, by $\hat{D}$ the strict transform of $D$ by $\eta$, and by $E$ be the union of irreducible curves in $X$ contracted by both $\pi$ and $\eta$. Then $\hat{C}\cup E$ is the exceptional locus of $\eta$ and its irreducible components form a tree, since $\eta$ is a $(-1)$-tower resolution. Likewise, $\hat{D}\cup E$ is the exceptional locus of $\pi$ and is a tree of irreducible curves. We thus have isomorphisms $C \setminus \{p\} \simeq \hat{C} \setminus (E \cup \hat{D})$ and $D \setminus \{q\} \simeq \hat{D} \setminus (E \cup \hat{C})$ induced by $\pi$ and $\eta$ respectively. Since $\hat{C}$ and $\hat{D}$ are both isomorphic to $\p^1$ and they both intersect $E$ transversely it follows that $C \setminus \{p\}$ and $D \setminus \{q\}$ are isomorphic to open subsets of $\p^1$. The number of intersection points between $\hat{C}$ and $E \cup \hat{D}$ is given by
$$\#(\hat{C}\cap E) + \#(\hat{C} \cap \hat{D}) - \#(\hat{C}\cap E\cap \hat{D}).$$
For $\hat{D}$ the same formula holds with $\hat{C}$ and $\hat{D}$ exchanged. It thus suffices to show that $\#(\hat{C}\cap E) = \#(\hat{D}\cap E)$. Since the graphs of curves of $\hat{C} \cup E$ and $\hat{D} \cup E$ define a tree, it follows that $\#(\hat{C} \cap E)$ and $\#(\hat{D} \cap E)$ respectively is the number of connected components of $E$.
\end{proof}

As a direct consequence, we get the following observation, which we can already find in \cite{Yos84} and \cite{Bla09}.

\begin{corollary}
Let $C,D \subset \p^2$ be irreducible closed curves and $\varphi \colon \p^2 \setminus C \to \p^2 \setminus D$ an isomorphism. If $C$ is not rational or has more than one proper singular point, then $\varphi$ extends to an automorphism of $\p^2$.
\end{corollary}

\begin{proposition}\label{Prop:rigidity}
Let $C \subset \p^2$ be an irreducible curve and $\varphi \colon \p^2 \setminus C \hookrightarrow \p^2$ an open embedding that does not extend to an automorphism of $\p^2$. Let $p \in C$ be a point such that $C \setminus \{p\}$ is isomorphic to $\p^1 \setminus \{p_1,\ldots,p_k\}$, where $p_1,\ldots,p_k \in \p^1$ are distinct points. If $k \geq 3$, then $\varphi$ is uniquely determined up to a left-composition with an automorphism of $\p^2$.
\end{proposition}

\begin{proof}
By Lemma~$\ref{Lem:tower}$ there exists a $(-1)$-tower resolution $\pi \colon X = X_n \xrightarrow{\pi_n} \ldots \xrightarrow{\pi_2} X_1 \xrightarrow{\pi_1} \p^2$ with exceptional curves $E_1,\ldots,E_n$ and a $(-1)$-tower resolution $\eta \colon X \to \p^2$ of some curve $D \subset \p^2$ such that $\varphi \circ \pi = \eta$. We denote by $E = E_1 \cup \ldots \cup E_{n-1}$ the union of irreducible curves in $X$ that are contracted by both $\pi$ and $\eta$. Moreover, we denote by $\hat{C} = C_n$ the strict transform of $C$ by $\pi$ in $X$, and by $\hat{D} = E_n$ the strict transform of $D$ by $\eta$ in $X$. Since $\pi$ and $\eta$ are $(-1)$-tower resolutions, we know that $E \cup \hat{C}$ and $E \cup \hat{D}$ have a tree structure such that $\hat{C}$ and $\hat{D}$ each intersect $E$ in $1$ or $2$ points. It also follows that $k = \#\hat{C}\cap (E \cup \hat{D})$.

Let us assume first that $k \geq 4$. Then it follows that $\hat{C}$ and $\hat{D}$ intersect in at least two points. This implies that the image of $\hat{C}$ after contracting the $(-1)$-curve $\hat{D}$ is singular. Hence $\pi$ is the minimal resolution of singularities of $C$, i.e.\ the blow-up of all the singular points of $C$. By the same argument $\eta$ is the minimal resolution of singularities of $D$. Thus the base-points of $\pi$ and $\eta$ are completely determined by $C$ and $D$ respectively. But this means that for any other birational map $\psi \colon \p^2 \dashrightarrow \p^2$ that restricts to an isomorphism $\p^2 \setminus C \to \p^2 \setminus D$ the composition $\psi \circ \varphi^{-1}$ is an automorphism of $\p^2$. Thus the claim follows in this case.

We now assume that $k = 3$. Then $\hat{C}$ and $\hat{D}$ intersect in $1,2$, or $3$ points. Assume first that $\hat{C}$ and $\hat{D}$ intersect in $2$ or $3$ points. Then the image of $\hat{C}$ after contracting $\hat{D}$ is singular, so $\pi$ is the minimal resolution of singularities of $C$, and analogously $\eta$ is the minimal resolution of singularities of $D$. Then for the same reason as before, any other isomorphism $\p^2 \setminus C \to \p^2 \setminus D$ is just $\varphi$ composed with an automorphism of $\p^2$.

Finally, we assume that $k=3$ and that $\hat{C}$ and $\hat{D}$ intersect in only one point. We can assume that this intersection is transversal, otherwise, if they were tangent, $\pi$ and $\eta$ would again be the minimal resolutions of the singularities of $C$ and $D$ respectively and we could argue as before. The curve $\hat{D}$ intersects $E$ in two distinct components, say $E_i$ and $E_j$. If we contract the $(-1)$-curve $\hat{D}$, there is a triple intersection between the images of $\hat{C}$, $E_i$ and $E_j$. But this means that $\pi$ is the minimal resolution of $C$ such that the pull-back $\pi^*(C)$ is a SNC-divisor on $X$. Hence the base-points of $\pi$ are again completely determined by the curve $C$. Likewise, the base-points of $\eta$ are determined by $D$. We then argue as before that any isomorphism $\p^2 \setminus C \to \p^2 \setminus D$ is the composition of $\varphi$ with an automorphism of $\p^2$.
\end{proof}

\begin{corollary}\label{cor:uniqueD}
Let $C \subset \p^2$ be an irreducible curve such that there exists no point $p \in C$ such that $C \setminus \{p\}$ is isomorphic to $\A^1$ or $\A^1 \setminus \{0\}$. Then there exists at most one curve $D \subset \p^2$, up to projective equivalence, such that $\p^2 \setminus C$ and $\p^2 \setminus D$ are isomorphic and such that $D$ is not projectively equivalent to $C$.
\end{corollary}

\begin{proof}
This is a direct consequence of Proposition~$\ref{Prop:rigidity}$.
\end{proof}

\begin{remark}
{\rm
P. Costa's example \cite{Cos12} shows that Corollary~$\ref{cor:uniqueD}$ does in general not hold when there exists a point $p$ such that $C \setminus \{p\} \simeq \A^1$. On the other hand, there is no known example of pairwise projectively non-equivalent curves $C,D,E \subset \p^2$ such that all $3$ curves have isomorphic complements and there exists a point $p \in C$ such that $C \setminus \{p\} \simeq \A^1 \setminus \{0\}$.}
\end{remark}

\section{Unicuspidal curves with a very tangent line}

\subsection{Very tangent lines}

Let $C \subset \p^2$ be an irreducible curve. A singular point $p \in C$ is called a \emph{cusp} if the preimage of $p$ under the normalization $\hat{C} \to C$ consists of only one point. A curve is called \emph{unicuspidal} if it has one cusp and is smooth at all other points. We call a line $L\subset \p^2$ \emph{very tangent} to $C$ if there exists a point $q$ such that $(C\cdot L)_q = \deg(C)$. By B\'ezout's theorem this means that $L$ intersects $C$ in only one point. A line that is very tangent to $C$ is also tangent in the usual sense, except in the special case where $C$ is a line and the intersection is transversal. 

\begin{lemma}\label{Lem:unicuspidal}
Let $C \subset \p^2$ be an irreducible curve and $L \subset \p^2$ a line. Then $C \setminus L \simeq \A^1$ if and only if $L$ is very tangent to $C$ and one of the following holds:
\begin{enumerate}
\item[$(i)$] $C$ is a line.
\item[$(ii)$] $C$ is a conic.
\item[$(iii)$] $C$ is rational and unicuspidal and $L$ passes through the singular point of $C$.
\end{enumerate}
\end{lemma}

\begin{proof}
Assume that $L$ is very tangent to $C$. If $C$ is a line or a conic, then $C$ is isomorphic to $\p^1$ and thus $C \setminus L \simeq \A^1$. We thus assume that $C$ is rational and unicuspidal with singular point $p$, where $L$ passes through $p$. It follows that $C$ has a normalization $\eta \colon \p^1 \to C$ such that $\eta^{-1}(p)$ consists of only one point and thus $C \setminus \{p\} \simeq \p^1 \setminus \eta^{-1}(p) \simeq \A^1$. Since $L$ is very tangent to $C$, the intersection $C \cap L$ consists only of the point $p$. It follows that $C \setminus L \simeq C \setminus \{p\} \simeq \A^1$. 

To prove the converse, assume that $C \setminus L \simeq \A^1$. It follows that $C$ is rational and $\Sing(C) \subset C \cap L$. We consider the normalization $\eta \colon \p^1 \to C$ and obtain $ C \setminus L \subset C \setminus \Sing(C) \simeq \p^1 \setminus \eta^{-1}(\Sing(C))$. Since $C \setminus L \simeq \A^1$, it follows that $\eta^{-1}(\Sing(C))$ consists of at most one point. If $\eta^{-1}(\Sing(C))$ is empty, then $C \simeq \p^1$ is smooth and thus either a line or a conic, by the genus-degree formula. Since $C \setminus L \simeq \A^1$, it follows that $L$ intersects $C$ in only one point and is thus very tangent to $C$. If $\eta^{-1}(\Sing(C))$ is not empty, then it contains exactly one point and thus $C$ is unicuspidal and $C \setminus L = C \setminus \Sing(C)$. Since $C \cap L = \Sing(C)$ consists of only one point, the line $L$ is very tangent to $C$.
\end{proof}

If $C$ is unicuspidal and rational and has a very tangent line $L$ through the singular point, then $C\setminus L \simeq \A^1$. In other words, $C$ is equivalent to the closure of the image of a closed embedding $\A^1 \hookrightarrow \A^2 \simeq \p^2 \setminus L$. Note that not all rational unicuspidal curves admit a very tangent line through the singular point. For instance, there exists such a unicuspidal quintic curve that is studied in detail in Section~$\ref{subsec:quintic}$.

We call $C \setminus L \subset \p^2 \setminus L \simeq \A^2$ \emph{rectifiable} if there exists an automorphism $\theta \in \Aut(\p^2 \setminus L)$ such that $\theta(C) = L' \setminus L$ for some line $L' \subset \p^2$ that is distinct from $L$. Suppose that there exists an open embedding $\varphi \colon \p^2 \setminus C \hookrightarrow \p^2$ that does not extend to an automorphism of $\p^2$, then the induced birational map $\p^2 \dashrightarrow \p^2$ contracts the curve $C$ to a point. It turns out that $C \setminus L \subset \p^2 \setminus L$ is then rectifiable. This is a consequence of the following proposition, proven in \cite[Proposition~$3.16$]{BFH16}. It also follows from the work of \cite{KM83} and \cite{Gan85} (see \cite[Remark~$2.30$]{BFH16}).

\begin{proposition}\label{Prop:rectifiabilty}
Let $C \subset \A^2 = \p^2 \setminus L_\infty$ be a closed curve, isomorphic to $\A^1$, and denote by $\overline{C}$ the closure of $C$ in $\p^2$. Then the following are equivalent:
\begin{enumerate}
\item[$(i)$] There exists an automorphism of $\A^2$ that sends $C$ to a line.
\item[$(ii)$] There exists a birational transformation of $\p^2$ that sends $\overline{C}$ to a point.
\end{enumerate}
\end{proposition}

We call a curve satisfying condition $(ii)$ of Proposition~$\ref{Prop:rectifiabilty}$ \emph{Cremona-contractible}. Note that condition $(i)$ is always satisfied if the characteristic of $\k$ is $0$ by the Abhyankar-Moh-Suzuki theorem (\cite{AM75}, \cite{Suz74}), but in general not in positive characteristic. It follows from Proposition~$\ref{Prop:rectifiabilty}$ that Theorem~$\ref{Thm:Yoshihara}$ holds if $C \setminus L \subset \p^2 \setminus L$ is not rectifiable.

\subsection{Automorphisms of $\A^2$ and de Jonqui\`eres maps}

\begin{definition}
{\rm
Let $L \subset \p^2$ be a line and $p \in L$. We denote by $\Jon(\p^2,L,p)$ the group of automorphisms of $\p^2 \setminus L$ that preserve the pencil of lines through $p$. We call an element in $\Jon(\p^2,L,p)$ a \emph{de Jonqui\`eres map with respect to $L$ and $p$}. }
\end{definition}

We recall the following standard terminology, for instance as used in \cite{Alb02}.
\begin{definition}
{\rm Let $X$ be a surface and let $p \in X$ be a point. Let $E$ be the exceptional curve of the blow-up of $p$. We then say that a point $q \in E$ lies in the \emph{first neighborhood of $p$}. For $k > 1$, we say that a point lies in the \emph{$k$-th neighborhood of $p$} if it lies in the first neighborhood of some point in the $(k-1)$-th neighborhood of $p$. We say that a point is \emph{infinitely near to $p$} if it lies in the $k$-th neighborhood of $p$, for some $k\geq 1$. We call a point $q$ \emph{proximate to $p$} (denoted $q \succ p$) if $q$ lies on the strict transform of the exceptional curve of the blow-up of $p$. We sometimes call the points of $X$ \emph{proper} to distinguish them from infinitely near points.}
\end{definition}

Throughout this section, we fix a line $L \subset \p^2$ and a point $p \in L$. Moreover, we fix projective coordinates $[x:y:z]$ on $\p^2$ and denote the lines
$$L_x \colon x = 0 \qquad \qquad L_y \colon y = 0 \qquad \qquad L_z \colon z = 0.$$

\begin{lemma}\label{Lem:decomposition}
Let $j \in \Jon(\p^2,L,p) \setminus \Aut(\p^2)$ be of degree $d$. Then the minimal resolution of $j$ has $2d-1$ base-points with exceptional curves $E_1,\ldots,E_{2d-1}$ as in the following configuration
$$
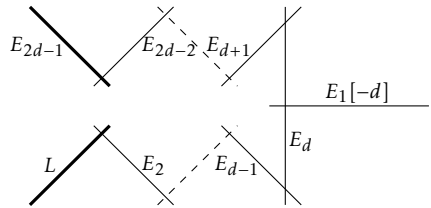
\begin{tikzpicture}[scale=1.05]
\draw[very thick] (0,2.5)--(1,1.5);
\draw (0.1,2) node{\scriptsize $E_{2d-1}$};
\draw (0.8,1.5)--(1.8,2.5);
\draw (1.75,2) node{\scriptsize $E_{2d-2}$};
\draw[dashed] (1.6,2.5)--(2.6,1.5);
\draw (2.4,1.5)--(3.4,2.5);
\draw (2.5,2) node{\scriptsize $E_{d+1}$};

\draw[very thick] (0,0)--(1,1);
\draw (0.25,0.5) node{\scriptsize $L$};
\draw (0.8,1)--(1.8,0);
\draw (1.55,0.5) node{\scriptsize $E_2$};
\draw[dashed] (1.6,0)--(2.6,1);
\draw (2.4,1)--(3.4,0);
\draw (2.6,0.5) node{\scriptsize $E_{d-1}$};

\draw (3.2,2.5)--(3.2,0);
\draw (3.4,0.8) node{\scriptsize $E_d$};
\draw (3,1.25)--(5,1.25);
\draw (4.1,1.42) node{\scriptsize $E_1[-d]$};
\end{tikzpicture}
$$
where the self-intersection numbers are $-1$ for thick lines, $-2$ for thin lines, or otherwise are indicated in square brackets.
\end{lemma}

\begin{proof}
The map $j$ is an automorphism of $\p^2 \setminus L$ that does not extend to an automorphism of $\p^2$, thus by Lemma~$\ref{Lem:tower}$ there exists a $(-1)$-tower resolution $\pi \colon X = X_n \xrightarrow{\pi_n} \ldots \xrightarrow{\pi_2} X_1 \xrightarrow{\pi_1} X_0 = \p^2$ of $L$ with exceptional curves $E_1,\ldots,E_n$ and a $(-1)$-tower resolution $\eta \colon X \to \p^2$ of $L$ such that $j \circ \pi = \eta$.  The unique proper base-point of $j$ is $p$, which is thus the base-point of the first blow-up with exceptional curve $E_1$. Since $\pi$ is a $(-1)$-tower resolution of $L$, the next base-point is the intersection point between $E_1$ and the strict transform of $L$. After this blow-up, the strict transform of $L$ has self-intersection $-1$ and thus there is no more base-point on this curve. We observe that $E_1$ is the last curve contracted by $\eta$, since $j$ preserves the pencil of lines through $p$. The next base-point is thus either the intersection point $q$ between $E_1$ and $E_2$ or a point on $E_2 \setminus (E_1 \cup L)$. Let $m \geq 0$ be the number of base-points proximate to $q$. After blowing up these $m$ points we have the following resolution.
$$
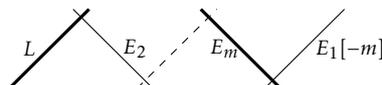
\begin{tikzpicture}[scale=1.05]
\draw[very thick] (0,1.5)--(1,2.5);
\draw (0.25,2) node{\scriptsize $L$};
\draw (0.8,2.5)--(1.8,1.5);
\draw[dashed] (1.6,1.5)--(2.6,2.5);
\draw[very thick] (2.4,2.5)--(3.4,1.5);
\draw (1.58,2) node{\scriptsize $E_2$};
\draw (2.7,2) node{\scriptsize $E_{m}$};
\draw (3.2,1.5)--(4.2,2.5);
\draw (4.27,2) node{\scriptsize $E_1[-m]$};
\end{tikzpicture}
$$
The next base-point then lies on $E_{m} \setminus E_1$. It cannot be the intersection point with $E_{m-1}$, because then $E_{m-1}$ would have self-intersection $<-2$ in $X$. But $\eta$ first contracts $L$ and then the curves $E_2,\ldots,E_{m-2}$. After these contractions the self-intersection of the image of $E_{m-1}$ must be $-1$. Hence the next base-point lies on $E_{m} \setminus (E_1 \cup E_{m-1})$. We observe moreover that after $\eta$ contracts $L, E_2, \ldots, E_{m}$ the image of $E_1$ has self-intersection $-m+1$. Thus there is a chain of $(-2)$-curves of length $m-1$ attached to $E_m$, which are obtained by successively blowing up points that lie on the last exceptional curve but not on the intersection with another one. Since $E_1$ is the last curve contracted by $\eta$, it follows that $E_{2m-1}$ is the last exceptional curve of $\pi$.

Let us now determine the degree of $j$. For this we look at the degree of the image of a line $L'$ that does not pass through the base-points of $j$. The strict transform of $L'$ is drawn in the diagram on the left below.
$$
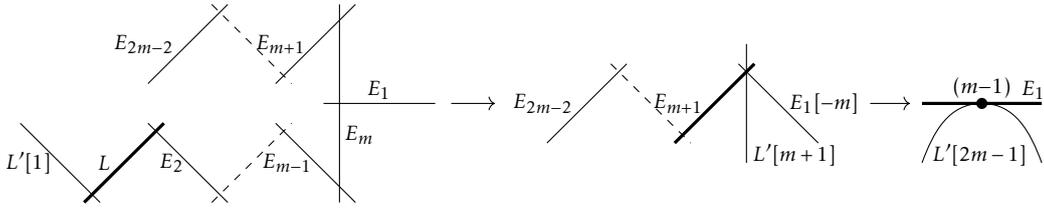
\begin{tikzpicture}[scale=1.05]
\draw (0.8,1.5)--(1.8,2.5);
\draw (0.75,2) node{\scriptsize $E_{2m-2}$};
\draw[dashed] (1.6,2.5)--(2.6,1.5);
\draw (2.4,1.5)--(3.4,2.5);
\draw (2.47,2) node{\scriptsize $E_{m+1}$};

\draw[very thick] (0,0)--(1,1);
\draw (0.25,0.5) node{\scriptsize $L$};
\draw (0.8,1)--(1.8,0);
\draw (1.1,0.5) node{\scriptsize $E_2$};
\draw[dashed] (1.6,0)--(2.6,1);
\draw (2.4,1)--(3.4,0);
\draw (2.55,0.5) node{\scriptsize $E_{m-1}$};

\draw (3.2,2.5)--(3.2,0);
\draw (3.45,0.85) node{\scriptsize $E_m$};
\draw (3,1.25)--(4.4,1.25);
\draw (3.7,1.42) node{\scriptsize $E_1$};

\draw (0.2,0)--(-0.8,1);
\draw (-0.7,0.5) node{\scriptsize $L'[1]$};

\draw[->] (4.6,1.25)--(5.15,1.25);

\draw (5.8,0.75)--(6.8,1.75);
\draw (5.75,1.25) node{\scriptsize $E_{2m-2}$};
\draw[dashed] (6.6,1.75)--(7.6,0.75);
\draw[very thick] (7.4,0.75)--(8.4,1.75);
\draw (7.45,1.25) node{\scriptsize $E_{m+1}$};
\draw (8.2,1.75)--(9.2,0.75);
\draw (9.27,1.25) node{\scriptsize $E_1[-m]$};
\draw (8.9,0.6) node{\scriptsize $L'[m+1]$};

\draw (8.3,0.5)--(8.3,2);

\draw[->] (9.85,1.25)--(10.35,1.25);

\draw[very thick] (10.5,1.25)--(12,1.25);
\draw (11.9,1.42) node{\scriptsize $E_1$};
\draw (11.25,1.45) node{$\scriptstyle (m-1)$};
\fill (11.25,1.25) circle (2pt);
\draw (11.22,0.6) node{\scriptsize $L'[2m-1]$};

\draw (10.5,0.5) to [out=70,in=180] (11.25,1.25);
\draw (11.25,1.25) to [out=0,in=110] (12,0.5);

\end{tikzpicture}
$$
After the curves $L,E_2,\ldots,E_m$ are contracted the image of $L'$ has self-intersection $m+1$ and $L'$ intersects $E_{m+1}$ and $E_1$, as shown in the diagram in the middle. Next, the curves $E_{m+1},\ldots,E_{2m-2}$ are contracted and the image of $L$ has self-intersection $2m-1$ and $L$ intersects $E_1$ with multiplicity $(m-1)$. Thus after $E_1$ is contracted the self-intersection of the image of $L$ is $2m-1 + (m-1)^2 = m^2$ and hence the degree $d$ of $j$ is equal to $m$.
\end{proof}

We often identify $\p^2 \setminus L_z$ with the affine plane $\A^2$ with coordinates $x,y$, via the open embedding $(x,y) \mapsto [x:y:1]$. We call $j \in \Aut(\A^2)$ an \emph{affine de Jonqui\`eres map} if it is the restriction of a de Jonqui\`eres map with respect to $L_z$ and $[0:1:0]$. Affine de Jonqui\`eres maps then preserve the fibration $(x,y) \mapsto x$.

\begin{lemma}
Let $j \in \Aut(\A^2)$ be an affine de Jonqui\`eres map. Then $j$ is of the form
$$(x,y) \mapsto \left(ax+b,cy+f(x)\right)$$
where $a,c \in \k^*$, $b \in \k$, and $f \in \k[x]$.
\end{lemma}

\begin{proof}
The map $j$ sends $(x,y)$ to $(a(x,y),b(x,y))$, where $a,b \in \k[x,y]$. Since $j$ is an automorphism of $\A^2$, the polynomials $a$ and $b$ are irreducible. Moreover, $j$ preserves the fibration $(x,y) \mapsto x$, thus $a$ is a scalar multiple of some element $x-\lambda$ with $\lambda \in \k$. We can then apply an affine coordinate change and may assume that $a=x$. But then $j$ induces a $\k[x]$-automorphism of the polynomial ring $\k[x][y]$, and thus $b$ is of degree $1$ in the variable $y$. Moreover, the coefficient of $y$ is an element in $\k[x]^*= \k^*$ und thus the claim follows.
\end{proof}

We will use the well known structure theorem of Jung and van der Kulk in the sequel. We denote by $\Aff(\p^2,L)$ the \emph{affine group with respect to $L$}, which consists of the automorphisms of $\p^2$ that preserve $L$. Moreover, we denote by $B(\p^2,L,p)$ the intersection $\Aff(\p^2,L) \cap \Jon(\p^2,L,p)$.

\begin{thm}[\cite{Jun42}, \cite{vdK53}]\label{Thm:Jung}
The group $\Aut(\p^2 \setminus L)$ is generated by the subgroups $\Aff(\p^2,L)$ and $\Jon(\p^2,L,p)$. Moreover, $\Aut(\p^2 \setminus L)$ is a free product $$\Aff(\p^2,L) \ast_{B(\p^2,L,p)} \Jon(\p^2,L,p),$$ amalgamated over the intersection of these two subgroups.
\end{thm}

\begin{remark}
{\rm There exist many proofs of Theorem~$\ref{Thm:Jung}$. The proof in \cite{Lam02} uses blow-ups and contractions of the line $L_\infty = \p^2 \setminus \A^2$, in the spirit of the methods used in this article. For more proofs with a similar strategy see \cite{BD11} and \cite{BS15}.}
\end{remark}

\begin{lemma}\label{lemma:jonqprod}
Let $\theta \in \Aut(\p^2 \setminus L)$ with $$\theta = a \circ j_n \circ a_n \circ \ldots \circ j_1 \circ a_1,$$ where $a_1, a \in \left(\Aff(\p^2,L) \setminus \Jon(\p^2,L,p)\right) \cup \{\id\}$, $a_i \in \Aff(\p^2,L) \setminus \Jon(\p^2,L,p)$ for $i = 2,\ldots,n$ and where $j_i \in \Jon(\p^2,L,p) \setminus \Aff(\p^2,L)$ for $i = 1,\ldots,n$. Then $\theta$ has unique proper base-point $a_1^{-1}(p)$. Moreover, the degree of $\theta$ is $\prod_{i=1}^n \deg(j_i)$.
\end{lemma}

\begin{proof}
The map $j_1$ has unique proper base-point $p$, and thus $j_1 \circ a_1$ has unique proper base-point $a_1^{-1}(p)$ and $(j_1 \circ a_1)^{-1}$ has unique proper base-point $p$. We proceed by induction and assume that $j_{n-1} \circ a_{n-1} \circ \ldots \circ j_1 \circ a_1$ has unique proper base-point $a_1^{-1}(p)$ and its inverse has unique proper base-point $p$. Moreover, the unique proper base-point of $(j_n \circ a_n)$ is $a_n^{-1}(p)$, which is different from $p$ since $a_n \notin \Jon(\p^2,L,p)$. It then follows that the composition $j_n \circ a_n \circ \ldots \circ j_1 \circ a_1$ again has $a_1^{-1}(p)$ as its unique proper base-point. This remains true after a left-composition with $a \in \Aff(\p^2,L)$.

To compute the degree of $\theta$, we observe that $\deg(j_i \circ a_i) = \deg(j_i)$ for all $i$, since the maps $a_i$ are affine and hence have degree $1$. We use again that $(j_{n-1} \circ a_{n-1} \circ \ldots \circ j_1 \circ a_1)^{-1}$ and $j_n \circ a_n$ have no common base-point and obtain the result by induction by using \cite[Proposition~$4.2.1$]{Alb02}.
\end{proof}

\begin{definition}
{\rm
Let $X$ be a surface and let $C \subset X$ be a curve. For a point $p \in C$, let $\mathcal{O}_{X,p}$ be the local ring at $p$, with unique maximal ideal $\mathfrak{m}_p$. Let moreover $f \in \mathcal{O}_{X,p}$ be a local equation of $C$ at $p$. We then define the \emph{multiplicity $m_p(C)$ of $C$ at $p$} to be the largest integer $m$ such that $f \in \mathfrak{m}_p^m$.

Let $\Lambda$ be a linear system of curves on $\p^2$ and let $p$ be a proper or infinitely near point of $\p^2$. We then define the \emph{multiplicity of $\Lambda$ at $p$} to be the smallest multiplicity $m_p(C)$ among all curves $C$ in $\Lambda$.

For a birational map $\theta \colon \p^2 \dasharrow \p^2$, we denote by $\Lambda_\theta$ the linear system of curves on $\p^2$, given by the preimage of $\theta$ of the linear system of lines on $\p^2$. For a proper or infinitely near point $p$ of $\p^2$, we define the \emph{multiplicity $m_p(\theta)$ of $\theta$ at $p$} to be the multiplicity of the linear system $\Lambda_\theta$ at $p$.

For a more detailed account of these notions, we refer to \cite{Alb02}.}
\end{definition}

We will use the following well known formula in the sequel.

\begin{lemma}\label{Lem:formula}
Let $\theta \colon \p^2 \dasharrow \p^2$ be a birational map and $C \subset \p^2$ a curve that is not contracted by $\theta$. Then the following formula holds:
$$\deg \theta(C) = \deg(\theta)\deg(C) - \sum_p m_p(\theta) m_p(C)$$
where the sum ranges over all proper and infinitely near points of $\p^2$, but only finitely many summands are different from $0$.
\end{lemma}

\begin{proof}
We consider a minimal resolution
$$\xymatrix{&X \ar[dr]^{\sigma_2} \ar[dl]_{\sigma_1} \\ \p^2 \ar@{-->}[rr]^{\theta} && \p^2}$$
where $\sigma_1$ and $\sigma_2$ are compositions of blow-ups. We denote by $p_1,\ldots,p_n$ the base-points of $\sigma_1$ and by $\overline{E}_1,\ldots,\overline{E}_n$ the total transforms of their exceptional divisors in $X$. Let moreover $L \subset \p^2$ be a line that does not pass through the base-points of $\theta$ and $\theta^{-1}$. We then have
$$\Pic(X) \simeq \mathbb{Z}\sigma_1^*(L) \oplus \mathbb{Z}\overline{E}_1 \oplus \ldots \oplus \mathbb{Z}\overline{E}_n$$
with the intersection-numbers $\overline{E}_i \cdot \overline{E}_j = - \delta_{ij}$ and $\overline{E}_i \cdot \sigma_1^*(L) = 0$ for $i,j =1,\ldots,n$ and $\sigma_1^*(L)^2=1$. We find for the strict transform $\hat{C}$ of $C$ by $\sigma_1$ and the total transform of $L$ by $\sigma_2$ the following divisor formulas:
\begin{align*}
\hat{C} =& \deg(C)\sigma_1^*(L) - \sum_{i=1}^n m_{p_i}(C)\overline{E}_i, \\
\sigma_2^*(L) =& \deg(\theta)\sigma_1^*(L) - \sum_{i=1}^n m_{p_i}(\theta)\overline{E}_i.
\end{align*}
The degree of $\theta(C)$ is equal to the intersection number $\theta(C) \cdot L$. Using the projection formula, we then obtain
$$\deg(\theta(C)) = \theta(C)\cdot L = \hat{C} \cdot \sigma_2^*(L) = \deg(C)\deg(\theta) - \sum_{i=1}^n m_{p_i}(C)m_{p_i}(\theta).$$
\end{proof}

\begin{lemma}\label{Lem:degree}
Let $\theta \in \Aut(\p^2 \setminus L_x) \setminus \Aut(\p^2)$ and let $C \subset \p^2$ be a curve different from $L_x$. Then the following holds.
\begin{enumerate}
\item[$(i)$] $\theta$ has a unique proper base-point and contracts $L_x$ to a point $p \in L_x$.
\item[$(ii)$] $\deg(\theta(C)) \leq \deg(\theta)\deg(C)$, and equality holds if and only if $p \notin C$.
\item[$(iii)$] If $L$ is a line and $\theta \in \Jon(\p^2,L_x,[0:1:0])$, then $\theta^{-1}(L)$ is a line if and only if $[0:1:0] \in L$.
\end{enumerate}
\end{lemma}

\begin{proof}
To prove $(i)$, consider the induced birational map $\theta \colon \p^2 \dasharrow \p^2$. Since $\theta$ does not extend to an automorphism of $\p^2$, it follows from Lemma~$\ref{Lem:tower}$ that $\theta$ has a minimal resolution
$$\xymatrix{&X \ar[dr]^{\sigma_2} \ar[dl]_{\sigma_1} \\ \p^2 \ar@{-->}[rr]^{\theta} && \p^2}$$
where $\sigma_1$ and $\sigma_2$ are $(-1)$-tower resolutions of $L_x$. In particular, $\theta$ has a unique proper base-point. The strict transform of $L_x$ in $X$ by $\sigma_1$ is the exceptional curve of the last blow-up in the tower of $\sigma_2$. This means that $\theta$ contracts $L_x$ to a point of $L_x$, which is moreover the unique proper base-point of $\theta^{-1}$. The statements $(ii)$ and $(iii)$ follow directly from the formula $$\deg \theta(C) = \deg(\theta)\deg(C) - \sum_q m_q(\theta) m_q(C)$$
of Lemma~$\ref{Lem:formula}$, since $\theta$ has a unique proper base-point (which is $[0:1:0]$ if $\theta \in \Jon(\p^2,L_x,[0:1:0])$).
\end{proof}

\subsection{Isomorphisms between complements of unicuspidal curves}

\begin{lemma}\label{Lem:minimality}
Let $C \subset \p^2$ be a unicuspidal curve such that
$$\Theta = \{\theta \in \Aut(\p^2 \setminus L_x) \mid \theta(C) = L_z \}$$
is non-empty. Then for any $\theta \in \Theta$ and any minimal resolution
$$\xymatrix{&X \ar[dr]^{\sigma_2} \ar[dl]_{\sigma_1} \\ \p^2 \ar@{-->}[rr]^{\theta} && \p^2}$$
the following are equivalent.
\begin{enumerate}
\item[$(i)$] $\deg{\theta} \leq \deg{\theta'}$ for all $\theta' \in \Theta$.
\item[$(ii)$] The unique proper base-point of $\theta^{-1}$ is different from $[0:1:0]$.
\item[$(iii)$] $\deg(\theta) = \deg(C)$.
\item[$(iv)$] The strict transform of $C$ by $\sigma_1$ intersects the strict transform of $L_x$ by $\sigma_2$ in $X$.
\item[$(v)$] The strict transform of $C$ by $\sigma_1$ in $X$ has self-intersection $1$.
\end{enumerate}
\end{lemma}

\begin{proof}
Let $\theta \in \Theta$. We first prove $(i) \Rightarrow (ii)$ and thus assume that $\theta$ has minimal degree in $\Theta$. We use Theorem~$\ref{Thm:Jung}$ to write $$\theta^{-1} = a_{n+1} \circ j_n \circ a_n \circ \ldots \circ j_1 \circ a_1,$$ where $a_1, a_{n+1} \in (\Aff(\p^2,L_x) \setminus \Jon(\p^2,L_x,[0:1:0])) \cup \{\id\}$, $a_i \in \Aff(\p^2,L_x) \setminus \Jon(\p^2,L_x,[0:1:0])$ for $i=2,\ldots,n$, and $j_i \in \Jon(\p^2,L_x,[0:1:0]) \setminus \Aff(\p^2,L_x)$ for $i=1,\ldots,n$. If $(j_1 \circ a_1)(L_z)$ is a line, we can find $a_1' \in \Aff(\p^2,L_x)$ such that $a_1'(L_z) = (j_1 \circ a_1)(L_z)$. But then $\theta' \coloneqq (a_{n+1} \circ j_n \circ a_n \circ \ldots \circ j_2 \circ a_2 \circ a_1')^{-1}$ lies in $\Theta$ and $\deg(\theta') < \deg(\theta)$ by Lemma~$\ref{lemma:jonqprod}$, which contradicts the minimality of the degree of $\theta$ in $\Theta$. It follows moreover from Lemma~$\ref{Lem:degree}$ that $(j_1 \circ a_1)(L_z)$ is a line if and only if $[0:1:0] \in a_1(L_z)$, i.e. $a_1^{-1}([0:1:0]) \in L_z$. Thus by the minimality of the degree of $\theta$, we have that $a_1^{-1}([0:1:0]) \notin L_z$. Since $a_1^{-1}([0:1:0])$ is the unique proper base-point of $\theta^{-1}$, it follows that it is different from $[0:1:0]$ and hence $(ii)$ is proved.

Assume now that the unique proper base-point of $\theta^{-1}$ is different from $[0:1:0]$. From Lemma~$\ref{Lem:formula}$ we obtain the formula
$$\deg(\theta) = \deg(\theta^{-1}) = \deg(C) + \sum_p m_p(\theta^{-1})m_p(L_z).$$
Since the unique proper base-point of $\theta^{-1}$ lies on $L_x$ and is different from $[0:1:0]$, we then have $\deg(\theta) = \deg(C)$. This shows $(ii) \Rightarrow (iii)$. Moreover, if we assume that $\deg(\theta) = \deg(C)$, then $\theta$ has minimal degree in $\Theta$. Thus the implication $(iii) \Rightarrow (i)$ is also proved.

Finally, we show that $(iv)$ and $(v)$ are both equivalent to $(ii)$. We consider a minimal resolution of the induced birational map by $\theta$: 
$$\xymatrix{&X \ar[dr]^{\sigma_2} \ar[dl]_{\sigma_1} \\ \p^2 \ar@{-->}[rr]^{\theta} && \p^2.}$$
Since $\theta \in \Aut(\p^2 \setminus L_x) \setminus \Aut(\p^2)$ both $\sigma_1$ and $\sigma_2$ are $(-1)$-tower resolutions of $L_x$. We denote by $\hat{L}_{x}$ the strict transform of $L_x$ by $\sigma_2$ in $X$ and by $\hat{C}$ the strict transform of $C$ by $\sigma_1$ (which is also the strict transform $\hat{L}_z$ of $L_z$ by $\sigma_2$). Suppose that the unique proper base-point of $\theta^{-1}$ is different from $[0:1:0]$. Then $\hat{L}_x$ intersects $\hat{L}_z = \hat{C}$ and $\hat{C}$ has self-intersection $1$. This shows that $(ii)$ implies $(iv)$ and $(v)$. On the other hand, if we blow up the point $[0:1:0]$, then the strict transforms of $L_x$ and $L_z$ do not intersect and have self-intersection $<1$. Thus the implications $(iv) \Rightarrow (ii)$ and $(v) \Rightarrow (ii)$ also follow.
\end{proof}

\begin{proposition}\label{prop:diagram}
Let $\varphi \colon \p^2 \setminus C \to \p^2 \setminus D$ be an isomorphism, where $C, D \subset \p^2$ are curves such that $C$ is rational and unicuspidal with singular point $[0:1:0]$ and has very tangent line $L_x$. Let $\theta_C$ be an automorphism of $\p^2 \setminus L_x$ such that $\theta_C(C) = L_z$ and suppose that $\theta_C$ is of minimal degree with this property.

Then $D$ is also rational and unicuspidal and, after a suitable change of coordinates, has singular point $[0:1:0]$ and very tangent line $L_x$. Moreover, there exists an automorphism $\theta_D$ of $\p^2 \setminus L_x$ such that $\theta_D(D) = L_z$ and $\psi \in \Aut(\p^2 \setminus L_z)$ that preserves the line $L_x$ such that the following diagram commutes:
$$\xymatrix{\p^2 \ar@{-->}[d]_{\theta_C} \ar@{-->}[r]^{\varphi} & \p^2 \ar@{-->}[d]^{\theta_D} \\ \p^2 \ar@{-->}[r]^{\psi} & \p^2.}$$
Furthermore, $\theta_D$ can be chosen such that in the chart $z=1$, the map $\psi$ has the form $$(x,y) \mapsto \left(x,y+x^2f(x)\right)$$ for some polynomial $f \in \k[x]$.
\end{proposition}

\begin{proof}
The map $\theta_C$ induces a birational map $\p^2 \dasharrow \p^2$. It does not extend to an automorphism of $\p^2$ since $C$ is singular but its image by $\theta_C$ is a line. Thus $\theta_C$ contracts $L_x$ and no other curves. We consider a minimal resolution of $\theta_C$:
$$\xymatrix{&X \ar[dr]^{\sigma_2} \ar[dl]_{\sigma_1} \\ \p^2 \ar@{-->}[rr]^{\theta_C} && \p^2.}$$
By Lemma~$\ref{Lem:tower}$, the morphisms $\sigma_1$ and $\sigma_2$ are $(-1)$-tower resolutions of $L_x$. In particular, $\theta_C$ has a unique proper base-point. Since the image of $C$ is a line, the unique proper base-point of $\theta_C$ is the singular point $[0:1:0]$ and the strict transform of $C$ by $\sigma_1$ in $X$ is smooth. Hence $\sigma_1$ factors through the minimal SNC-resolution of $C$. Moreover, by the minimality of the degree of $\theta_C$, it follows from Lemma~$\ref{Lem:minimality}$ that the strict transform of $C$ by $\sigma_1$ intersects the strict transform of $L_x$ by $\sigma_2$ in $X$, i.e.\ the last exceptional curve of $\sigma_1$. It follows that the strict transform of $C$ by $\sigma_1$ in $X$ has self-intersection~$1$ by Lemma~$\ref{Lem:minimality}$. In fact, $\sigma_1$ is the minimal $1$-tower resolution of $C$ that factors through the SNC-resolution of $C$.

We now consider the induced birational map $\varphi \colon \p^2 \dasharrow \p^2$. We assume that $\varphi$ does not extend to an automorphism of $\p^2$, otherwise the proof is finished. Thus by Lemma~$\ref{Lem:tower}$ the map $\varphi$ has a minimal resolution
$$\xymatrix{&Y \ar[dr]^{\eta} \ar[dl]_{\pi} \\ \p^2 \ar@{-->}[rr]^{\varphi} && \p^2}$$
where $\pi$ and $\eta$ are $(-1)$-tower resolutions of $C$ and $D$ respectively. Hence $\varphi$ has a unique proper base-point, which is the singular point $[0:1:0]$ of $C$. Since $C$ is unicuspidal, it follows that after each blow-up in the resolution $\pi$, the strict transform of $C$ and the exceptional curve intersect in a unique point. Since $\sigma_1$ is the minimal $1$-tower resolution of $C$ that factors through the SNC-resoltion, it follows that $\pi$ factors through $\sigma_1$. We then get the following commutative diagram:
$$\xymatrix{
&&Y \ar[dl] \ar[ddrr]^{\eta} \\
&X \ar[dl]_{\sigma_2} \ar[dr]^{\sigma_1} && \\
\p^2 \ar@{<--}[rr]^{\theta_C} && \p^2 \ar@{-->}[rr]^{\varphi} && \p^2.} $$
The morphism $Y \to X$ is given by a tower of blow-ups. For $i \in \{0,\ldots,n\}$, we denote the intermediate surfaces by $X_i$, where $X_0 = X$ and $X_n = Y$ and $X_i$ is obtained after the $i$-th blow-up in this tower. The corresponding exceptional curves, as well as their strict transforms, are denoted by $E_i$. Moreover, we denote by $C_i$ the strict transform of $C$ in $X_i$. In the surface $X=X_0$, the curves $L_x$ and $C_0$ intersect transversely in a unique point and have self-intersections~$-1$ and $1$ respectively. Since $\pi$ is a $(-1)$-tower resolution of $C$, the base-point in $X_0$ lies on the previous exceptional curve, which is the strict transform of $L_x$ by $\sigma_2$. Moreover, since the self-intersection of $C_0$ is $1$, the base-point in $X_0$ also lies on $C_0$, otherwise $C_n$ would have self-intersection $1$ in $Y$. Thus the base-point of $\pi$ in $X_0$ is the intersection point between $C_0$ and $L_x$. We argue similarly that the base-point in $X_1$ is the intersection point between $C_1$ and $E_1$. In $X_2$ we then have the minimal $(-1)$-resolution of $C$ and thus have the following configuration of curves, where the dashed line represents the remaining exceptional curves, the unlabeled curves have self-intersection $-2$, and the thick lines represent $(-1)$-curves:

$$
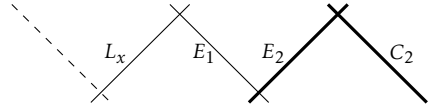
\begin{tikzpicture}[scale=1.3]
\draw [dashed] (-0.8,1)--(0.2,0);
\draw (0,0)--(1,1);
\draw (0.25,0.5) node{\scriptsize $L_x$};
\draw (0.8,1)--(1.8,0);
\draw (1.15,0.5) node{\scriptsize $E_1$};
\draw[very thick] (1.6,0)--(2.6,1);
\draw (1.85,0.5) node{\scriptsize $E_2$};
\draw[very thick] (2.4,1)--(3.4,0);
\draw (3.15,0.5) node{\scriptsize $C_2$};
\end{tikzpicture}
$$

\noindent Since $C_2$ has self-intersection $-1$, none of the subsequent base-points of $\pi$ lie on $C_2$, respectively its strict transforms, otherwise $C_n$ would have self-intersection $<-1$. Since the curves $E_1$ and $C_2$ are not connected in $X_2$ via the other exceptional curves (except $E_2$), it follows that $\pi$ has another base-point in $X_2$, which must lie on $E_2$. This base-point is either the intersection point $p$ between $E_1$ and $E_2$ or lies on $E_2 \setminus (E_1 \cup C_2)$. Let $k \geq 0$ denote the number of base-points proximate to $p$. After blowing up these points, we obtain the following configuration in $X_{k+2}$:

$$
\begin{tikzpicture}[scale=1.3]
\draw [dashed] (-0.8,1)--(0.2,0);
\draw (0,0)--(1,1);
\draw (0.25,0.5) node{\scriptsize $L_x$};
\draw (0.8,1)--(1.8,0);
\draw (1.17,0.5) node{\scriptsize $E_1[-k-2]$};
\draw[very thick] (1.6,0)--(2.6,1);
\draw (2.48,0.5) node{\scriptsize $E_{k+2}$};
\draw (2.4,1)--(3.4,0);
\draw (3.27,0.5) node{\scriptsize $E_{k+1}$};
\draw [dashed] (3.2,0)--(4.2,1);
\draw (4,1)--(5,0);
\draw (4.75,0.5) node{\scriptsize $E_2$};
\draw[very thick] (4.8,0)--(5.8,1);
\draw (5.7,0.5) node{\scriptsize $C_{k+2}$};
\end{tikzpicture}
$$

\noindent Again, we see that $E_1$ is not connected to $E_{k+1}\cup \ldots \cup E_2 \cup C_{k+2}$ and thus $\pi$ has a base-point on $E_{k+2}$, which now lies on $E_{k+2} \setminus E_1$. This base-point is not the intersection point between $E_{k+2}$ and $E_{k+1}$ since the morphism $\eta$ first contracts $C_n$ and then the chain of curves $E_2,\ldots,E_k$. This implies that $E_{k+1}$ is a $(-2)$-curve in $X$. Thus the next base-point lies on $E_{k+2} \setminus (E_1 \cup E_{k+1})$.

We observe that $\eta$ first contracts the chain of curves $C_n, E_2,\ldots,E_{k+2}$. After contracting this chain, the image of $E_1$ has self-intersection $-(k+1)$. This implies that there is a chain of $k$ $(-2)$-curves attached to $E_{k+2}$, which then are contracted by $\eta$, so the image of $E_1$ has self-intersection $-1$ after this chain is contracted. It follows that we have the following configuration in $X_{2k+3}$:

$$
\begin{tikzpicture}[scale=1.3]
\draw [dashed] (-0.8,1)--(0.2,0);
\draw (0,0)--(1,1);
\draw (0.25,0.5) node{\scriptsize $L_x$};
\draw (0.8,1)--(1.8,0);
\draw (1.17,0.5) node{\scriptsize $E_1[-k-2]$};
\draw (1.6,0)--(2.6,1);
\draw (2.47,0.5) node{\scriptsize $E_{k+2}$};
\draw (2.4,1)--(3.4,0);
\draw (3.27,0.5) node{\scriptsize $E_{k+1}$};
\draw [dashed] (3.2,0)--(4.2,1);
\draw (4,1)--(5,0);
\draw (4.75,0.5) node{\scriptsize $E_2$};
\draw[very thick] (4.8,0)--(5.8,1);
\draw (5.75,0.5) node{\scriptsize $C_{2k+3}$};

\draw (2,0.6)--(3.4,-0.8);
\draw (2.58,-0.3) node{{\scriptsize $E_{k+3}$}};
\draw[dashed] (3.2,-0.8)--(4.2,0.2);
\draw (4,0.2)--(5,-0.8);
\draw (4.88,-0.3) node{{\scriptsize $E_{2k+2}$}};
\draw[very thick] (4.8,-0.8)--(5.8,0.2);
\draw (5.75,-0.3) node{{\scriptsize $E_{2k+3}$}};
\end{tikzpicture}
$$

\noindent We now argue that this resolution is in fact $\pi$ itself. Suppose it were not, then there would be another base-point on $E_{2k+3} \setminus E_{2k+2}$, and thus $E_{2k+3}$ is also contracted by $\eta$. We observe that $\eta$ first contracts $C_n$, followed by $E_{2},\ldots,E_{k+2}$, and then $E_{k+3},\ldots,E_{2k+2}$. After these contractions, the image of $E_1$ has self-intersection $-1$ and is contracted next. After that, $L_x$ and all the exceptional curves of $\sigma_1$ are contracted. The next contracted curve must then be the image of $E_{2k+3}$. But we observe that the image of $E_{2k+3}$ after these contractions is singular. This follows from the fact that $C$ is singular and from the symmetry of the configuration in $X_{2k+3}$. But then $E_{2k+3}$ cannot be contracted by $\eta$ and we have a contradiction. It follows that $E_{2k+3}$ is the last exceptional curve in the $(-1)$-tower resolution $\pi$.

We observe moreover, also by the symmetry of the configuration, that $\eta(L_x)$ is a line in $\p^2$ that is very tangent to $D=\eta(E_{2k+3})$ at the singular point. In fact, using the symmetry of the resolution, we obtain a diagram
$$
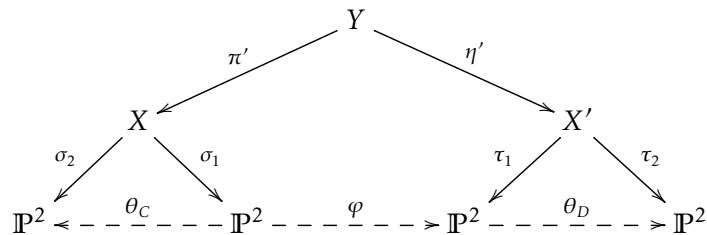
\xymatrix{
&&& Y \ar[dll]_{\pi'} \ar[drr]^{\eta'} &&\\
& X \ar[dl]_{\sigma_2} \ar[dr]^{\sigma_1} &&&& X' \ar[dl]_{\tau_1} \ar[dr]^{\tau_2} & \\
\p^2 \ar@{<--}[rr]^{\theta_C} && \p^2 \ar@{-->}[rr]^{\varphi} && \p^2 \ar@{-->}[rr]^{\theta_D} && \p^2}$$
such that $\eta = \tau_1 \circ \eta'$ where $\tau_1$ is the minimal $1$-tower resolution of $D$, $\eta'$ is the contraction of the curves $C,E_1,\ldots,E_{2k+3}$, and $\theta_D$ is an automorphism of $\p^2 \setminus L_x$ that sends $D$ to $L_z$.

We now consider the birational map $\psi = \theta_D \circ \varphi \circ (\theta_C)^{-1}$, which is an automorphism of $\p^2 \setminus L_z$. With the resolution above, we see that $\psi$ preserves $L_x$. Hence, in the affine chart $z=1$, the map $\psi$ has the form $(x,y) \mapsto \left(ax,by+cx+x^2f(x)\right)$, where $a,b \in \k^*, c \in \k$ and $f \in \k[x]$. Let $\alpha$ be the map $[x:y:z] \mapsto [a^{-1}x:b^{-1}(y-cx):z]$, which is an automorphism of $\p^2 \setminus (L_x \cup L_z)$. We define $\psi' \coloneqq \alpha \circ \psi$ and $\theta_D' \coloneqq \alpha \circ \theta_D$. Then $\psi'$ has the form $(x,y) \mapsto \left(x,y+x^2f(x)\right)$, as claimed.
\end{proof}

\begin{definition}\label{Def:Loc}
{\rm
Let $X$ be an irreducible surface, $C \subset X$ an irreducible curve, and $p \in C$ a point. Let $\mathfrak{a}$ be the kernel of the restriction homomorphism $\mathcal{O}_{X,p} \to \mathcal{O}_{C,p}$, $f \mapsto f|_C$. Then we denote by $\Loc(X,C,p)$ the group of birational maps $\varphi \colon X \dasharrow X$ fixing $p$, such that $\varphi^*$ induces
\begin{enumerate}
\item[$(i)$] an automorphism of $\mathcal{O}_{X,p}$,
\item[$(ii)$] a bijection $\mathfrak{a} \to \mathfrak{a}$,
\item[$(iii)$] the identity on $\mathcal{O}_{X,p} / \mathfrak{a}^2$,
\item[$(iv)$] the identity on $\mathfrak{a}/\mathfrak{a}^3$.
\end{enumerate}}
\end{definition}

\begin{remark}
{\rm If $\varphi \in \Loc(X,C,p)$, then $\varphi$ induces a local isomorphism in a neighborhood of $p$ in $X$ and $C$. Thus for a birational map $\theta \colon X \dasharrow Y$ that is a local isomorphism in a neighborhood of $p \in X$, the conjugation $\psi \mapsto \theta^{-1} \circ \psi \circ \theta$ induces an isomorphism $\Loc(Y,\theta(C),\theta(p)) \to \Loc(X,C,p)$.}
\end{remark}

\begin{lemma}\label{Lem:characterization}
For any $\lambda \in \k$, $\Loc(\A^2,L_x,(0,\lambda))$ coincides with the group of birational maps $\varphi \colon \A^2 \dasharrow \A^2$ such that $\varphi$ and $\varphi^{-1}$ each can be written of the form
$$(x,y) \mapsto \left(x+x^3\alpha(x,y), y+x^2\beta(x,y)\right)$$
for some $\alpha, \beta \in \mathcal{O}_{\A^2,(0,\lambda)}$.
\end{lemma}

\begin{proof}
Let $\varphi$ be a birational map of $\A^2$ of the proposed form. Then $\varphi$ is defined at $(0,\lambda)$ and fixes $(0,\lambda)$. The same is true for $\varphi^{-1}$, so it is a local isomorphism at $(0,\lambda)$ and thus satisfies $(i)$ of Definition~$\ref{Def:Loc}$. One then checks points $(ii)-(iv)$ for the ideal $\mathfrak{a} = (x) \subset \k[x,y]_{(x,y-\lambda)} = \mathcal{O}_{\A^2,(0,\lambda)}$. It follows that $\varphi \in \Loc(\A^2,L_x,(0,\lambda))$.

To prove the converse, let us consider $\varphi \in \Loc(\A^2,L_x,(0,\lambda))$. Since $\varphi^*$ induces an automorphism of $\mathcal{O}_{\A^2,(0,\lambda)}=\k[x,y]_{(x,y-\lambda)}$ we can write $\varphi^*(x) = f$ and $\varphi^*(y) = g$ for some $f,g \in \mathcal{O}_{\A^2,(0,\lambda)}$. As $\varphi^*$ preserves the ideal $(x)$ and induces the identity on $\mathcal{O}_{\A^2,(0,\lambda)} / (x^2)$, we can express $f(x,y) = x+x^2\alpha(x,y)$ and $g(x,y) = y+x^2\beta(x,y)$, for some $\alpha, \beta \in \mathcal{O}_{\A^2,(0,\lambda)}$. Finally, since $\varphi^*$ induces the identity on $(x)/(x^3)$, it follows that $x$ divides $\alpha$ and hence $\varphi$ is of the desired form. Since $\Loc(\A^2,L_x,(0,\lambda))$ is a group, also the inverse of $\varphi$ can be written in this form.
\end{proof}

\begin{proposition}\label{prop:conjugation}
Let $L \subset \p^2$ be a line and $q_1,q_2 \in L$ with $q_1 \neq q_2$. Let $\psi \in \cap_{p \in L \setminus \{q_2\}}\Loc(\p^2,L,p)$ and $\theta \in \Aut(\p^2 \setminus L) \setminus \Aut(\p^2)$ such that $\theta^{-1}$ has base-point $q_1$ and $\theta$ has base-point $q_2$. Then $\theta^{-1} \circ \psi \circ \theta$ lies in $\cap_{p \in L \setminus \{q_2\}}\Loc(\p^2,L,p)$.
\end{proposition}

\begin{proof}
Since the base-point of $\theta^{-1}$ is $q_1$ and the base-point of $\theta$ is not $q_1$ we can by Theorem~$\ref{Thm:Jung}$ write $\theta = j_n \circ a_n \circ \ldots \circ j_1 \circ a_1$ with $j_i \in \Jon(\p^2,L,q_1) \setminus \Aff(\p^2,L)$ and $a_i \in \Aff(\p^2,L) \setminus \Jon(\p^2,L,q_1)$ for $i=1,\ldots,n$. By induction, it suffices to prove the claim for $\theta = j \circ a$ with $j \in \Jon(\p^2,L,q_1) \setminus \Aff(\p^2,L)$ and $a \in \Aff(\p^2,L) \setminus \Jon(\p^2,L,q_1)$.

We then find a minimal resolution
$$\xymatrix{&X \ar[dr]^{\pi} \ar[dl]_{\eta} \\ \p^2 \ar@{-->}[rr]^{j \circ a} && \p^2}$$
where $\pi^{-1}$ has the same base-points as $j^{-1} \in \Jon(\p^2,L,q_1)$. Let $d \geq 2$ be the degree of $j^{-1}$, so we can write $\pi$ as a composition of $2d-1$ blow-ups $\pi \colon X = X_{2d-1} \xrightarrow{\pi_{2d-1}} \ldots \xrightarrow{\pi_2} X_1 \xrightarrow{\pi_1} X_0 = \p^2$, as described in Lemma~$\ref{Lem:decomposition}$. We denote the exceptional curve of $\pi_i$ by $E_i$ for $i=1,\ldots,2d-1$.

We want to lift $\psi$ to a birational transformation of $X$ by conjugation with $\pi$. To do this, we choose coordinates on $\p^2$ such that $L=L_x$ and $q_1=[0:0:1]$ and $q_2 = [0:1:0]$. By Lemma~$\ref{Lem:characterization}$, we can locally express $\psi$ as $$(x,y) \mapsto \left(x+x^3\alpha(x,y),y+x^2\beta(x,y)\right)$$ for some $\alpha, \beta \in \cap_{\lambda \in \k}\mathcal{O}_{\A^2,(0,\lambda)}$. We proceed by conjugating $\psi$ step-by-step with the blow-ups $\pi_i$.

The first blow-up has base-point $(0,0)$ and is locally given by $\pi_1 \colon (x,y) \mapsto (xy,y)$. We thus obtain:

\begin{equation*}
\begin{split}
\pi_1^{-1}\psi\pi_1(x,y) &= \left(\frac{xy+x^3y^3\alpha(xy,y)}{y+x^2y^2\beta(xy,y)},y+x^2y^2\beta(xy,y)\right) \\ 
&= \left(x+x^3y\frac{(y\alpha(xy,y)-b(xy,y))}{1+x^2y\beta(xy,y)},y+x^2y^2\beta(xy,y)\right) \\
&\eqqcolon \left(x+x^3y\alpha_1(x,y),y+x^2y^2\beta_1(x,y)\right) \eqqcolon \psi_1(x,y)
\end{split}
\end{equation*}
In local coordinates of $\A^2 \subset X_1$, the exceptional curve $E_1$ of $\pi_1$ is given by $y=0$ and $\alpha_1, \beta_1 \in \cap_{\lambda \in \k}\mathcal{O}_{\A^2,(0,\lambda)}$.

The base-point of $\pi_2$ is then the point $(0,0) \in E_1$. Indeed, the base-points of $\pi_2,\ldots,\pi_{d}$ all lie on $E_1$, hence each of these blow-ups is of the form $(x,y) \mapsto (x,xy)$, in local coordinates. We can thus write $\pi_2 \circ \ldots \circ \pi_d \colon (x,y) \mapsto (x,x^{d-1}y)$ and thus conjugation with this map yields:
\begin{equation*}
\begin{split}
\psi_d(x,y) &=\left(x+x^{d+2}y\alpha_1(x,x^{d-1}y),\frac{x^{d-1}y+x^{2d}y^2\beta_1(x,x^{d-1}y)}{\left(x+x^{d+2}y\alpha_1(x,x^{d-1}y)\right)^{d-1}}\right) \\ 
&= \left(x+x^{d+2}\alpha_1(x,x^{d-1}y),y+x^{d+1}y^2 \frac{x^{d-1}y^2\beta_1(x,x^{d-1}y)+\ldots}{ \left(1+x^{d+1}y\alpha_1(x,x^{d-1}y)\right)^{d-1}}\right)
\end{split}
\end{equation*}
In local coordinates of $\A^2 \subset X_d$, we can write
$$\psi_d(x,y) = \left(x+x^{d+2}\alpha_d(x,y),y+x^{d+1}\beta_d(x,y)\right)$$
for some $\alpha_d, \beta_d \in \cap_{\lambda \in \k}\mathcal{O}_{\A^2,(0,\lambda)}$.

The base-point of the blow-up $\pi_{d+1}$ is a point on $E_d$ but not $E_{d-1}$. In local coordinates, this means that $\pi_{d+1}$ can be expressed as $(x,y) \mapsto (x,xy+\mu)$, for some $\mu \in \k^*$. The conjugated map is then:
\begin{equation*}
\begin{split}
\psi_{d+1}(x,y)& = \left(x+x^{d+2}\alpha_d(x,xy+\mu), \frac{xy+x^{d+1}\beta_d(x,xy+\mu)}{ x+x^{d+2}\alpha_d(x,xy+\mu)}\right) \\
&= \left(x+x^{d+2}\alpha_d(x,xy+\mu),y+x^d \frac{\beta_d(x,xy+\mu)-xy\alpha_d(x,xy+\mu)}{ 1+x^{d+1}\alpha_d(x,xy+\mu)}\right)
\end{split}
\end{equation*}
and thus we can find $\alpha_{d+1}, \beta_{d+1} \in \cap_{\lambda \in \k}\mathcal{O}_{\A^2,(0,\lambda)}$ such that
$$\psi_{d+1}(x,y) = \left(x+x^{d+2}\alpha_{2d-1}(x,y),y+x^d\beta_{2d-1}(x,y)\right).$$
After conjugating with the $d-2$ remaining blow-ups $\pi_{d+2},\ldots,\pi_{2d-1}$, we thus obtain
$$\psi_{2d-1}(x,y) = \left(x+x^{d+2}\alpha_{2d-1}(x,y),y+x^2\beta_{2d-1}(x,y)\right)$$
for some $\alpha_{2d-1},\beta_{2d-1} \in \cap_{\lambda \in \k}\mathcal{O}_{\A^2,(0,\lambda)}$ and hence it follows that $\psi_{2d-1} \in \Loc(X,E_{2d-1},(0,\lambda))$ for all $\lambda \in \k$ by Lemma~$\ref{Lem:characterization}$.

We now consider the following commutative diagram:
$$\xymatrix{
\p^2 \ar@{-->}[rrr]^{(j \circ a)^{-1}\circ \psi \circ (j \circ a)} \ar@/_3pc/@{-->}[dd]^{j\circ a} &&&\p^2 \ar@/^3pc/@{-->}[dd]_{j\circ a} \\
X \ar@{-->}[rrr]^{\psi_{2d-1}} \ar[d]_{\pi} \ar[u]^{\eta} &&& X \ar[d]^{\pi} \ar[u]_{\eta} \\
\p^2 \ar@{-->}[rrr]^{\psi} &&&\p^2
}$$
For any $p \in L_x \setminus [0:1:0]$, it follows that $\eta$ induces a local isomorphism $\eta^{-1}(p) \to p$ and we thus have $(j \circ a)^{-1} \circ \psi \circ (j \circ a) = \eta \circ \psi_{2d-1} \circ \eta^{-1} \in \Loc(\p^2,L_x,p)$.
\end{proof}

\begin{proof}[Proof of Theorem~$\ref{Thm:Yoshihara}$]
By Lemma~$\ref{Lem:equaldegree}$ the curves $C$ and $D$ have the same degree. Thus the claim of the theorem is clear for lines and conics and we can assume that $C$ has degree at least $3$ and is hence singular, in fact unicuspidal. The isomorphism $\varphi \colon \p^2 \setminus C \to \p^2 \setminus D$ induces a birational map $\p^2 \dasharrow \p^2$. If $\varphi$ extends to an automorphism of $\p^2$, then $C$ and $D$ are projectively equivalent. We thus assume that $\varphi$ does not extend to an automorphism of $\p^2$, i.e.\ $C$ is contracted by $\varphi$. Since $C \setminus L \simeq \A^1$, we can apply Proposition~$\ref{Prop:rectifiabilty}$ by identifying $\p^2 \setminus L \simeq \A^2$, so there exists an automorphism of $\p^2 \setminus L$ that sends $C$ to a line. We can then use Proposition~$\ref{prop:diagram}$ and for suitable coordinates obtain the diagram
$$\xymatrix{\p^2 \ar@{-->}[d]_{\theta_C} \ar@{-->}[r]^{\varphi} & \p^2 \ar@{-->}[d]^{\theta_D} \\ \p^2 \ar@{-->}[r]^{\psi} & \p^2}$$
where $\theta_C, \theta_D \in \Aut(\p^2 \setminus L_x)$ with $\theta_C(C)=\theta_D(D)=L_z$ and $\psi \in \Aut(\p^2 \setminus L_z)$ has the following form $(x,y) \mapsto (x,y+x^2f(x))$; it thus lies in $\Loc(\p^2,L_x,[0:\lambda:1])$ for all $\lambda \in \k$. The base-point $p$ of $\theta_C$ is different from $[0:1:0]$ and is thus of the form $[0:\lambda:1]$ for some $\lambda \in \k$. We then define the map $\rho = (\theta_C)^{-1} \circ \psi \circ \theta_C$, which is an automorphism of $\p^2 \setminus (L_x \cup C)$. It follows from Proposition~$\ref{prop:conjugation}$ that $\rho$ lies in $\Loc(\p^2,L_x,[0:0:1])$ and in particular preserves the line $L_x$. Thus $\rho$ is an automorphism of $\p^2 \setminus C$ and consequently $\varphi' \coloneqq \varphi \circ \rho^{-1}$ is an isomorphism $\p^2 \setminus C \to \p^2 \setminus D$. On the other hand, $\varphi' = (\theta_D)^{-1} \circ \theta_C$ is an automorphism of $\p^2 \setminus L_x$ and hence does not contract $C$. We conclude that $\varphi'$ contracts no curves and is indeed an automorphism of $\p^2$, making the curves $C$ and $D$ projectively equivalent.
\end{proof}

\section{Curves of low degree}
In this section we study Conjecture~$\ref{conjecture:yoshihara}$ for curves of low degree, i.e.\ degree $\leq 8$. It is a case study on the multiplicity sequences that occur (see Definition~$\ref{Def:multseq}$).

\subsection{Cases by multiplicity sequences}

\begin{lemma}\label{Lem:inequalities1}
Let $C \subset \p^2$ be an irreducible curve of degree $d \geq 3$ such that there exists an open embedding $\p^2 \setminus C \hookrightarrow \p^2$ that does not extend to an automorphism of $\p^2$. Then $C$ is a rational curve, where all the proper and infinitely near singular points of $C$ can be ordered from  $p_1$ to $p_k$, with multiplicities $m_1 \geq \ldots \geq m_k \geq 2$, such that $p_1 \in C$ is a proper point and $p_{i+1}$ lies in the first neighborhood of $p_i$, for $i=1,\ldots,k-1$. Moreover, the multiplicities satisfy the following relations:
\begin{align}
d^2 - 3d + 2 &= \sum_{i=1}^k m_i(m_i - 1), \label{genus} \\
d^2 + 1 &\geq \sum_{i=1}^k m_i^2. \label{squares}
\end{align}
\end{lemma}

\begin{proof}
Let $\varphi \colon \p^2 \setminus C \hookrightarrow \p^2$ be an open embedding that does not extend to an automorphism of $\p^2$. Then by Lemma~$\ref{Lem:tower}$ there exists a $(-1)$-tower resolution $\pi \colon X = X_n \xrightarrow{\pi_n} \ldots \xrightarrow{\pi_2} X_1 \xrightarrow{\pi_1} X_0 = \p^2$ of $C$ with base-points $p_1,\ldots,p_n$ and exceptional curves $E_1,\ldots,E_n$, and a $(-1)$-tower resolution $\eta \colon X \to \p^2$ of some curve $D \subset \p^2$ such that $\varphi \circ \pi = \eta$. For $i \in \{1,\ldots,n\}$, we denote by $m_i$ the multiplicity of $C_i$ at $p_i$, so we have $m_1 \geq \ldots \geq m_n$. The strict transform $C_n$ in $X$ is smooth, thus $\pi$ factors through the minimal resolution of singularities of $C$ and blows up all its $k \leq n$ singular points, hence the first part of the claim follows.

For equation~(\ref{genus}), we observe that $C$ is a rational curve since $C_n \simeq \p^1$ and thus has genus $g(C) = 0$. By the genus-degree formula for plane curves we get
$$0 = g(C) = \frac{(d-1)(d-2)}{2} - \sum_{i=1}^{k} \frac{m_i(m_i-1)}{2}$$
and hence identity (\ref{genus}) follows. To see the inequality (\ref{squares}), it is enough to observe that for a blow-up $\pi_i$ with exceptional curve $E_i$, we get
$$\pi_i^*(C_i) = C_{i+1} + m_{i}E_i$$
and hence $(C_{i+1})^2 = (C_i)^2 - m_i^2$, using the identities $(E_i)^2 = -1$ and $C_{i+1} \cdot E_i = m_i$. We then inductively obtain
$$-1 = (C_n)^2 = d^2 - \sum_{i=1}^n m_i^2.$$
The claim then follows from the fact that the number $k$ of singular points is $\leq n$.
\end{proof}

The previous lemma motivates the following definition.

\begin{definition}\label{Def:multseq}
{\rm
Let $C \subset \p^2$ be a curve and let $m_1 \geq \ldots \geq m_k \geq 2$ be some integers. We say that $C$ has \emph{multiplicity sequence} $(m_1,\ldots,m_k)$ if $C$ has (proper or infinitely near) singular points $p_1,\ldots,p_k$ with multiplicities $m_1,\ldots,m_k$ such that $p_1 \in C$ is a proper point and $p_{i+1}$ lies in the first neighborhood of $p_{i}$ for $i \geq 1$, and moreover $C$ is smooth at all other points. For a constant subsequence $(m,\ldots,m)$ of length $l \geq 1$, we also use the short notation $(m_{(l)})$.}
\end{definition}

\begin{remark}
{\rm It is not known to the author whether there exist irreducible curves $C,D \subset \p^2$ that have isomorphic complements but have different multiplicity sequences.}
\end{remark}

\begin{lemma}\label{Lem:inequalities2}
Let $C \subset \p^2$ be an irreducible curve of degree $d \geq 3$ with multiplicity sequence $(m_1, \ldots, m_k)$, where we set $m_2 \coloneqq 1$ if $k=1$. If there exists an open embedding $\p^2 \setminus C \hookrightarrow \p^2$ that does not extend to an automorphism of $\p^2$, then the following inequalities hold:
$$m_1+m_2 \leq d < 3m_1.$$
\end{lemma}

\begin{proof}
We use the set-up of the proof of Lemma~$\ref{Lem:inequalities1}$ and extend the multiplicity sequence $(m_1, \ldots, m_k)$ by $m_{k+1} = \ldots=m_n = 1$ such that both (\ref{genus}) and (\ref{squares}) from Lemma~$\ref{Lem:inequalities1}$ become equalities. We then subtract (\ref{genus}) from (\ref{squares}) for the extended multiplicity sequence and obtain
$$3d-1 = \sum_{i=1}^n m_i.$$
We then multiply this equation by $\frac{d}{3}$ and subtract (\ref{squares}), so we get
$$-\left(1+\frac{d}{ 3}\right) = \sum_{i=1}^n m_i\left(\frac{d}{3} - m_i\right).$$
Since the right-hand side of this equation is negative, so is the left-hand side. Thus, at least one of the terms $\frac{d}{3}-m_i$ is negative. The inequality $d < 3m_1$ now follows from the fact that the multiplicity sequence is non-increasing.

The inequality $m_1+m_2 \leq d$ follows from B\'ezout's theorem, where we intersect $C$ with a line going through points $p_1$ and $p_2$ of multiplicity $m_1$ and $m_2$ respectively.
\end{proof}

\begin{corollary}\label{Cor:table}
Let $C \subset \p^2$ be an irreducible curve of degree $\leq 8$ such that there exists an open embedding $\p^2 \setminus C \hookrightarrow \p^2$ that does not extend to an automorphism of $\p^2$. Then $C$ has one of the multiplicity sequences shown in the following table.
\begin{table}[ht]
\centering
\begin{tabular}{|c|l|}
\hline
degree & multiplicity sequences \\ \hline
$3$ & $(2)$ \\ \hline
$4$ & $(3); (2_{(3)})$ \\ \hline
$5$ & $(4); (3,2_{(3)}); (2_{(6)})$ \\ \hline
$6$ & $(5); (4,2_{(4)}); (3_{(3)},2); (3_{(2)},2_{(4)}); (3,2_{(7)})$ \\ \hline
$7$ & $(6); (5,2_{(5)}); (4,3_{(3)}); (4,3_{(2)},2_{(3)}); (4,3,2_{(6)}); (3_{(4)},2_{(3)})$ \\ \hline
$8$ & $(7); (6,2_{(6)}); (5,3_{(3)},2_{(2)}); (5,3_{(2)},2_{(5)}); (4_{(3)},3); (4_{(3)},2_{(3)}); (4_{(2)},3_{(3)}); $ \\
& $(4_{(2)},3_{(2)},2_{(3)}); (4_{(2)},3,2_{(6)}); (4,3_{(5)}); (4,3_{(4)},2_{(3)}); (3_{(7)})$ \\ \hline
\end{tabular}
\caption{Multiplicity sequences for degree $\leq 8$.}
\label{table:sequences}
\end{table}
\end{corollary}

\begin{proof}
This follows from computations using Lemma~$\ref{Lem:inequalities1}$ and Lemma~$\ref{Lem:inequalities2}$, but we need to look at one case more carefully. In degree $7$ the multiplicity sequence $(3_{(5)})$ is consistent with the inequalities in Lemma~$\ref{Lem:inequalities1}$ and Lemma~$\ref{Lem:inequalities2}$. Suppose that there exists such a curve $C$ and denote by $p_1,p_2,p_3$ the first $3$ singular points, all of multiplicity $3$. By B\'ezout's theorem these points are not collinear. Moreover, $p_3$ is not proximate to $p_1$ as the sum of the multiplicities of the strict transform of $C$ at $p_2$ and $p_3$ is larger than the multiplicity at $p_1$. Thus there exists a quadratic transformation $q$ with base-points $p_1,p_2,p_3$. The degree of $q(C)$ is then $2\cdot7 - 3 - 3 - 3 = 5$ by Lemma~$\ref{Lem:formula}$ and $q(C)$ has two singular points of multiplicity $3$. But this is not possible by Lemma~$\ref{Lem:inequalities2}$. Hence no curve of of degree $7$ with multiplicity sequence $(3_{(5)})$ exists.
\end{proof}

The case of cubic curves is then straightforward.

\begin{lemma}
Let $C \subset \p^2$ be a cubic curve and let $\varphi \colon \p^2 \setminus C \to \p^2 \setminus D$ an isomorphism, where $D \subset \p^2$ is some curve. Then $C$ and $D$ are projectively equivalent.
\end{lemma}

\begin{proof}
If $\varphi$ extends to an automorphism of $\p^2$, the claim is clear. If not, then $C$ is rational and hence singular with a point of multiplicity $2$. It is a well known fact that can be checked by simple computations that there are only two singular cubic curves, up to projective equivalence. One class is represented by the cuspidal cubic curve $x^2z-y^3 = 0$ and the other class by the nodal cubic curve $x^2z - y^3 - y^2z = 0$. It follows from Lemma~$\ref{Lem:equaldegree}$ that $D$ is again a cubic curve and by Proposition~$\ref{Prop:form}$ that the singularity of $D$ is of the same type as the singularity of $C$, i.e.\ $D \setminus \Sing(D) \simeq \A^1$ if $C$ is unicuspidal or $D \setminus \Sing(D) \simeq \A^1 \setminus \{0\}$ if $C$ is nodal. Hence $C$ and $D$ are projectively equivalent.
\end{proof}

\begin{remark}\label{remark:cubic}
{\rm The complement of a nodal cubic curve has infinitely many automorphisms, up to composition with automorphisms of $\p^2$. For a description, see for instance \cite[Lemma~$2.24$]{Yos85}. The automorphism group of the complement of a cuspidal cubic is even infinite dimensional, see \cite[Theorem~A $(6)$]{Yos85}.}
\end{remark}

We will frequently use the following formula for intersection numbers.

\begin{lemma}\label{Lem:proximate mult}
Let $C \subset \p^2$ be a curve and $\pi \colon X_n \xrightarrow{\pi_n} \ldots \xrightarrow{\pi_2} X_1 \xrightarrow{\pi_1} X_0 = \p^2$ a $(-1)$-tower resolution of $C$ with base-points $p_1,\ldots,p_n$ and exceptional curves $E_1,\ldots,E_n$. For $i \leq k \leq n$, we then have
$$C_k \cdot E_i = m_{p_i}(C_i) - \sum_{p_j \succ p_i, j \leq k} m_{p_j}(C_j).$$
\end{lemma}

\begin{proof}
Let $i,k \in \N$ with $i \leq k \leq n$. We denote by $\overline{E}_j$ the total transform of $E_j$ in $X_k$ for $j=1,\ldots,k$. By \cite[Corollary~$1.1.25$]{Alb02}, we can then write
$$E_i = \overline{E}_i - \sum_{p_j \succ p_i, j\leq k} \overline{E}_j.$$
By \cite[Corollary~$1.1.27$]{Alb02}, we have $C_k \cdot \overline{E}_j = m_{p_j}(C_j)$ and the claim follows.
\end{proof}

\begin{lemma}\label{Lem:jump}
Let $C \subset \p^2$ be an irreducible curve that has multiplicity sequence $(m_1, \ldots, m_k)$. If there exist $r < s \leq k-2$ such that
\begin{align*}
m_{r+1} + m_{r+2} &> m_r > m_{r+1}, \\
m_{s+1} + m_{s+2} &> m_s > m_{s+1}, \\
m_s + m_{s+1} &> m_{s-1},
\end{align*}
then every open embedding $\p^2 \setminus C \hookrightarrow \p^2$ extends to an automorphism of $\p^2$.
\end{lemma}

\begin{proof}
Suppose that there exists an open embedding $\varphi \colon \p^2 \setminus C \hookrightarrow \p^2$ that does not extend to an automorphism of $\p^2$. Then by Lemma~$\ref{Lem:tower}$ there exists $\pi \colon X = X_n \xrightarrow{\pi_n} \ldots \xrightarrow{\pi_2} X_1 \xrightarrow{\pi_1} X_0 = \p^2$ a $(-1)$-tower resolution of $C$ with base-points $p_1,\ldots,p_n$ and exceptional curves $E_1,\ldots,E_n$, and a $(-1)$-tower resolution $\eta \colon X \to \p^2$ of some curve $D \subset \p^2$ such that $\varphi \circ \pi = \eta$. For any $i \in \{1,\ldots,k\}$, we obtain from Lemma~$\ref{Lem:proximate mult}$ the equation
$$C_n \cdot E_i = m_i - \sum_{p_j \succ p_i} m_j.$$
The point $p_{r+1}$ is proximate to $p_r$, but $p_{r+2}$ is not, as $C_n \cdot E_r \geq 0$ and $m_{r+1} + m_{r+2} > m_r$. Hence we have $C_n \cdot E_r = m_{r} - m_{r+1} > 0$. Analogously we get $C_n \cdot E_s > 0$. The curve $E_1 \cup \ldots \cup E_{n-1} \cup C_n$ in $X$ is the exceptional locus of $\eta$ and thus has a tree structure. By the same argument as before, the point $p_{s+1}$ is not proximate to $p_{s-1}$, hence it follows that the curves $E_r$ and $E_s$ are connected in $E_1 \cup \ldots \cup E_{n-1}$ via some chain of curves. Since $E_r$ and $E_s$ are also connected via $C_n$, this yields a contradiction to the tree structure of $E_1 \cup \ldots \cup E_{n-1} \cup C_n$.
\end{proof}

\begin{corollary}\label{cor:jumpcase}
Let $C \subset \p^2$ be an irreducible rational curve with one of the multiplicity sequences $(4,3,2_{(6)})$, $(4,3_{(2)},2_{(3)})$, $(4,3_{(4)},2_{(3)})$, $(4_{(2)},3,2_{(6)})$, $(4_{(2)},3_{(2)},2_{(3)})$, $(5,3_{(2)},2_{(5)})$, or $(5,3_{(3)},2_{(2)})$. Then any open embedding $\p^2 \setminus C \hookrightarrow \p^2$ extends to an automorphism of $\p^2$.
\end{corollary}

\begin{proof}
This follows directly from Lemma~$\ref{Lem:jump}$.
\end{proof}

\subsection{The unicuspidal case and a special quintic curve}
\label{subsec:quintic}

If $C \subset \p^2$ is a unicuspidal curve that admits a very tangent line through the singular point, then Theorem~$\ref{Thm:Yoshihara}$ gives an affirmative answer to Conjecture~$\ref{conjecture:yoshihara}$. In low degrees this is often the case, as we will see using the following lemma, which we can already find in \cite{Yos84}.

\begin{lemma}\label{lemma:first2}
Let $C \subset \p^2$ be a curve with multiplicity sequence $(m_1,\ldots,m_k)$, where we set $m_2 =1$ if $k=1$. If $\deg(C) = m_1 + m_2$, then there exists a very tangent line to $C$ through the proper singular point.
\end{lemma}

\begin{proof}
Let $p_1 \in C$ be the proper singular point of multiplicity $m_1$ and $p_2$ a point infinitely near to $p_1$ with multiplicity $m_2$. Then there exists a line $L$ through $p_1$ and $p_2$. We then get the local intersection $(C \cdot L)_{p_1} \geq m_1 + m_2 = \deg(C)$. By B\'ezout's theorem $L$ intersects $C$ in no other point and we have equality $(C \cdot L)_{p_1} = \deg(C)$, and thus $L$ is very tangent to $C$.
\end{proof}

In Table~$\ref{table:sequences}$, we find the multiplicity sequence $(2_{(6)})$ for quintic curves. It follows from B\'ezout's theorem that such curves do not admit a very tangent line through the singular point and hence Theorem~$\ref{Thm:Yoshihara}$ does not apply. We thus have to study this case separately. This seems to be a well known class of curves and was already considered in \cite{Yos84} and \cite{Yos79}, but without full proofs. Over the field of complex numbers, unicuspidal quintic curves were classified in \cite[Theorem~$2.3.10.$]{Nam84}. For the sake of completeness, we give a self-contained treatment of the case of unicuspidal curves with multiplicity sequence $(2_{(6)})$ below.

\begin{lemma}\label{Lemma:6points}
Let $C$ and $D \subset \p^2$ be irreducible unicuspidal quintic curves with multiplicity sequence $(2_{(6)})$ with singular points $p_1,\ldots,p_6$ and $q_1,\ldots,q_6$ respectively. Then there exists $\alpha \in \Aut(\p^2)$ such that $\alpha(p_i) = q_i$ for $i=1,\ldots,6$.
\end{lemma}

\begin{proof}
Let $L \subset \p^2$ be the line through $p_1$ and $p_2$. The singular points $p_1,p_2,p_3$ of $C$ all have multiplicity $2$, thus they are not collinear by B\'ezout's theorem. It follows that there exists a quadratic map $\theta_1 \colon \p^2 \dasharrow \p^2$ with base-points $p_1,p_2,p_3$ and exceptional curves $E_1, E_2, E_3$. The map $\theta_1$ is then given by first blowing up $p_1, p_2, p_3$ and then contracting $L_3,E_2,E_1$, as shown below. We denote by $p_1',p_2',p_3'$ the base-points of $(\theta_1)^{-1}$ and by $p_4',p_5',p_6'$ the singular points of $C' \coloneqq \theta_1(C)$.
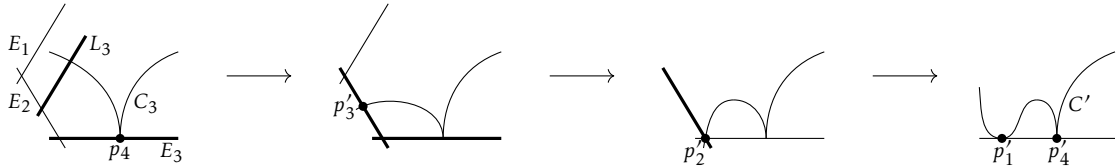
\begin{center}
\begin{tikzpicture}[scale=0.85]

\draw (0.75,0)--(0,1.25);
\node at (0.05,0.7) {\scriptsize $E_2$};
\draw (0,1)--(0.75,2.25);
\node at (0.05,1.6) {\scriptsize $E_1$};
\draw[very thick] (0.32,0.5)--(1.07,1.75);
\node at (1.3,1.6) {\scriptsize $L_3$};
\draw[very thick] (0.5,0.15)--(2.5,0.15);
\node at (2.4,-0.07) {\scriptsize $E_3$};
\draw (1.6,0.15) to [out=90, in=200] (2.5,1.5);
\draw (1.6,0.15) to [out=90, in=-20] (0.5,1.5);
\node at (2,0.7) {\scriptsize $C_3$};
\fill (1.6,0.15) circle (2pt);
\node at (1.6,-0.07) {\scriptsize $p_4$};

\draw[->] (3.25,1.125) -- (4.25,1.125);

\draw[very thick] (5.75,0)--(5,1.25);
\draw (5,1)--(5.75,2.25);
\draw[very thick] (5.5,0.15)--(7.5,0.15);
\draw (6.6,0.15) to [out=90, in=200] (7.5,1.5);
\draw (6.6,0.15) to [out=90, in=30] (5.25,0.6);
\fill (5.36,0.65) circle (2pt);
\node at (5.08,0.65) {\scriptsize $p_3'$};

\draw[->] (8.25,1.125) -- (9.25,1.125);

\draw[very thick] (10.75,0)--(10,1.25);
\draw (10.5,0.15)--(12.5,0.15);
\draw (11.6,0.15) to [out=90, in=200] (12.5,1.5);
\draw (11.6,0.15) to [out=90, in=0] (11.1,0.75);
\draw (11.1,0.75) to [out=180, in=80] (10.63,0);
\fill (10.66,0.15) circle (2pt);
\node at (10.5,-0.07) {\scriptsize $p_2'$};

\draw[->] (13.25,1.125) -- (14.25,1.125);

\draw (14.9,0.15)--(17,0.15);
\draw (16.1,0.15) to [out=90, in=200] (17,1.5);
\draw (16.1,0.15) to [out=90, in=0] (15.79,0.75);
\draw (15.79,0.75) to [out=180, in=0] (15.25,0.15);
\draw (15.25,0.15) to [out=180, in=-85] (14.9,0.95);
\node at (16.45,0.7) {\scriptsize $C'$};
\fill (15.25,0.15) circle (2pt);
\node at (15.29,-0.07) {\scriptsize $p_1'$};
\fill (16.1,0.15) circle (2pt);
\node at (16.14,-0.07) {\scriptsize $p_4'$};

\end{tikzpicture}
\end{center}
By Lemma~$\ref{Lem:formula}$, the degree of $C'$ is $2 \cdot 5 - 1 \cdot 2 - 1 \cdot 2- 1 \cdot 2 = 4$ and hence $C'$ is a unicuspidal quartic curve. Likewise, there exists a quadratic map $\theta_2$ that sends $D$ to a unicuspidal quartic curve $D'$, where we analogously denote the points $q_1',\ldots,q_6'$.

We show that there exists an automorphism $\alpha' \in \Aut(\p^2)$ such that $\alpha'(p_i') = q_i'$ for $i=1,\ldots,6$, which implies that the map $\alpha = (\theta_2)^{-1} \circ \alpha' \circ \theta_1$ is an automorphism of $\p^2$ that sends $p_i$ to $q_i$, for $i=1,\ldots,6$, since the base-points of $(\theta_1)^{-1}$ are sent to the base-points of $(\theta_2)^{-1}$.

We can assume that we have $p_1' = q_1' = [0:0:1]$ and $p_4' = q_4' = [0:1:0]$ (after a linear change of coordinates). By B\'ezout's theorem the points $p_1', p_4', p_5'$ are not collinear, thus we can moreover assume that $p_5'$, respectively $q_5'$, corresponds to the tangent direction $L_z$.

The points $p_1', p_2', p_4'$ are in fact collinear and thus $p_2'$ corresponds to the tangent direction $L_x$, and the same is the case for $q_2'$. The linear maps fixing $p_1', p_2', p_4', p_5'$ then correspond to matrices in $\PGL_3$ of the form
$$
\begin{pmatrix}
a & 0 & 0 \\
b & c & 0 \\
0 & 0 & 1
\end{pmatrix}
$$
where $a,b,c \in \k$ and $ac \neq 0$. We now consider the action of these linear maps on the points $p_3'$ and $p_6'$. We thus blow up the point $p_1' = [0:0:1]$. In local coordinates, this blow-up is given by $(u,v) \mapsto [uv:v:1]$ and moreover $p_2' = (0,0)$. With a linear map of the above form, we get $[uv:v:1] \mapsto [auv:buv+cv:1]$ and the induced map in the blow-up is locally given by $(u,v) \mapsto \left(\frac{au}{ bu+c},(bu+c)v\right)$. The induced map on the exceptional curve is then $[u:v] \mapsto [\frac{a}{c} u : c v] = [\frac{a}{c^2}u : v]$. We observe that $p_3'$ is not proximate to $p_1'$ and that $p_3'$ is not collinear with $p_1',p_2'$ and $p_4'$ by B\'ezout's theorem. Thus $p_3'$ is neither of the points $[0:1]$ or $[1:0]$ on the exceptional curve and we can assume that $p_3' = q_3' = [1:1]$. From this we obtain the condition $a = c^2$.

For the point $p_6'$, we consider the blow-up of $p_4' = [0:1:0]$,  given by $(u,v) \mapsto [u:1:uv]$ in local coordinates, and $p_5' = (0,0)$. Applying a linear map of the form above, we get $[u:1:uv] \mapsto [au:bu+c:uv]$ and the induced map on the blow-up is given by $(u,v) \mapsto \left(\frac{au}{bu+c},\frac{v}{a}\right)$, in local coordinates. The induced map on the exceptional curve is $[u:v] \mapsto [\frac{a}{c} u : \frac{1}{a} v] = [\frac{a^2}{c}u : v] = [c^3 u : v]$. As before, we see that $p_6'$ is not proximate to $p_4'$ and is not collinear with $p_4'$ and $p_5'$. Hence we can also assume that $p_6' = q_6' = [1:1]$ and get the condition $c=1$.

We have thus found a linear map that sends $p_i'$ to $q_i'$ for $i=1,\ldots,6$ and the claim follows.
\end{proof}

\begin{proposition}\label{proposition:quintic}
Let $C \subset \p^2$ be an irreducible unicuspidal quintic curve with multiplicity sequence $(2_{(6)})$. Then $C$ is projectively equivalent to the curve
$$Q \colon (xz+y^2) \left( (xz+y^2)z + 2x^2y \right) - x^5=0.$$
\end{proposition}

\begin{proof}
We start by constructing a birational map $\p^2 \dasharrow \p^2$ that sends the line $L_z$ to the quintic curve $Q$. To do this we consider first the quadratic map $\theta_1 \colon [x:y:z] \dashmapsto [x^2:xy:xz+y^2]$. This map is an automorphism of $\p^2 \setminus L_x$ and sends the line $L_z$ to the conic $xz+y^2 = 0$. Next, consider the quadratic map $\theta_2 \colon [x:y:z] \dashmapsto [xz:x^2-yz:z^2]$, which induces an automorphism of $\p^2 \setminus L_z$. We compute the composition $\psi \coloneqq (\theta_1)^{-1} \circ \theta_2 \circ \theta_1$ and obtain
$$[x:y:z] \dashmapsto [x(xz+y^2)^2:(xz+y^2)\big(x^3 - y(xz+y^2) \big):(xz+y^2)\big(z(xz+y^2)+2x^2y)\big)-x^5].$$
The map $\psi$ is an automorphism of the complement of the conic $xz+y^2=0$ in $\p^2$ and is moreover an involution. Hence both $\psi$ and $\psi^{-1}$ contract the conic $xz+y^2 = 0$ and have a unique proper base-point $[0:0:1]$. The image of the line $L_z$ by $\psi$ is exactly the quintic curve $Q$. The degree of $\psi$ is $5$ and the linear system of $\psi$ contains the curve $Q$ whose only proper singular point is $[0:0:1]$ with multiplicity $2$, thus by the Noether equations $\psi$ has $6$ base-points of multiplicity $2$, which then must be the same as the singular points of $Q$.

Let $C$ be any unicuspidal quintic curve with multiplicity sequence $(2_{(6)})$. We can assume by Lemma~$\ref{Lemma:6points}$ that after a change of coordinates the $6$ (proper and infinitely near) singular points of $C$ and $Q$ coincide. Hence by Lemma~$\ref{Lem:formula}$ the birational map $\psi^{-1}$ sends the curve $C$ to a curve of degree $5\cdot5-2\cdot2-2\cdot2-2\cdot2-2\cdot2-2\cdot2-2\cdot2 = 1$, i.e.\ a line. This line is tangent to the conic $xz+y^2 = 0$ since $C$ is unicuspidal and the line does not pass through the base-point $[0:0:1]$ of $\psi$. The tangents to the conic $xz+y^2 = 0$ that do not pass through $[0:0:1]$ are parametrized by the family $L_\alpha \colon \alpha^2x+2\alpha y-z = 0$, where $\alpha \in \k$. We then compute the equation of the image of $L_\alpha$ under $\psi$ and get
$$Q_\alpha \colon (xz+y^2) \left( (xz+y^2)(\alpha^2x-2\alpha y-z)+2x^2(\alpha x-y) \right) +x^5 = 0.$$
Thus $C = Q_\alpha$, for some $\alpha \in \k$. A short computation shows that the automorphism of $\p^2$ given by
$$[x : y : z] \mapsto [x : \alpha x + y : -\alpha^2x - 2\alpha y + z]$$
sends the curve $Q_\alpha$ to the curve $Q_0 = Q$.
\end{proof}

\begin{corollary}\label{cor:onequintic}
Let $Q \subset \p^2$ be an irreducible unicuspidal quintic curve with multiplicity sequence $(2_{(6)})$ and $\varphi \colon \p^2 \setminus Q \to \p^2 \setminus D$ an isomorphism, where $D \subset \p^2$ is some curve. Then $D$ is projectively equivalent to $Q$.
\end{corollary}

\begin{proof}
By Lemma~$\ref{Lem:equaldegree}$ and Proposition~$\ref{Prop:form}$, the curve $D$ is also a rational unicuspidal quintic. It thus has one of the multiplicity sequences $(4), (3,2_{(3)})$, or $(2_{(6)})$ by Corollary~$\ref{Cor:table}$. In the first two cases, $D$ admits a very tangent line through the singular point by Lemma~$\ref{lemma:first2}$, and thus by Theorem~$\ref{Thm:Yoshihara}$, this would also hold for the curve $Q$. Since $Q$ does not admit a very tangent line through the singular point, it follows that $D$ has multiplicity sequence $(2_{(6)})$ and is hence projectively equivalent to $Q$ by Proposition~$\ref{proposition:quintic}$.
\end{proof}

To conclude the case of unicuspidal curves, we need two more observations.

\begin{lemma}\label{lemma:notunicuspidal}
Let $C \subset \p^2$ be a rational irreducible curve with one of the multiplicity sequences $(3_{(4)},2_{(3)})$, $(4,3_{(5)})$, $(4,3_{(4)},2_{(3)})$, or $(5,2_{(5)})$. Then $C$ is not unicuspidal.
\end{lemma}

\begin{proof}
Let $\pi \colon X = X_k \xrightarrow{\pi_{k}} \ldots \xrightarrow{\pi_2} X_1 \xrightarrow{\pi_1} X_0 = \p^2$ be a minimal resolution of singularities of $C$, where $\pi_i$ is the blow-up of the singular point $p_i \in X_i$ of multiplicity $m_i$ and has exceptional curve $E_i$ for $i=1,\ldots,k$. It follows that $C_k$ intersects $E_k$ with multiplicity $m_k$. If there exists some $i \leq k-2$ such that $m_i - m_{i+1} = 1$, it follows from Lemma~$\ref{Lem:proximate mult}$ that
$$C_k \cdot E_i = m_i - \sum_{p_j \succ p_i} m_j = m_i - m_{i+1} = 1$$
since $C_k \cdot E_i \geq 0$ and $m_{i+2} \geq 2$. If $E_i$ does moreover not intersect $E_k$, it follows that $C$ is not unicuspidal, as $C_k$ intersects the exceptional locus $E_1 \cup \ldots \cup E_k$ of $\pi$ in at least two points, one on $E_i$ and one on $E_k$. We observe that this is the case for the multiplicity sequences $(3,2_{(7)})$, $(3_{(4)},2_{(3)})$, $(4,3_{(5)})$, and $(4,3_{(4)},2_{(3)})$, since in each case the exceptional curves in their minimal resolution of singularities form a chain where $E_i$ and $E_k$ do not intersect, as one checks with Lemma~$\ref{Lem:proximate mult}$.

Similarly, we see with Lemma~$\ref{Lem:proximate mult}$ that for the multiplicity sequence $(5,2_{(5)})$, either $p_3$ is proximate to $p_1$ or not, but in both cases the curve $C_7$ intersects $E_1$ and $E_7$ in distinct points and thus $C$ is again not unicuspidal.
\end{proof}

\begin{lemma}\label{lemma:unicuspidalcontractible}
Let $C \subset \p^2$ be a rational, unicuspidal curve of degree $d$ and multiplicity sequence $(m_1,\ldots,m_k)$. There exists an open embedding $\p^2 \setminus C \hookrightarrow \p^2$ that does not extend to an automorphism of $\p^2$ if and only if exactly one of the following possibilities holds.

\begin{enumerate}
\item[$(i)$] $d^2 - \sum_{i=1}^k m_i^2 = -1$ and $m_{k-1} - m_k = 1$.
\item[$(ii)$] $d^2 - \sum_{i=1}^k m_i^2 - m_k = -2$ and $m_k = 2$, $m_{k-1} \neq 3$.
\item[$(iii)$] $d^2 - \sum_{i=1}^k m_i^2 - m_k \geq -1$.
\end{enumerate}
\end{lemma}

\begin{proof}
We first prove the direction $(\Rightarrow)$ and suppose that there exists an open embedding $\varphi \colon \p^2 \setminus C \hookrightarrow \p^2$ that does not extend to an automorphism of $\p^2$ and show that we are in one of the cases $(i), (ii)$, or $(iii)$. It follows by Lemma~$\ref{Lem:tower}$ that there exists a $(-1)$-tower resolution $\pi \colon X = X_n \xrightarrow{\pi_n} \ldots \xrightarrow{\pi_2} X_1 \xrightarrow{\pi_1} X_0 = \p^2$ of $C$ with base-points $p_1,\ldots,p_n$ and exceptional curves $E_1,\ldots,E_n$, and a $(-1)$-tower resolution $\eta \colon X \to \p^2$ of some curve $D \subset \p^2$ such that $\varphi \circ \pi = \eta$. Then $E_1 \cup \ldots \cup E_{n-1} \cup C_n$ is the exceptional locus of $\eta$, being the support of an SNC-divisor that has a tree structure. The minimal resolution of singularities of $C$ is $\pi_1 \circ \ldots \circ \pi_k$. The curve $C_k$ intersects $E_k$ and since $C$ is unicuspidal this intersection is in a single point with multiplicity $m_k$ (see Figure~$\ref{fig:unicuspidal}$ on the left). Since $\pi$ is a $(-1)$-tower resolution of $C$, the self-intersection of $C_k$ is $\geq -1$.

Suppose that $(C_k)^2 = -1$. Then $\pi$ has no other base-point, as this point would lie on $E_k \setminus C_k$, and this would imply that $C_n$ and $E_k$ do not intersect transversely in $X$. Moreover, the configuration of the curves $E_1,\ldots,E_{k-1},C_k$ is connected, i.e.\ $C_k$ transversely intersects exactly one curve $E \in \{E_1,\ldots,E_{k-1}\}$ in its intersection point with $E_k$. We observe that $C_k$ intersects $E_1\cup\ldots \cup E_{k-1}$ only in the curve $E$, and thus $E_1\cup\ldots \cup E_{k-1}$ is connected. But this implies that $E_k$ intersects only one curve from $E_1,\ldots,E_{k-1}$, and thus $E = E_{k-1}$. Now it follows from the fact that $E_{k-1} \cdot C_k = 1$ and from Lemma~$\ref{Lem:proximate mult}$ that $m_{k-1} - 1= m_k$ and we are thus in case $(i)$.

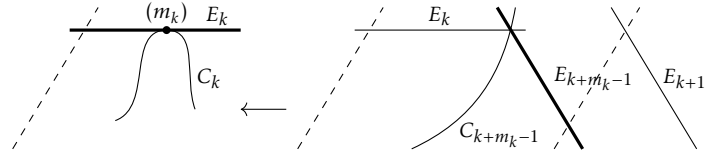
\begin{figure}[ht]
\centering
\begin{tikzpicture}[scale=0.75]
\draw[dashed] (2,0) -- (3.5,2.5);

\draw[very thick] (3,2.1) -- (6,2.1);
\node at (5.6,2.35) {\scriptsize $E_{k}$};

\draw (3.8,0.5) to [out=20, in=180] (4.7,2.1);
\draw (4.7,2.1) to [out=0, in=110] (5.2,0.7);

\draw (4.7,2.4) node{$\scriptstyle (m_k)$};
\fill (4.7,2.1) circle (2pt);

\node at (5.45,1.25) {\scriptsize $C_{k}$};

\draw[<-] (6,0.7) -- (6.8,0.7);

\draw[dashed] (7,0) -- (8.5,2.5);

\draw (8,2.1) -- (11,2.1);
\node at (9.5,2.35) {\scriptsize $E_{k}$};

\draw[very thick] (10.5,2.5) -- (12,0);
\draw[dashed] (11.5,0) -- (13,2.5);
\draw (12.5,2.5) -- (14,0);

\draw (10.82,2.5) to [out=-100, in=25] (9,0);
\node at (10.55,0.25) {\scriptsize $C_{k+m_k-1}$};
\node at (13.8,1.25) {\scriptsize $E_{k+1}$};
\node at (12.2,1.25) {\scriptsize $E_{k+m_k-1}$};

\end{tikzpicture}
\caption{Blow-up of the points $p_k,\ldots,p_{k+m_k-2}$.}\label{fig:unicuspidal}
\end{figure}

Suppose now that $(C_k)^2 \neq -1$. Then $\pi$ has a base-point on $E_k \cap C_k$. Thus $k<n$ and the union of the curves $E_1,\ldots,E_{n-1},C_n$ is SNC in $X$. It follows that the base-point $p_{i+1}$ is the intersection point between $C_i$ and $E_k$ for $i=k,\ldots,k+m_k-2$. The configuration of curves in $X_{k+m_k-1}$ is shown in the diagram on the right in Figure~$\ref{fig:unicuspidal}$. The self-intersection of $C_{k+m_k-1}$ is then $d^2 - \sum_{i=1}^k m_i^2 - (m_k-1)$, and this number is $\geq -1$, since $\pi$ is a $(-1)$-tower resolution of $C$. 

Assume that $d^2 - \sum_{i=1}^k m_i^2 - m_k = -2$, i.e.\ there is no base-point on $C_{k+m_k-1}$. But this means that there is no more base-point at all, since there is a triple intersection between $E_k, E_{k+m_k-1}$ and $C_{k+m_k-1}$, which would violate the SNC structure of the exceptional divisor of $\eta$ if $E_{k+m_k-1}$ was not the last exceptional curve of $\pi$. Since the union of $E_1, \ldots$, $E_{k+m_k-2}$, $C_{k+m_k-1}$ is connected, it follows that $m_k=2$ (see Figure~$\ref{fig:unicuspidal}$). It also follows that the union of $E_1,\ldots,E_{k+m_k-1}$ is connected and hence $C_k$ does not intersect any other exceptional curve apart from $E_k$ in $X_k$. It then follows from Lemma~$\ref{Lem:proximate mult}$ that $m_{k-1} - m_k \neq 1$ and thus $m_{k-1} \neq 3$. We are thus in case $(ii)$.

The last remaining case is when $d^2 - \sum_{i=1}^k m_i^2 - m_k \neq -2$, but then this expression is $\geq -1$ and we are in case $(iii)$. We observe moreover that the cases $(i)$, $(ii)$, $(iii)$ are mutually exclusive.

We now prove the direction $(\Leftarrow)$. In each case we first blow up the $k$ singular points of $C$ (with exceptional curves $E_1,\ldots,E_k$). In case~$(i)$, this yields the resolution in Figure~$\ref{fig:unicuspidal1}$. By the symmetry of the configuration, there exists a morphism from this surface to $\p^2$ contracting $C_k,E_{k-1},\ldots,E_1$.

\begin{figure}[ht]
\centering
\begin{tikzpicture}[scale=0.75]
\draw[dashed] (1,2.5) --(2.5,0);
\draw (2,0) -- (3.5,2.5);
\node at (2.98,0.7){\scriptsize $E_{k-1}$};

\draw[very thick] (2,2.1) -- (5,2.1);
\node at (4.4,2.35) {\scriptsize $E_{k}$};

\draw[very thick] (2.35,1) to [out=70, in=180] (3.25,2.1);
\draw[very thick] (3.25,2.1) to [out=0, in=110] (4.15,1);
\node at (4.39,1.25) {\scriptsize $C_{k}$};

\end{tikzpicture}
\caption{Case $(i)$.}\label{fig:unicuspidal1}
\end{figure}

In case~$(ii)$, we also blow up the the intersection point of $C_k$ and $E_k$ and obtain the diagram in Figure~$\ref{fig:unicuspidal2}$. Again, by the symmetry of the configuration, there exists a morphism to $\p^2$ that contracts $C_{k+1},E_k,\ldots,E_1$.

\begin{figure}[ht]
\centering
\begin{tikzpicture}[scale=0.75]

\draw[dashed] (9,0) -- (10.5,2.5);

\draw (10,2.1) -- (13,2.1);
\node at (11.5,2.35) {\scriptsize $E_{k}$};

\draw[very thick] (12.5,2.5) -- (14,0);
\node at (13.9,1.2){\scriptsize $E_{k+1}$};

\draw (12.82,2.5)[very thick] to [out=-100, in=25] (11,0);
\node at (12.15,0.2) {\scriptsize $C_{k+1}$};

\end{tikzpicture}
\caption{Case $(ii)$.}\label{fig:unicuspidal2}
\end{figure}

Finally, in case~$(iii)$, we blow up $m_k$ points, with exceptional curves $E_{k+1},\ldots,E_{k+m_k}$, all proximate to the intersection point between $C_k$ and $E_k$. Then $C_{k+m_k}$ intersects $E_{k+m_k}$ transversely and the self-intersection of $C_{k+m_k}$ is $\geq -1$. We can thus continue to blow up points until we have a $(-1)$-tower resolution of $C$, where $C_{n-1}$ intersects $E_{n-1}$ transversely. We then blow up any point on $E_{n-1}$ that does not lie on $C_{n-1}$ or any other exceptional curve. We then obtain the configuration in Figure~$\ref{fig:unicuspidal3}$. By the symmetry of this configuration, there exists a morphism to $\p^2$ by contracting the curves $C_n,E_{n-1},\ldots,E_1$.

\begin{figure}[ht]
\centering
\begin{tikzpicture}[scale=0.75]

\draw[dashed] (9,0) -- (10.5,2.5);

\draw (10,2.1) -- (13,2.1);
\node at (11,2.35) {\scriptsize $E_{n-1}$};

\draw[very thick] (11.5,2.5) -- (13,0);
\node at (12.25,0.7){\scriptsize $E_{n}$};

\draw (12.6,2.5)[very thick] to [out=-90, in=160] (13.9,0);
\node at (13.45,0.7) {\scriptsize $C_{n}$};

\end{tikzpicture}
\caption{Case $(iii)$.}\label{fig:unicuspidal3}
\end{figure}
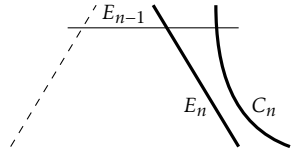
\end{proof}

\begin{remark}
{\rm Lemma~$\ref{lemma:unicuspidalcontractible}$ allows us to determine for a unicuspidal curve $C \subset \p^2$, whether  there exists an open embedding $\p^2 \setminus C \hookrightarrow \p^2$ that does not extend to an automorphism of $\p^2$, simply by looking at the multiplicity sequence of $C$.}
\end{remark}

\begin{corollary}\label{cor:unicuspidal8}
Let $C \subset \p^2$ be an irreducible unicuspidal curve of degree $\leq 8$ and let $\varphi \colon \p^2 \setminus C \to \p^2 \setminus D$ be an isomorphism, where $D \subset \p^2$ is some curve. Then $C$ and $D$ are projectively equivalent.
\end{corollary}

\begin{proof}
If $\varphi$ extends to an automorphism of $\p^2$, the claim is trivial. If not, then $C$ has one of the multiplicity sequences in Table~$\ref{table:sequences}$, by Corollary~$\ref{Cor:table}$. In the case of the multiplicity sequence $(2_{(6)})$, the claim follows from Corollary~$\ref{cor:onequintic}$. For the multiplicity sequences $(3,2_{(7)})$, $(3_{(4)}$, $2_{(3)})$, $(4,3_{(5)})$, $(4,3_{(4)},2_{(3)})$ the claim follows from Lemma~$\ref{lemma:notunicuspidal}$ and for $(3_{(7)})$ from Lemma~$\ref{lemma:unicuspidalcontractible}$, since $8^2 - 7\cdot 3^2 - 3 = -2 < -1$. In all other cases, there exists a very tangent line through the proper singular point of $C$ by Lemma~$\ref{lemma:first2}$. Then the claim follows from Theorem~$\ref{Thm:Yoshihara}$.
\end{proof}

\subsection{Some special multiplicity sequences}

In this section we present some extension results for isomorphisms between curves that are not unicuspidal and have a multiplicity sequence of a special form. Together with the previous results this will lead to the proof of Theorem~$\ref{Thm:degree8}$.

\begin{proposition}\label{Prop:nojump}
Let $C$ be an irreducible rational curve of degree $d \geq 4$ and multiplicity sequence $(m_{(k)})$, where $m \geq 2$ and $k \geq 1$, and let $\varphi \colon \p^2 \setminus C \hookrightarrow \p^2$ be an open embedding that does not extend to an automorphism of $\p^2$. If $C$ is not unicuspidal, then $C \setminus \Sing(C)$ is isomorphic to $\A^1 \setminus \{0\}$ and $C$ has either degree $8$ with multiplicity sequence $(3_{(7)})$ or degree $16$ with multiplicity sequence $(6_{(7)})$.
\end{proposition}

\begin{proof}
Suppose that $C$ is not unicuspidal. By Lemma~$\ref{Lem:tower}$, there exists a $(-1)$-tower resolution of $C$ given by $\pi \colon X = X_n \xrightarrow{\pi_n} \ldots \xrightarrow{\pi_2} X_1 \xrightarrow{\pi_1} X_0 = \p^2$ with base-points $p_1,\ldots,p_n$ and exceptional curves $E_1,\ldots,E_n$, and a $(-1)$-tower resolution $\eta \colon X \to \p^2$ of some curve $D \subset \p^2$ such that $\varphi \circ \pi = \eta$. Then $E_1 \cup \ldots \cup E_{n-1} \cup C_n$ is the exceptional locus of $\eta$, being the support of an SNC-divisor that has a tree structure. The composition $\pi_k \circ \ldots \circ \pi_1$ is the minimal resolution of singularities of $C$. By Lemma~$\ref{Lem:proximate mult}$ we obtain that in the surface $X_k$, we have the intersection numbers $C_k \cdot E_i = 0$, for $i = 1,\ldots,k-1$, and $C_k \cdot E_k = m$. Since $E_1 \cup \ldots \cup E_{k-1} \cup C_k$ is not connected, we know that $n > k$, hence more points are blown up to obtain the $(-1)$-tower resolution $\pi$. Since we assumed $C$ not to be unicuspidal, the curves $C_k$ and $E_k$ intersect in at least two points in $X_k$. If $C_k$ and $E_k$ intersect in at least $3$ points, then it follows that $C_n$ and $E_k$ intersect in at least two points in $X$, which is not possible by the tree structure of $E_1 \cup \ldots \cup E_{n-1} \cup C_n$. It thus follows that $C_k$ and $E_k$ intersect in exactly two points and hence $C \setminus \Sing(C) = C \setminus \{p_1\} \simeq \A^1 \setminus \{0\}$. Moreover, it follows (again by the tree structure) that $C_n$ intersects $E_k$ transversely in one point in the surface $X$, thus $C_k$ intersects $E_k$ in one point transversely and in the point $p_{k+1}$ with intersection multiplicity $m-1$ in $X_k$.  The configuration of curves is illustrated in the diagram on the left in Figure~$\ref{fig:mmm1}$, where the dashed lines represent chains of $(-2)$-curves. Again by the fact that $C_n$ and $E_k$ intersect only in one point, the base-points of the blow-ups $\pi_{k+1},\ldots,\pi_{k+m-1}$ are proximate to $p_{k+1}$ (i.e.\ all lie on $E_k$) and we obtain  $E_k^2 = -m$ in $X_{k+m-1}$, as illustrated in the diagram on the right of Figure~$\ref{fig:mmm1}$. We denote the self-intersection of $C_{k+m-1}$ by $\delta$ and thus have $\delta = d^2 - km^2 - (m-1)$. Since $\pi$ is a $(-1)$-tower resolution of $C$ we have $\delta \geq -1$.

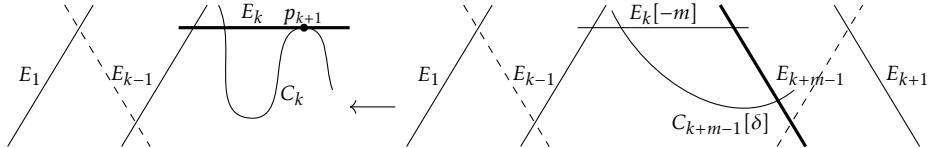
\begin{figure}[ht]
\centering
\begin{tikzpicture}[scale=0.75]
\draw (0,0) -- (1.5,2.5);
\node at (0.4,1.25) {\scriptsize $E_{1}$};
\draw[dashed] (1,2.5) --(2.5,0);
\draw (2,0) -- (3.5,2.5);
\node at (2.2,1.25) {\scriptsize $E_{k-1}$};

\draw[very thick] (3,2.1) -- (6,2.1);
\node at (4.3,2.35) {\scriptsize $E_{k}$};

\fill (5.2,2.1) circle (2pt);
\node at (5.2,2.3) {${\scriptstyle p_{k+1}}$};

\draw (3.7,2.5) to [out=-65, in=180] (4.3,0.5);
\draw (4.3,0.5) to [out=0, in=180] (5.2,2.1);
\draw (5.2,2.1) to [out=0, in=110] (5.7,1);

\node at (5,0.9) {\scriptsize $C_{k}$};

\draw[<-] (6,0.7) -- (6.8,0.7);

\draw (7,0) -- (8.5,2.5);
\node at (7.4,1.25) {\scriptsize $E_{1}$};
\draw[dashed] (8,2.5) -- (9.5,0);
\draw (9,0) -- (10.5,2.5);
\node at (9.25,1.25) {\scriptsize $E_{k-1}$};

\draw (10,2.1) -- (13,2.1);
\node at (11.5,2.35) {\scriptsize $E_{k}[-m]$};

\draw[very thick] (12.5,2.5) -- (14,0);
\draw[dashed] (13.5,0) -- (15,2.5);
\draw (14.5,2.5) -- (16,0);

\node at (14.1,1.25) {\scriptsize $E_{k+m-1}$};
\node at (15.8,1.25) {\scriptsize $E_{k+1}$};

\draw (10.6,2.4) to [out=-65, in=220] (13.8,1);
\node at (12.5,0.4) {\scriptsize $C_{k+m-1}[\delta]$};

\end{tikzpicture}
\caption{Minimal SNC-resolution of $C$.}\label{fig:mmm1}
\end{figure}

To simplify the later cases we first prove the following.

\begin{claim*}[1]
If $k =1$, we reach a contradiction.
\end{claim*}

\begin{subproof}[Proof of Claim~$(1)$.]
Since the degree of $C$ is $d \geq 4$, we obtain $m=d - 1 \geq 3$ by the rationality of $C$ and the genus-degree formula and hence we have $\delta = d+1 \geq 5$. Since $C_n$ has self-intersection $-1$, the base-point $p_{i+1}$ is the unique intersection point between $C_i$ and $E_i$ in $X_i$ for $i = m,\ldots,m+1+\delta$, as shown in Figure~$\ref{fig:(m)}$.

\begin{figure}[ht]
\centering
\begin{tikzpicture}[scale=0.75]

\draw (3,2.1)--(6,2.1);
\node at (4.6,2.4) {\scriptsize $ E_{1}[-m]$};

\draw (5.5,2.5)--(7,0);
\node at (6.68,1.21){\scriptsize $ E_{m}$};
\draw (6.5,0)--(8,2.5);
\draw[dashed] (7.5,2.5)--(9,0);
\node at (9.52,1.21){\scriptsize $ E_{2}$};
\draw (8.5,0)--(10,2.5);
\node at (7.81,1.21){\scriptsize $ E_{m-1}$};

\draw (6.7,1)--(5.5,-1);
\node at (6.52,-0.3){\scriptsize $ E_{m+1}$};
\draw[dashed] (6,-1)--(4.8,1);
\draw[very thick] (5.2,1)--(4,-1);
\node at (3.32,-0.65){\scriptsize $ E_{m+1+\delta}$};

\draw[very thick] (3.6,2.4) to [out=-80, in=160] (4.7,-0.6);
\node at (2.9,0.8) {\scriptsize $C_{m+1+\delta}$};

\end{tikzpicture}
\caption{Case $(m)$.}\label{fig:(m)}
\end{figure}

\noindent If $\pi$ has another base-point in $X_{m+1+\delta}$, then it lies on $E_{m+1+\delta} \setminus C_{m+1+\delta}$. We know that $\delta \geq 5$ and thus the curves $E_{m}$ and $E_{m+1}$ have self-intersection $-2$ in $X$. Moreover, the curves $E_1,\ldots,E_{n-1},C_n$ have a tree structure in $X$, thus $C_n$ and $E_m$ are uniquely connected via $E_1$ in this tree. The map $\eta$ successively contracts the curves in this tree, starting with $C_n$. The chain of curves that connects $C_n$ to $E_{m-1}$, respectively $E_{m+1}$, contains $E_{m}$, thus $\eta$ contracts $E_1$ before $E_{m-1}$ and $E_{m+1}$. But this is not possible since after contracting $E_{m}$, the images of both $E_{m-1}$ and $E_{m+1}$ have self-intersection $-1$. We thus get a contradiction and conclude that $k \geq 2$.
\end{subproof}

In the sequel, we separately study the cases $\delta \geq 1$, $\delta = 0$, and $\delta = -1$.

\begin{claim*}[2]
If $\delta \geq 1$, we reach a contradiction.
\end{claim*}

\begin{subproof}[Proof of Claim~$(2)$.]
Since $\pi$ is a $(-1)$-tower resolution of $C$ the base-point $p_{i+1}$ is the unique intersection point between $C_i$ and $E_i$ in $X_i$ for $i = k+m-1,\ldots,k+m+\delta$ (see Figure~$\ref{fig:mmm5}$).

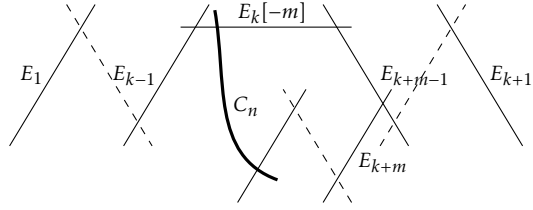
\begin{figure}[ht]
\centering
\begin{tikzpicture}[scale=0.75]
\draw (0,0)--(1.5,2.5);
\node at (0.4,1.25) {\scriptsize $E_{1}$};
\draw[dashed] (1,2.5)--(2.5,0);
\draw (2,0)--(3.5,2.5);
\node at (2.2,1.25) {\scriptsize $E_{k-1}$};

\draw (3,2.1)--(6,2.1);
\node at (4.6,2.35) {\scriptsize $E_{k}[-m]$};

\draw (5.5,2.5)--(7,0);
\node at (7.12,1.25){\scriptsize $E_{k+m-1}$};
\draw[dashed] (6.5,0)--(8,2.5);
\draw (7.5,2.5)--(9,0);
\node at (8.82,1.25){\scriptsize $E_{k+1}$};

\draw (6.7,1)--(5.5,-1);
\node at (6.55,-0.3){\scriptsize $E_{k+m}$};
\draw[dashed] (6,-1)--(4.8,1);
\draw (5.2,1)--(4,-1);

\draw[very thick] (3.6,2.4) to [out=-80, in=160] (4.7,-0.6);
\node at (4.15,0.7) {\scriptsize $C_n$};

\end{tikzpicture}
\caption{Case $\delta \geq 1$.}\label{fig:mmm5}
\end{figure}

\noindent Since $\delta \geq 1$, it follows that the curve $E_{k+m-1}$ has self-intersection $-2$ in $X$. Moreover, we know that $k \geq 2$ (i.e. there is a $(-2)$-curve $E_{k-1}$ as pictured in Figure~$\ref{fig:mmm5}$). The map $\eta$ contracts the curves $E_{k-1}$ and $E_{k+m-1}$ after $E_{k}$, since in the tree of curves $E_1,\ldots,E_{n-1},C_n$ the curves $C_n$ and $E_{k-1}$, respectively $E_{k+m-1}$, are connected via $E_k$. But after contracting $E_k$, the self-intersections of the images of $E_{k-1}$ and $E_{k+m-1}$ are both $-1$, which is not possible. We thus conclude that $\delta \geq 1$ is not possible.
\end{subproof}

\begin{claim*}[3]
If $\delta = 0$, we reach a contradiction.
\end{claim*}

\begin{subproof}[Proof of Claim~$(3)$.]
Since $\delta = 0$, the base-point of the next blow-up $\pi_{k+m}$ is the unique intersection point between $C_{k+m-1}$ and $E_{k+m-1}$ and we obtain the configuration of curves in the left part of Figure~\ref{fig:mmm3}.
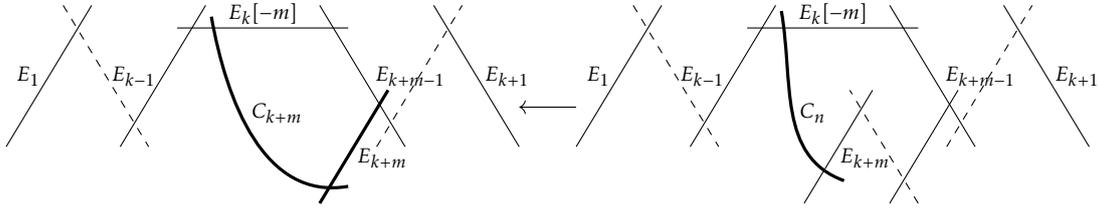
\begin{figure}[ht]
\centering
\begin{tikzpicture}[scale=0.75]
\draw (0,0)--(1.5,2.5);
\node at (0.4,1.25) {\scriptsize $E_{1}$};
\draw[dashed] (1,2.5)--(2.5,0);
\draw (2,0)--(3.5,2.5);
\node at (2.25,1.25) {\scriptsize $E_{k-1}$};

\draw (3,2.1)--(6,2.1);
\node at (4.5,2.35) {\scriptsize $E_{k}[-m]$};

\draw (5.5,2.5)--(7,0);
\node at (7.1,1.25){\scriptsize $E_{k+m-1}$};
\draw[dashed] (6.5,0)--(8,2.5);
\draw (7.5,2.5)--(9,0);
\node at (8.8,1.25){\scriptsize $E_{k+1}$};

\draw[very thick] (6.7,1)--(5.5,-1);
\node at (6.6,-0.2){\scriptsize $E_{k+m}$};

\draw[very thick] (3.6,2.3) to [out=-80, in=190] (6,-0.7);
\node at (4.75,0.6) {\scriptsize $C_{k+m}$};

\draw[<-] (9,0.7)--(10,0.7);

\draw (10,0)--(11.5,2.5);
\node at (10.4,1.25) {\scriptsize $E_{1}$};
\draw[dashed] (11,2.5)--(12.5,0);
\draw (12,0)--(13.5,2.5);
\node at (12.22,1.25) {\scriptsize $E_{k-1}$};

\draw (13,2.1)--(16,2.1);
\node at (14.5,2.35) {\scriptsize $E_{k}[-m]$};

\draw (15.5,2.5)--(17,0);
\node at (17.1,1.25){\scriptsize $E_{k+m-1}$};
\draw[dashed] (16.5,0)--(18,2.5);
\draw (17.5,2.5)--(19,0);
\node at (18.8,1.25){\scriptsize $E_{k+1}$};

\draw (16.7,1)--(15.5,-1);
\draw[dashed] (16,-1)--(14.8,1);
\draw (15.2,1)--(14,-1);
\node at (15.07,-0.2){\scriptsize $E_{k+m}$};

\draw[very thick] (13.6,2.4) to [out=-80, in=160] (14.7,-0.6);
\node at (14.15,0.6) {\scriptsize $C_n$};

\end{tikzpicture}
\caption{Case $\delta = 0$.}\label{fig:mmm3}
\end{figure}

In the surface $X$, the curves $E_{k+m},\ldots,E_n$ all lie in a chain (not necessarily in this order) between $C_n$ and $E_{k+m-1}$, i.e.\ the base-points always lie on the intersection points of the chain between $C_n$ and $E_{k+m-1}$, as otherwise there would be a loop in the configuration of the curves $E_1,\ldots,E_{n-1},C_n$ in $X$ (see the right part of Figure~$\ref{fig:mmm3}$). Moreover, $E_{k+m}$ intersects $C_n$ in this chain. The map $\eta$ first contracts $C_n$ and after this contraction the image of $E_k$ has self-intersection $-m+1$. It follows that  in the chain of curves between $C_n$ and $E_{k+m-1}$, after $C_n$ there is a chain of $(-2)$-curves of length $m-2$, such that the image of $E_k$ is $-1$, after this chain is contracted. This means that the base-points $p_{i+1}$ for $i=k+m,\ldots,k+m+(m-3)$ all lie on $E_{k+m-1}$. Denote the next curve in the chain after the $m-2$ $(-2)$-curves by $E$. After $C_n$ and the chain of $m-2$ $(-2)$-curves are contracted, the images of $E_k$ and $E$ intersect. Moreover, the self-intersection of $E_k$ is $-1$ in this surface and thus $\eta$ then contracts $E_k,\ldots,E_1$. Since we assume $k \geq 2$, it follows that the image of $E$ is tangent to $E_{k+m-1}$. But this means that $E$ is not contracted by $\eta$ and must in fact be $E_n = E_{k+m+(m-2)}$. Since the base-points $p_{k+m+1},\ldots,p_{k+m+(m-2)}$ all lie on $E_{k+m-1}$, the self-intersection of $E_{k+m-1}$ in $X$ is $-m$. We observe that after $\eta$ contracts $C_n$ and the chain $E_k,\ldots,E_1$ the image of $E_{k+m-1}$ has self-intersection $-m+k$, which has to be equal to $-1$, and thus $k=m-1$. From the condition $\delta =0$ and the genus-degree formula we obtain the equations
\begin{align*}
0 &= d^2 - (m-1)m^2 - m + 1, \\
0 &= d^2 - 3d + 2 - (m-1)m^2.
\end{align*}
Subtracting the second equation from the first then yields $3d - m^2 -1 = 0$. We can then substitute $d = \frac{m^2+1}{3}$ in the first equation and obtain
$$0 = \frac{(m^2+1)^2}{9} - (m-1)m^2 - (m-1) = (m^2+1)\left(\frac{m^2+1}{9} - m + 1\right),$$
which has no integer solutions in $m$. We conclude that $\delta = 0$ is not possible.
\end{subproof}

\begin{claim*}[4]
If $\delta = -1$, then $C$ is of degree $8$ or $16$ with multiplicity sequence $(3_{(7)})$ or $(6_{(7)})$ respectively.
\end{claim*}

\begin{subproof}[Proof of Claim~$(4)$.]
We already have a $(-1)$-tower resolution of $C$ in this case (see Figure~$\ref{fig:mmm2}$). We observe that blowing up the intersection point between $E_k$ and $E_{k+m-1}$ yields a symmetric diagram and thus there exists a morphism $X \to \p^2$ whose contracted locus is exactly $E_1\cup \ldots \cup E_{k+m-1} \cup C_{k+m}$. 

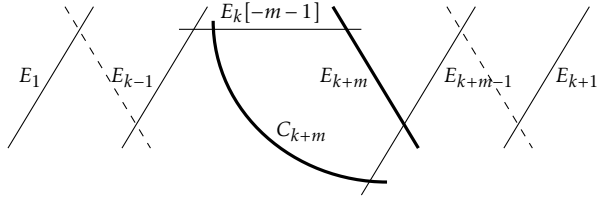
\begin{figure}[ht]
\centering
\begin{tikzpicture}[scale=0.75]
\draw (0,0) -- (1.5,2.5);
\node at (0.4,1.25) {\scriptsize $E_{1}$};
\draw[dashed] (1,2.5) -- (2.5,0);
\draw (2,0) -- (3.5,2.5);
\node at (2.2,1.25) {\scriptsize $E_{k-1}$};

\draw (3,2.1) -- (6.2,2.1);
\node at (4.6,2.4) {\scriptsize $E_{k}[-m-1]$};

\draw[very thick] (5.7,2.5) -- (7.2,0);
\node at (5.9,1.25) {\scriptsize $E_{k+m}$};
\draw (6.2,-0.83) -- (8.2,2.5);
\node at (8.27,1.25){\scriptsize $E_{k+m-1}$};
\draw[dashed] (7.7,2.5) -- (9.2,0);
\node at (10,1.25){\scriptsize $E_{k+1}$};
\draw (8.7,0) -- (10.2,2.5);

\draw[very thick] (3.6,2.25) to [out=-90, in=180] (6.65,-0.6);
\node at (5.15,0.3) {\scriptsize $C_{k+m}$};

\end{tikzpicture}
\caption{Case $\delta = -1$.}\label{fig:mmm2}
\end{figure}

\noindent The condition $\delta = -1$ and the genus-degree formula give us the following equations for the values of $d,m,k$:
\begin{align*}
0 &= d^2 - km^2 - m + 2,\\
0 &= d^2 - 3d + 2 - km^2 - km.
\end{align*}
We see from the first equation that any integer factor of $d$ and $m$ also divides $2$. Hence the greatest common divisor of $d$ and $m$ is $1$ or $2$. Subtracting the equations yields $3d - m - km = 0$, from which we conclude that $m$ divides $3d$. It thus follows that $m$ divides $6$. Next, we replace $k = \frac{3d - m}{m}$ in the first equation above and get $d^2 - 3dm - m^2 - m + 2 = 0$. We then check for natural solutions in $d$ for $m \in \{2,3,6\}$ and find $(d,m) =(8,3)$ or $(16,6)$ (both with $k=7$) as the only possibilities.
\end{subproof}
This concludes the proof of Proposition~$\ref{Prop:nojump}$.
\end{proof}

\begin{remark}
{\rm The assumption that $d = \deg(C) \geq 4$ in Proposition~$\ref{Prop:nojump}$ is necessary since the the complement of a nodal cubic has non-extendable automorphisms (see Remark~$\ref{remark:cubic}$).}
\end{remark}

\begin{corollary}\label{cor:equalcase}
Let $C \subset \p^2$ be an irreducible rational curve with one of the multiplicity sequences $(2_{(3)})$, $(3)$, $(4)$, $(2_{(6)})$, $(5)$, $(6)$, or $(7)$. If $C$ is not unicuspidal, then any open embedding $\p^2 \setminus C \hookrightarrow \p^2$ extends to an automorphism of $\p^2$.
\end{corollary}

\begin{proof}
This is a direct consequence of Proposition~$\ref{Prop:nojump}$.
\end{proof}

\begin{proposition}\label{Prop:1jump}
Let $C \subset \p^2$ be an irreducible rational curve of degree $d$ and multiplicity sequence $\left(m_{(k)},(m-1)_{(l)} \right)$, where $m \geq 3$ and $k,l \geq 1$ and let $\varphi \colon \p^2 \setminus C \hookrightarrow \p^2$ be an open embedding that does not extend to an automorphism of $\p^2$. Then either $C$ is unicuspidal or of degree $6$ with multiplicity sequence $(3,2_{(7)})$ or of degree $13$ with multiplicity sequence $(5_{(6)},4)$.
\end{proposition}

\begin{proof}
We suppose that $C$ is not unicuspidal. Since $\varphi$ does not extend to an automorphism of $\p^2$, it follows by Lemma~$\ref{Lem:tower}$ that there exists a $(-1)$-tower resolution $\pi \colon X = X_n \xrightarrow{\pi_n} \ldots \xrightarrow{\pi_2} X_1 \xrightarrow{\pi_1} X_0 = \p^2$ of $C$ with base-points $p_1,\ldots,p_n$ and exceptional curves $E_1,\ldots,E_n$, and a $(-1)$-tower resolution $\eta \colon X \to \p^2$ of some curve $D \subset \p^2$ such that $\varphi \circ \pi = \eta$. Then $E_1 \cup \ldots \cup E_{n-1} \cup C_n$ is the exceptional locus of $\eta$, being the support of an SNC-divisor on $X$ that has a tree structure. The composition $\pi_{k+l} \circ \ldots \circ \pi_1$ is the minimal resolution of the singularities of $C$. By Lemma~$\ref{Lem:proximate mult}$ we obtain that in the surface $X_{k+l}$, we have the intersection numbers $C_{k+l} \cdot E_k = 1$ and $C_{k+l} \cdot E_i = 0$ for $i=1,\ldots,k-1$ and $i=k+1,\ldots,k+l-1$.

\setcounter{claim}{0}
\begin{claim*}[1]\label{claim5}
If $k \geq 2$ and $l \geq 2$, we reach a contradiction.
\end{claim*}

\begin{subproof}[Proof of Claim~$(1)$]
By Lemma~$\ref{Lem:proximate mult}$ we have $C_{k+l} \cdot E_{k+l} = m-1$. The configuration is shown in Figure~$\ref{fig:i1}$, where the dashed lines represent chains of $(-2)$-curves.

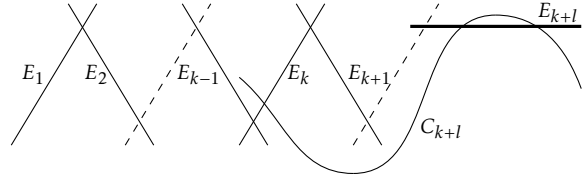
\begin{figure}[ht]
\centering
\begin{tikzpicture}[scale=0.75]
\draw (0,0) -- (1.5,2.5);
\node at (0.4,1.25) {\scriptsize $E_{1}$};
\draw (1,2.5) --(2.5,0);
\node at (1.5,1.25) {\scriptsize $E_{2}$};
\draw[dashed] (2,0) -- (3.5,2.5);
\draw (3,2.5) -- (4.5,0);
\node at (3.3,1.25) {\scriptsize $E_{k-1}$};
\draw (4,0) -- (5.5,2.5);
\node at (5.05,1.25) {\scriptsize $E_{k}$};
\draw (5,2.5) -- (6.5,0);
\node at (6.3,1.25) {\scriptsize $E_{k+1}$};
\draw[dashed] (6,0) -- (7.5,2.5);

\draw[very thick] (7,2.1) -- (10,2.1);
\node at (9.6,2.35) {\scriptsize $E_{k+l}$};

\draw (4,1.2) to [out=-40, in=180] (6,-0.5);
\draw (6,-0.5) to [out=0, in=180] (8.5,2.3);
\draw (8.5,2.3) to [out=0, in=110] (10,1);

\node at (7.55,0.3) {\scriptsize $C_{k+l}$};

\end{tikzpicture}
\caption{Minimal resolution of singularities of $C$.}\label{fig:i1}
\end{figure}

\noindent If $\pi$ has a base-point in $X_{k+l}$, then it lies on the intersection with $C_{k+l}$ and $E_{k+l}$, otherwise there would be a loop formed by $E_k,\ldots,E_{k+l}$ and $C_n$ in $X_n$, which is not possible by the tree structure of the curves $E_1,\ldots,E_{n-1},C_n$. Since $E_{k+l}$ does not intersect the $(-2)$-curves $E_{k-1}$, $E_k$, and $E_{k+1}$, it follows that their self-intersections in $X$ are also $-2$. We observe that the map $\eta$ contracts the curve $E_k$ before $E_{k-1}$ and $E_{k+1}$, since $C_n$ and $E_{k-1}$, respectively $E_{k+1}$, are connected via $E_k$ in the graph of the curves $E_1,\ldots,E_{n-1},C_n$. But after contracting $E_k$, the images of $E_{k-1}$ and $E_{k+1}$ both have self-intersection $-1$, which is a contradiction since $\eta$ is a $(-1)$-tower resolution.
\end{subproof}

In the sequel, we separately look at the more involved cases where $k=1$ or $l=1$ (parts (A) and (B) below).

\bigskip

(A) We assume that $k=1$.

\begin{claim*}[\mathrm{A.}1]\label{claim6}
If $(C_{l+1})^2 = -1$, then $C$ has degree $6$ and multiplicity sequence $(3,2_{{(7)}})$.
\end{claim*}

\begin{subproof}[Proof of Claim~$(\mathrm{A.}1)$]
By Lemma~$\ref{Lem:proximate mult}$ we have $C_{l+1} \cdot E_{l+1} = m-1$. If $C_{l+1}$ has self-intersection $-1$, then by the symmetry of the configuration (see Figure~$\ref{fig:32222222}$), there exists a morphism $X \to \p^2$ whose contracted locus is $E_1 \cup \ldots \cup E_{l} \cup C_{l+1}$.

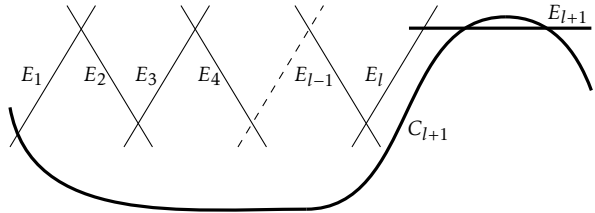
\begin{figure}[ht]
\centering
\begin{tikzpicture}[scale=0.75]

\draw (0,0) -- (1.5,2.5);
\node at (0.4,1.25) {\scriptsize $E_{1}$};
\draw (1,2.5) --(2.5,0);
\node at (1.5,1.25) {\scriptsize $E_{2}$};
\draw (2,0) -- (3.5,2.5);
\node at (2.4,1.25) {\scriptsize $E_{3}$};
\draw (3,2.5) -- (4.5,0);
\node at (3.5,1.25) {\scriptsize $E_{4}$};
\draw[dashed] (4,0) -- (5.5,2.5);
\draw (5,2.5) -- (6.5,0);
\node at (5.35,1.25) {\scriptsize $E_{l-1}$};
\draw (6,0) -- (7.5,2.5);
\node at (6.4,1.25) {\scriptsize $E_{l}$};

\draw[very thick] (7,2.1) -- (10.2,2.1);
\node at (9.8,2.35) {\scriptsize $E_{l+1}$};

\draw[very thick] (0,0.7) to [out=-80, in=180] (5.2,-1.1);
\draw[very thick] (5.2,-1.1) to [out=0, in=180] (8.7,2.3);
\draw[very thick] (8.7,2.3) to [out=0, in=110] (10.2,1);

\node at (7.35,0.3) {\scriptsize $C_{l+1}$};

\end{tikzpicture}
\caption{Case $k=1$, $(C_{l+1})^2=-1$.}\label{fig:32222222}
\end{figure}

From $(C_{l+1})^2 = -1$ and the genus-degree formula we obtain the following two identities:
\begin{align*}
0 &= d^2 - m^2 - l(m-1)^2 + 1, \\
0 &= d^2 -3d + 2 - m(m-1) - l(m-1)(m-2).
\end{align*}
Subtracting the second equation from the first yields $3d - 1 - m - l(m-1) = 0$. We then substitute the equality $l(m-1) = 3d - 1 - m$ in the first equation and obtain $d^2 = 3d(m-1)$ and thus $d = 3(m-1)$. Finally, we get
$$0 = 3d - 1 - m - l(m-1) = (9-l)(m-1) - (m+1)$$
and for positive integer values this equation is only satisfied with $m=2$ and $l=7$ since $1<9-l=\frac{m+1}{m-1}<2$, for $m\geq 3$. This leads to the multiplicity sequence $(3,2_{(7)})$ in degree $6$. The corresponding resolution diagram is shown in Figure~$\ref{fig:32222222}$, where the dashed line represents one $(-2)$-curve.
\end{subproof}

We suppose from now on that we are not in the case of the multiplicity sequence $(3,2_{(7)})$. We then have $(C_{l+1})^2 > -1$. This implies that $\pi$ has a base-point in the intersection of $C_{l+1}$ with $E_{l+1}$. In fact, the curves $C_n$ and $E_{l+1}$ do not intersect in $X$, otherwise there would be a loop in the graph of the curves $E_1,\ldots,E_{n-1},C_n$. Thus $C_{l+1}$ and $E_{l+1}$ intersect in a single point in $X_{l+1}$, and hence the intersection multiplicity is $m-1$. We have thus the configuration of curves shown in the left part of Figure~$\ref{fig:i2}$.

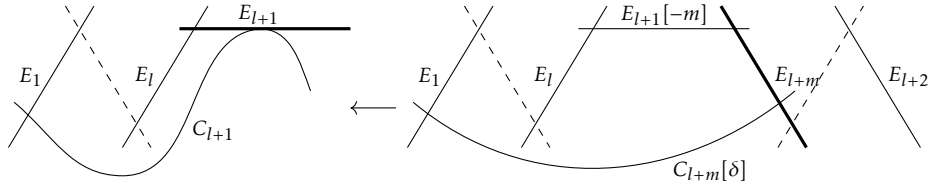
\begin{figure}[ht]
\centering
\begin{tikzpicture}[scale=0.75]
\draw (0,0) -- (1.5,2.5);
\node at (0.4,1.25) {\scriptsize $E_{1}$};
\draw[dashed] (1,2.5) --(2.5,0);
\draw (2,0) -- (3.5,2.5);
\node at (2.4,1.25) {\scriptsize $E_{l}$};

\draw[very thick] (3,2.1) -- (6,2.1);
\node at (4.4,2.35) {\scriptsize $E_{l+1}$};

\draw (0.1,0.8) to [out=-40, in=180] (2,-0.5);
\draw (2,-0.5) to [out=0, in=180] (4.5,2.1);
\draw (4.4,2.1) to [out=0, in=110] (5.3,1);

\node at (3.6,0.3) {\scriptsize $C_{l+1}$};

\draw[<-] (6,0.7) -- (6.8,0.7);

\draw (7,0) -- (8.5,2.5);
\node at (7.4,1.25) {\scriptsize $E_{1}$};
\draw[dashed] (8,2.5) -- (9.5,0);
\draw (9,0) -- (10.5,2.5);
\node at (9.4,1.25) {\scriptsize $E_{l}$};

\draw (10,2.1) -- (13,2.1);
\node at (11.5,2.35) {\scriptsize $E_{l+1}[-m]$};

\draw[very thick] (12.5,2.5) -- (14,0);
\draw[dashed] (13.5,0) -- (15,2.5);
\draw (14.5,2.5) -- (16,0);

\node at (13.85,1.25) {\scriptsize $E_{l+m}$};
\node at (15.8,1.25) {\scriptsize $E_{l+2}$};

\draw (7.1,0.8) to [out=-40, in=220] (13.8,1);
\node at (12.3,-0.4) {\scriptsize $C_{l+m}[\delta]$};

\end{tikzpicture}
\caption{Minimal SNC-resolution of $C$ for $k=1$.}\label{fig:i2}
\end{figure}

Since $C_n$ and $E_{l+1}$ do not intersect in $X$, it follows that the base-point $p_{i+1}$ for $i=l+1,\ldots,l+m-1$ is the unique intersection point between $C_i$ and $E_{i}$, which also lies on $E_{l+1}$. The configuration of curves in $X_{l+m}$ is shown in the right part of Figure~$\ref{fig:i2}$. We denote the self-intersection number of $C_{l+m}$ by $\delta$ and this number is equal to $d^2 - m^2 - l(m-1)^2 - (m-1)$. Since $\pi$ is a $(-1)$-tower resolution we have that $\delta \geq -1$.

\begin{claim*}[\mathrm{A.}2]\label{claim7}
If $\delta = -1$, we reach a contradiction.
\end{claim*}

\begin{subproof}[Proof of Claim~$(\mathrm{A}.2)$]
From $\delta = -1$ and the genus-degree formula we obtain
\begin{align*}
0 &= d^2 - m^2 - l(m-1)^2 - m + 2, \\
0 &= d^2 -3d + 2 - m(m-1) - l(m-1)(m-2).
\end{align*}
Subtracting the second equation from the first yields $3d - 2m - l(m-1) = 0$. We then replace $l = \frac{3d - 2m}{m-1}$ in the first equation and obtain the identity
$$0 = d^2 - m^2 - (3d-2m)(m-1) - m + 2 = d^2 - (m-1)(3d-m+2).$$
It follows that $m-1$ divides $d^2$. Let $p$ be a prime number that divides $m-1$. Then $p$ divides $d^2$ and thus also $d$. From the equality $l(m-1)=3d-2m$ it follows that $p$ divides $2m$. Since $m-1$ and $m$ are coprime, it follows that $p=2$. We can then write $m-1 = 2^r$ for some $r \geq 1$. We observe that $2^r$ divides $d^2$. Moreover, $2^r$ divides $3d-2(2^r+1)$ and thus also $3d-2$. But then $2^r$ divides $d^2-3d+2 = (d-1)(d-2)$. Since $d$ is even, it follows that $2^r$ divides $(d-2)$. Since $2^r$ divides $3d-2 = (d-2) + 2d$, it follows that $2^{r-1}$ divides $d$, but also $d-2$, and thus $r$ must be $1$ or $2$. Using these values for $r$, it is easy to check that the equations above have no integral solutions for $d$. We can thus conclude that $\delta \neq -1$.
\end{subproof}

\begin{claim*}[\mathrm{A.}3]\label{claim8}
If $\delta = 0$, we reach a contradiction.
\end{claim*}

\begin{subproof}[Proof of Claim~$(\mathrm{A}.3)$]
The curves $C_{l+m}$ and $E_{l+m}$ have a unique intersection point, hence this is the base-point $p_{l+m+1}$. After blowing up $p_{l+m+1}$ we obtain a $(-1)$-tower resolution of $C$ (see the left part of Figure~$\ref{fig:i3}$).
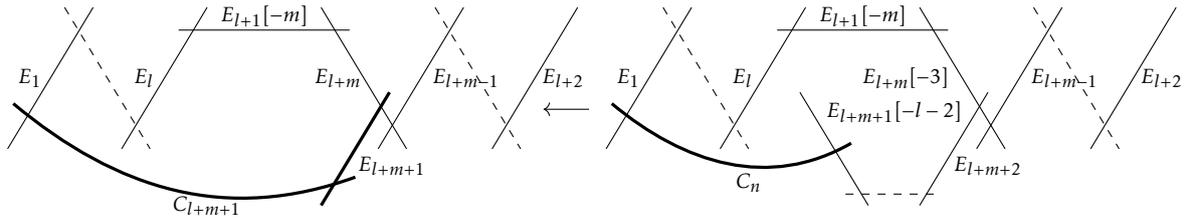
\begin{figure}[ht]
\centering
\begin{tikzpicture}[scale=0.75]
\draw (0,0) -- (1.5,2.5);
\node at (0.4,1.25) {\scriptsize $E_{1}$};
\draw[dashed] (1,2.5) -- (2.5,0);
\draw (2,0) -- (3.5,2.5);
\node at (2.4,1.25) {\scriptsize $E_{l}$};

\draw (3,2.1) -- (6,2.1);
\node at (4.5,2.35) {\scriptsize $E_{l+1}[-m]$};

\draw (5.5,2.5) -- (7,0);
\node at (5.8,1.25){\scriptsize $E_{l+m}$};
\draw (6.5,0) -- (8,2.5);
\node at (8.05,1.25){\scriptsize $E_{l+m-1}$};
\draw[dashed] (7.5,2.5) -- (9,0);
\node at (9.75,1.25){\scriptsize $E_{l+2}$};
\draw (8.5,0) -- (10,2.5);

\draw[very thick] (6.7,1) -- (5.5,-1);
\node at (6.75,-0.3){\scriptsize $E_{l+m+1}$};

\draw[very thick] (0.1,0.8) to [out=-40, in=200] (6.1,-0.5);
\node at (3.5,-1) {\scriptsize $C_{l+m+1}$};

\draw[<-] (9.4,0.7) -- (10.2,0.7);

\draw (10.5,0) -- (12,2.5);
\node at (10.9,1.25) {\scriptsize $E_{1}$};
\draw[dashed] (11.5,2.5) -- (13,0);
\draw (12.5,0) -- (14,2.5);
\node at (12.9,1.25) {\scriptsize $E_{l}$};

\draw (13.5,2.1) -- (16.5,2.1);
\node at (15,2.35) {\scriptsize $E_{l+1}[-m]$};

\draw (16,2.5) -- (17.5,0);
\node at (15.75,1.25){\scriptsize $E_{l+m}[-3]$};
\draw (17,0) -- (18.5,2.5);
\node at (18.55,1.25){\scriptsize $E_{l+m-1}$};
\draw[dashed] (18,2.5) -- (19.5,0);
\node at (20.25,1.25){\scriptsize $E_{l+2}$};
\draw (19,0) -- (20.5,2.5);

\draw (17.2,1) -- (16,-1);
\node at (17.2,-0.3){\scriptsize $E_{l+m+2}$};
\draw[dashed] (14.7,-0.8) -- (16.3,-0.8);
\draw (15.1,-1)--(13.9,1);

\node at (15.55,0.65){\scriptsize $E_{l+m+1}[-l-2]$};

\draw[very thick] (10.6,0.8) to [out=-40, in=210] (14.8,0.1);
\node at (13,-0.6) {\scriptsize $C_n$};

\end{tikzpicture}
\caption{Case $k=1$, $\delta = 0$.}\label{fig:i3}
\end{figure}
In the surface $X$, the curves $E_{l+m+1},\ldots,E_n$ all lie in a chain (not necessarily in this order) between $C_n$ and $E_{l+m}$, otherwise there would be a loop in the configuration of the curves $E_1,\ldots,E_{n-1},C_n$. The curve $E_{l+m+1}$ intersects $C_n$ in this chain. The map $\eta$ contracts first $C_n$ and then the chain $E_1,\ldots,E_l$. The self-intersection of the image of $E_{l+m+1}$ after these contractions increases by $l+1$. Since $E_{l+1}$ is not a $(-1)$-curve after these contractions (as $m \geq 3$), it follows that $E_{l+m+1}$ is a $(-1)$-curve in this surface. This implies that in $X$ the curve $E_{l+m+1}$ has self-intersection $-(l+2)$. This means that the base-points $p_{l+m+2},\ldots,p_{l+m+(l+2)}$
 must lie on the strict transform of $E_{l+m+1}$. Assume first that $l \geq 2$. Then $E_{l+m+2}$ has self-intersection $-2$ in $X$. The map $\eta$ contracts $E_{l+m}$ before the $(-2)$-curves $E_{l+m-1}$ and $E_{l+m+2}$, but this is not possible, as the images of both $E_{l+m-1}$ and $E_{l+m+2}$ are $(-1)$-curves, after contracting $E_{l+m}$. Hence $l$ must be $1$ and the multiplicity sequence of $C$ is then $(m,m-1)$. The condition $\delta = 0$ and the genus-degree formula give
\begin{align*}
0 &= d^2 - m^2 - (m-1)^2 - m + 1, \\
0 &= d^2 - 3d + 2 - m(m-1) - (m-1)(m-2).
\end{align*}
Subtracting these equations yields the identity $3d=3m$, which is not possible as $m<d$. We conclude that $\delta \neq 0$.
\end{subproof}

\begin{claim*}[\mathrm{A.}4]\label{claim9}
If $\delta = 1$, we reach a contradiction.
\end{claim*}

\begin{subproof}[Proof of Claim~$(\mathrm{A}.4)$]
Again, the base-point $p_{l+m+1}$ is the intersection point between $E_{l+m}$ and $C_{l+m}$ and $p_{l+m+2}$ is the intersection point between $E_{l+m+1}$ and $C_{l+m+1}$. After blowing up $p_{l+m+1}$ and $p_{l+m+2}$ we have a $(-1)$-tower resolution of $C$ (see the left part of Figure~$\ref{fig:i5}$).

\begin{figure}[ht]
\centering
\begin{tikzpicture}[scale=0.75]
\draw (0,0) -- (1.5,2.5);
\node at (0.4,1.25) {\scriptsize $E_{1}$};
\draw[dashed] (1,2.5) -- (2.5,0);
\draw (2,0) -- (3.5,2.5);
\node at (2.4,1.25) {\scriptsize $E_{l}$};

\draw (3,2.1) -- (6,2.1);
\node at (4.5,2.35) {\scriptsize $E_{l+1}[-m]$};

\draw (5.5,2.5) -- (7,0);
\node at (5.8,1.25){\scriptsize $E_{l+m}$};
\draw (6.5,0) -- (8,2.5);
\node at (8.05,1.25){\scriptsize $E_{l+m-1}$};
\draw[dashed] (7.5,2.5) -- (9,0);
\node at (9.75,1.25){\scriptsize $E_{l+2}$};
\draw (8.5,0) -- (10,2.5);

\draw (6.7,1) -- (5.5,-1);
\node at (6.75,-0.25){\scriptsize $E_{l+m+1}$};
\draw[very thick] (6,-1) -- (4.8,1);
\node at (5.75,0.7){\scriptsize $E_{l+m+2}$};

\draw[very thick] (0.1,0.8) to [out=-40, in=210] (5.6,0.1);
\node at (2.7,-0.7) {\scriptsize $C_{l+m+2}$};

\draw[<-] (9.4,0.7) -- (10.2,0.7);

\draw (10.5,0) -- (12,2.5);
\node at (10.9,1.25) {\scriptsize $E_{1}$};
\draw[dashed] (11.5,2.5) -- (13,0);
\draw (12.5,0) -- (14,2.5);
\node at (12.9,1.25) {\scriptsize $E_{l}$};

\draw (13.5,2.1) -- (16.5,2.1);
\node at (15,2.35) {\scriptsize $E_{l+1}[-m]$};

\draw (16,2.5) -- (17.5,0);
\node at (16.3,1.25){\scriptsize $E_{l+m}$};
\draw (17,0) -- (18.5,2.5);
\node at (18.55,1.25){\scriptsize $E_{l+m-1}$};
\draw[dashed] (18,2.5) -- (19.5,0);
\node at (20.25,1.25){\scriptsize $E_{l+2}$};
\draw (19,0) -- (20.5,2.5);

\draw (17.2,1) -- (16,-1);
\node at (17.6,-0.25){\scriptsize $E_{l+m+1}[-3]$};
\draw[dashed] (14.7,-0.8) -- (16.3,-0.8);
\draw (15.1,-1)--(13.9,1);

\node at (15.5,0.65){\scriptsize $E_{l+m+2}[\leq-3]$};

\draw[very thick] (10.6,0.8) to [out=-40, in=210] (14.8,0.1);
\node at (13,-0.6) {\scriptsize $C_n$};

\end{tikzpicture}
\caption{Case $k=1$, $\delta = 1$.}\label{fig:i5}
\end{figure}
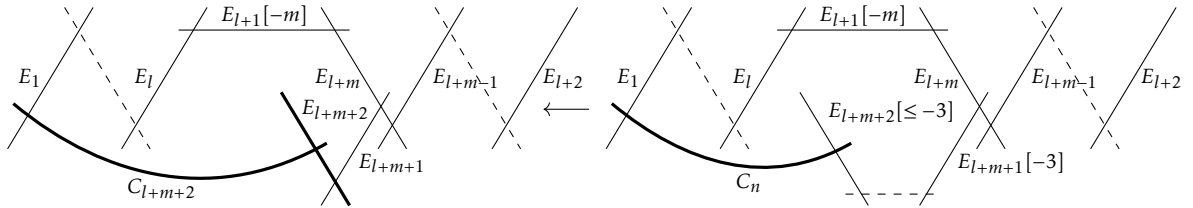

Suppose that this resolution is $\pi$. Then $\eta$ contracts $E_{l+m}$ before the $(-2)$-curves $E_{l+m-1}$ and $E_{l+m+1}$, but this is not possible. Hence $\pi$ has another base-point, which must be the intersection point between $E_{l+m+1}$ and $E_{l+m+2}$, otherwise there would be a loop in the resolution in $X$. Now in $X_{l+m+3}$, the curve $C_{l+m+3}$ intersects the $(-2)$-curves $E_1$ and $E_{l+m+2}$. Thus there is another base-point of  $\pi$, which is the intersection point between $E_{l+m+2}$ and $E_{l+m+3}$. But this implies that $E_{l+m+1}$ has self-intersection $-3$ in $X$ (see the right part of Figure~$\ref{fig:i5}$). We know that $\eta$ contracts $E_{l+m}$ before $E_{l+m-1}$ and $E_{l+m+1}$. After contracting $E_{l+m}$, the self-intersections of the images of $E_{l+m-1}$ and $E_{l+m+1}$ are $-1$ and $-2$ respectively. But then $E_{l+m-1}$ intersects no other $(-2)$-curve, so we have $E_{l+m-1} = E_{l+2}$ and hence $m=3$. The multiplicity sequence of $C$ is thus of the form $(3,2_{(l)})$. Using $\delta = 1$ and the genus degree formula, we obtain
\begin{align*}
0 &= d^2 - 4l - 10, \\
0 &= d^2 - 3d - 2l - 4.
\end{align*}
Subtracting these equations and rearranging terms, we obtain $l = \frac{3d-6}{2}$, which we can substitute in the first equation and get $d^2-6d+2=0$, which has no integer solution in $d$. Thus $\delta = 1$ is not possible.
\end{subproof}

\begin{claim*}[\mathrm{A.}5]\label{claim10}
If $\delta \geq 2$, we reach a contradiction.
\end{claim*}

\begin{subproof}[Proof of Claim~$(\mathrm{A}.5)$]
For $i=l+m,\ldots, l+m+\delta$, the base-point $p_{i+1}$ is then the unique intersection point between $C_i$ and $E_i$. As $\delta \geq 2$, this means that $E_{l+m+1}$ has self-intersection $-2$ in $X$ (see Figure~$\ref{fig:i6}$). But this leads to a contradiction, since $\eta$ contracts $E_{l+m}$ before the $(-2)$-curves $E_{l+m-1}$ and $E_{l+m+1}$, whose images both have self-intersection $-1$, after $E_{l+m}$ is contracted.

\begin{figure}[ht]
\centering
\begin{tikzpicture}[scale=0.75]
\draw (0,0) -- (1.5,2.5);
\node at (0.4,1.25) {\scriptsize $E_{1}$};
\draw[dashed] (1,2.5) -- (2.5,0);
\draw (2,0) -- (3.5,2.5);
\node at (2.4,1.25) {\scriptsize $E_{l}$};

\draw (3,2.1) -- (6,2.1);
\node at (4.5,2.35) {\scriptsize $E_{l+1}[-m]$};

\draw (5.5,2.5) -- (7,0);
\draw (6.5,0) -- (8,2.5);
\draw[dashed] (7.5,2.5) -- (9,0);
\draw (8.5,0) -- (10,2.5);
\node at (5.75,1.25){\scriptsize $E_{l+m}$};
\node at (8.05,1.25){\scriptsize $E_{l+m-1}$};
\node at (9.75,1.25){\scriptsize $E_{l+2}$};
\node at (6.73,-0.26){\scriptsize $E_{l+m+1}$};

\draw (6.7,1) -- (5.5,-1);
\draw[dashed] (4.2,-0.8) -- (5.9,-0.8);
\draw[very thick] (4.5,-1) -- (3.3,1);

\node at (4.5,0.8) {\scriptsize $E_{l+m+\delta+1}$};

\draw[very thick] (0.1,0.8) to [out=-40, in=210] (4.3,0.1);
\node at (2.5,-0.55) {\scriptsize $C_{l+m+\delta+1}$};

\end{tikzpicture}
\caption{Case $k=1$, $\delta \geq 2$.}\label{fig:i6}
\end{figure}
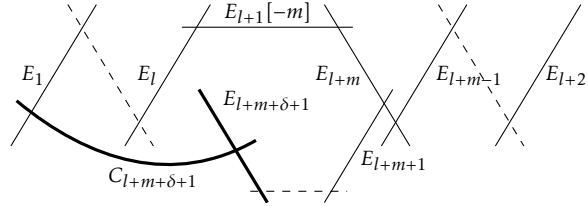
\end{subproof}

This concludes the case $k=1$.

\bigskip

(B) Assume now that $l=1$, as shown in Figure~$\ref{fig:ii0}$. We can also assume that $k \geq 2$, since we have already considered the case $k=1$. If $C_{k+1}$ has self-intersection~$-1$, then by the symmetry of the configuration, there exists a morphism $X \to \p^2$ whose contracted locus is $E_1 \cup \ldots \cup E_{n-1} \cup C_n$.

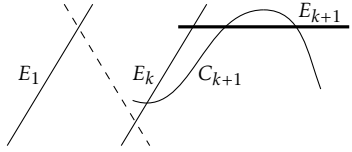
\begin{figure}[ht]
\centering
\begin{tikzpicture}[scale=0.75]
\draw (0,0) -- (1.5,2.5);
\node at (0.4,1.25) {\scriptsize $E_{1}$};
\draw[dashed] (1,2.5) --(2.5,0);
\draw (2,0) -- (3.5,2.5);
\node at (2.4,1.25) {\scriptsize $E_{k}$};

\draw[very thick] (3,2.1) -- (6,2.1);
\node at (5.5,2.35) {\scriptsize $E_{k+1}$};

\draw (2.2,0.8) to [out=-20, in=180] (4.5,2.4);
\draw (4.5,2.4) to [out=0, in=110] (5.5,1);

\node at (3.75,1.25) {\scriptsize $C_{k+1}$};

\end{tikzpicture}
\caption{Minimal resolution of singularities for $l=1$.}\label{fig:ii0}
\end{figure}

\noindent From $(C_{k+1})^2 = -1$ and the genus-degree formula we get the following two identities
\begin{align*}
0 &= d^2 - km^2 - (m-1)^2 + 1, \\
0 &= d^2 -3d + 2 - km(m-1) - (m-1)(m-1).
\end{align*}
Subtracting the second identity from the first yields $3d - 1 - km - (m-1) = 0$. We then substitute  the equality $km = 3d - 1 - (m-1)$ in the first equation and obtain $d^2 = m(3d-2)$. Let $p$ be a prime number that divides $3d-2$ and thus also $d$. But then $p=2$ and hence we can write $3d-2 = 2^r$ for some natural number $r$. It then follows that $m = \frac{(2^r+2)^2}{9\cdot 2^r}$, in particular $2^r$ divides $2^{2r}+4\cdot 2^r + 4$ and thus $r=1$ or $r=2$. If $r=1$, then $d = \frac{4}{3}$, which is absurd. If $r=2$, then $d=2$ and $m=1$, which is excluded by hypothesis.

We thus know that $(C_{k+1})^2 > -1$ and hence $\pi$ has a base-point on $E_{k+1}$ that also lies on $C_{k+1}$. Since $C$ is not unicuspidal, the curves $C_{k+1}$ and $E_{k+1}$ intersect in at least two points.

There are now two possibilities: either $C_{k+1}$ passes through the intersection point between $E_k$ and $E_{k+1}$, or it does not. We will look at these cases separately (parts (i) and (ii) below).

\bigskip

 (i) We suppose that $C_{k+1}$ passes through the intersection point between $E_k$ and $E_{k+1}$. Then this point is the next base-point of $\pi$, since there can be no triple intersections in the tree of the curves $E_1,\ldots,E_{n-1},C_n$ in $X$. Moreover the intersection multiplicity between $C_{k+1}$ and $E_{k+1}$ at $p_{k+2}$ is $m-2$ as $C_n$ and $E_{k+1}$ intersect transversely in $X$, see the configuration on the left in Figure~$\ref{fig:triplecase1}$.

\begin{figure}[ht]
\centering
\begin{tikzpicture}[scale=0.75]
\draw (0,0) -- (1.5,2.5);
\node at (0.4,1.25) {\scriptsize $E_{1}$};
\draw[dashed] (1,2.5) --(2.5,0);
\draw (2,0) -- (3.5,2.5);
\node at (3.1,1.25) {\scriptsize $E_{k}$};

\draw[very thick] (2.8,2.1) -- (6,2.1);
\node at (4.3,2.35) {\scriptsize $E_{k+1}$};

\draw (2.2,1) to [out=70, in=180] (3.3,2.1);
\draw (3.3,2.1) to [out=0, in=180] (4.4,1);
\draw (4.4,1) to [out=0, in=180] (5.4,2.5);

\node at (5.3,1.25) {\scriptsize $C_{k+1}$};

\draw[<-] (6,0.7) -- (6.8,0.7);

\draw (7,0) -- (8.5,2.5);
\node at (7.4,1.25) {\scriptsize $E_{1}$};
\draw[dashed] (8,2.5) -- (9.5,0);
\draw (9,0) -- (10.4,2.5);
\node at (9.73,1.25) {\scriptsize $E_{k}[-3]$};
\draw (10,2.5) -- (11.5,0);
\draw[dashed] (11,0) -- (12.5,2.5);
\node at (11.3,1.25) {\scriptsize $E_{k+2}$};

\draw[very thick] (12,2.1) -- (15,2.1);
\node at (13.47,2.35) {\scriptsize $E_{k+m-1}$};

\node at (16.6,1.25) {\scriptsize $E_{k+1}[-m+1]$};

\draw (14.5,2.5) -- (16,0);

\draw (12.6,2.4) to [out=-90, in=200] (16,0.5);
\node at (14.12,1.25) {\scriptsize $C_{k+m-1}[\delta]$};

\end{tikzpicture}
\caption{Blow-up of $p_{k+2},\ldots,p_{k+m-1}$.}\label{fig:triplecase1}
\end{figure}
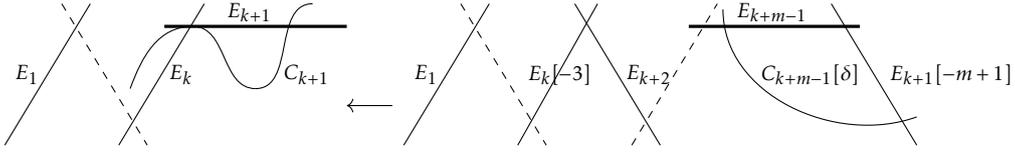

It follows that the base-point $p_{i+1}$ is the intersection point between $E_{k+1}$ and $E_i$ for $i=k+1,\ldots,k+m-2$. We then denote by $\delta$ the self-intersection of $C_{k+m-1}$ in $X_{k+m-1}$, see the configuration on the right in Figure~$\ref{fig:triplecase1}$. We have $\delta = d^2 - km^2 - (m-1)^2 - (m-2)$ and $\delta \geq -1$, since $\pi$ is a $(-1)$-tower resolution.

\begin{claim*}[\mathrm{B.i.}1]\label{claim11}
If $\delta = -1$, we reach a contradiction.
\end{claim*}

\begin{subproof}[Proof of Claim~$(\mathrm{B.i.}1)$]
From $\delta = -1$ and the genus-degree formula we obtain
\begin{align*}
0 &= d^2 - km^2 - (m-1)^2 - m + 3, \\
0 &= d^2 -3d + 2 - km(m-1) - (m-1)(m-2).
\end{align*}
Subtracting these identities yields $3d - km - 2m + 2 = 0$. Thus the greatest common divisor of $d$ and $m$ divides $2$. We then substitute $k = \frac{3d-2m+2}{m}$ in the first equation and obtain $d^2 - 3dm + m^2 - m +2 = 0$. Let $p$ be any prime number that divides $m$. Then $p$ divides $3d+2$ and also $d^2+2$. But then $p$ also divides $d^2-3d=d(d-3)$. Assume that $p$ does not divide $d$, then $p$ divides $d-3$. Then $p$ divides $3d+2 - 3(d-3) = 11$. On the other hand $p$ also divides $(d^2+2) - (d-3)^2 -3(d-3) = 2$ and thus we have a contradiction. It follows that $p$ divides $d$ and hence $p=2$. Dividing the equation above by $2$ yields
$$d\frac{d}{2} - 3d\frac{m}{2} + m\frac{m}{2} - \frac{m}{2} + 1 = 0.$$
We conclude that $\frac{m}{2}$ must be odd. Since $m$ is a power of $2$ it then follows that $m = 2$. We hence obtain the equation $d^2 - 6d + 4 = 0$, which has no integer solution in $d$. We conclude that $\delta = -1$ is not possible.
\end{subproof}

\begin{claim*}[\mathrm{B.i.}2]\label{claim12}
If $\delta = 0$, then $C$ has degree $13$ and multiplicity sequence $(5_{(6)},4)$.
\end{claim*}

\begin{subproof}[Proof of Claim~$(\mathrm{B.i.}2)$]
From $\delta = 0$ and the genus-degree formula we obtain
\begin{align*}
0 &= d^2 - km^2 - (m-1)^2 - m + 2, \\
0 &= d^2 -3d + 2 - km(m-1) - (m-1)(m-2).
\end{align*}
Subtracting these identities yields $3d - km - 2m + 1 = 0$. We thus see that $d$ and $m$ are coprime and that $m$ divides $3d+1$. We substitute $k = \frac{3d - 2m + 1}{m}$ in the first equation and obtain $d^2 - 3dm + m^2 + 1 = 0$. From this we see that $m$ divides $d^2 + 1$. But then $m$ also divides $(d^2+1)-(3d+1)=d(d-3)$. Since $d$ and $m$ are coprime, $m$ divides $d-3$. On the other hand, $m$ also divides $(d^2+1)+(3d+1)=(d+1)(d+2)$. Let $p$ be a prime number that divides $m$. Then $p$ divides $d-3$ and either $d+1$ or $d+2$, but not both since they are coprime. Thus $p$ must be either $2$ or $5$. Assume moreover that $p^2$ divides $m$. Then $p^2$ also divides $d^2+1$ and $3d+1$. Since $p$ divides $d-3$, it follows that $p^2$ divides $(d-3)^2 = d^2 - 6d + 9 = d^2+1 - 2(3d+1) + 10$. But then $p^2$ divides $10$, which is not possible. We conclude that $m \in \{5,10\}$ (since $m \geq 3$). We then check for integer solutions for $d$ in the equation $d^2 - 3dm + m^2 + 1 = 0$ for these values of $m$ and find $(d,m) = (13,5)$ as the only possibility. For a diagram of a resolution of such an isomorphism see Remark~$\ref{remark:5555554}$. We assume from now on that we are not in this case.
\end{subproof}

\begin{claim*}[\mathrm{B.i.}3]\label{claim13}
If $\delta = 1$, we reach a contradiction.
\end{claim*}

\begin{subproof}[Proof of Claim~$(\mathrm{B.i.}3)$]
From $\delta = 1$ and the genus-degree formula we get the equations
\begin{align*}
0 &= d^2 - km^2 - (m-1)^2 - m + 1, \\
0 &= d^2 -3d + 2 - km(m-1) - (m-1)(m-1).
\end{align*}
Subtracting these identities yields $3d - km - 2m = 0$. We then substitute $k = \frac{3d - 2m}{m}$ in the first equation and obtain $d^2 = m(3d+m+1)$. Let $p$ be any prime number that divides $m$. But then $p$ divides $d^2$ and thus also $d$. It then follows that $p$ divides $1$ and we have a contradiction.
\end{subproof}

\begin{claim*}[\mathrm{B.i.}4]\label{claim14}
If $\delta \geq 2$, we reach a contradiction.
\end{claim*}

\begin{subproof}[Proof of Claim~$(\mathrm{B.i.4})$]
Since $\pi$ is a $(-1)$-tower resolution of $C$, the base-point $p_{i+1}$ is the unique intersection point between $C_i$ and $E_i$, for $i = k+m-1,\ldots,k+m+\delta-1$. The configuration after these blow-ups is shown in Figure~$\ref{fig:triplecase2}$.
\begin{figure}[ht]
\centering
\begin{tikzpicture}[scale=0.75]
\draw (7,0) -- (8.5,2.5);
\node at (7.4,1.25) {\scriptsize $E_{1}$};
\draw[dashed] (8,2.5) -- (9.5,0);
\draw (9,0) -- (10.4,2.5);
\node at (9.73,1.25) {\scriptsize $E_{k}[-3]$};
\draw[dashed] (10,2.5) -- (11.5,0);
\draw (11,0) -- (12.5,2.5);
\node at (11.72,1.25) {\scriptsize $E_{k+m-2}$};

\draw (12,2.1) -- (15,2.1);
\node at (13.35,2.37) {\scriptsize $E_{k+m-1}$};

\node at (16.6,1.25) {\scriptsize $E_{k+1}[-m+1]$};

\draw (14.5,2.5) -- (16,0);

\draw (13.25,2.25)--(14,1);
\node at (13.2,1.5) {\scriptsize $E_{k+m}$};
\draw (14,1.25)--(13.25,0);
\node at (12.86,0.7) {\scriptsize $E_{k+m+1}$};
\draw[dashed] (13.25,0.25)--(14,-1);
\draw[very thick] (13.6,-0.8)--(15.5,-0.8);
\node at (14.9,-1.11) {\scriptsize $E_{k+m+\delta}$};

\draw[very thick] (15,-1.1) to [out=90, in=195] (16,0.7);
\node at (15.95,-0.25) {\scriptsize $C_{k+m+\delta}$};

\end{tikzpicture}
\caption{Case $l=1$, $\delta \geq 2$.}\label{fig:triplecase2}
\end{figure}
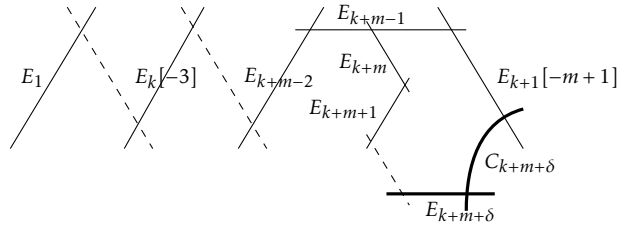\\
Since no more base-point of $\pi$ can lie on $E_{k+m}$, its strict transform in $X$ has self-intersection $-2$. If $m>3$, then $E_{k+m-1}$ intersects the two $(-2)$-curves $E_{k+m-2}$ and $E_{k+m}$ in $X$. But $\eta$ contracts $E_{k+m-1}$ before these two curves and thus this situation is not possible and we have $m=3$. Since $d < 3m = 9$ by Lemma~$\ref{Lem:inequalities2}$, the multiplicity sequence of $C$ is in Table~$\ref{table:sequences}$ and can only be $(3_{(3)},2)$ in degree $6$. In this case $\delta = 5$. But this implies that $E_{k+m+1}$ is also a $(-2)$-curve in $X$. We hence get a contradiction after $\eta$ contracts $E_{k+m-1}$. Then the image of $E_{k+m}$ intersects the $(-2)$-curves $E_k$ and $E_{k+m+1}$.
\end{subproof}

This concludes (i) of part (B).

\bigskip

(ii) Suppose now that $C_{k+1}$ does not pass through the intersection point between $E_k$ and $E_{k+1}$. Then $C_{k+1}$ intersects $E_{k+1}$ in one point with intersection multiplicity $m-1$, otherwise there would be a loop in the configuration of the curves $E_1,\ldots,E_{n-1},C_n$. The configuration of curves in $X_{k+1}$ is shown in the left part of Figure~$\ref{fig:ii1}$. Since $C_n$ and $E_{k+1}$ do not intersect in $X$, it follows that the base-point $p_{i+1}$ for $i=k+1,\ldots,k+m-1$ is the unique intersection point between $C_i$ and $E_i$, which also lies on $E_{k+1}$. The configuration of curves in $X_{k+m}$ is shown in the right part of Figure~$\ref{fig:ii1}$. We denote the self-intersection of $C_{k+m}$ by $\delta$ and this number is equal to $d^2 - km^2 - (m-1)^2 - (m-1)$. Since $\pi$ is a $(-1)$-tower resolution of $C$, it follows that $\delta \geq -1$.

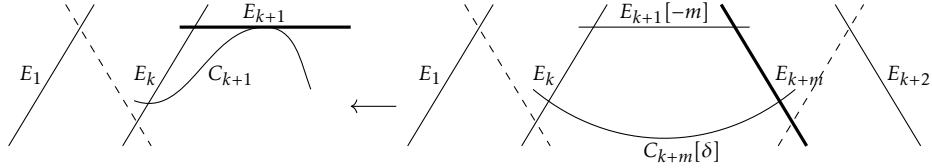
\begin{figure}[ht]
\centering
\begin{tikzpicture}[scale=0.75]
\draw (0,0) -- (1.5,2.5);
\node at (0.4,1.25) {\scriptsize $E_{1}$};
\draw[dashed] (1,2.5) --(2.5,0);
\draw (2,0) -- (3.5,2.5);
\node at (2.4,1.25) {\scriptsize $E_{k}$};

\draw[very thick] (3,2.1) -- (6,2.1);
\node at (4.5,2.35) {\scriptsize $E_{k+1}$};

\draw (2.2,0.8) to [out=-20, in=180] (4.5,2.1);
\draw (4.5,2.1) to [out=0, in=110] (5.3,1);

\node at (3.9,1.25) {\scriptsize $C_{k+1}$};

\draw[<-] (6,0.7) -- (6.8,0.7);

\draw (7,0) -- (8.5,2.5);
\node at (7.4,1.25) {\scriptsize $E_{1}$};
\draw[dashed] (8,2.5) -- (9.5,0);
\draw (9,0) -- (10.5,2.5);
\node at (9.4,1.25) {\scriptsize $E_{k}$};

\draw (10,2.1) -- (13,2.1);
\node at (11.5,2.35) {\scriptsize $E_{k+1}[-m]$};

\draw[very thick] (12.5,2.5) -- (14,0);
\draw[dashed] (13.5,0) -- (15,2.5);
\draw (14.5,2.5) -- (16,0);
\node at (13.88,1.25) {\scriptsize $E_{k+m}$};
\node at (15.8,1.25) {\scriptsize $E_{k+2}$};

\draw (9.2,1) to [out=-40, in=220] (13.8,1);
\node at (11.8,-0.1) {\scriptsize $C_{k+m}[\delta]$};

\end{tikzpicture}
\caption{Blow-up of $p_{k+2},\ldots,p_{k+m}$.}\label{fig:ii1}
\end{figure}

In the surface $X$, let $E \neq E_k$ in $\{E_1,\ldots,E_n\}$ be a curve that intersects $C_n$. We know that the map $\eta$ first contracts $C_n$ and then the chain $E_k,\ldots,E_1$. Since $k\geq 2$, it follows that the image of $E$ is tangent to $E_{k+1}$, after these contractions. This implies that $E$ is not contracted by $\eta$ and thus $E=E_n$ is the last exceptional curve in the $(-1)$-tower resolution $\pi$. We now discuss what happens for different values of $\delta$.

\begin{claim*}[\mathrm{B.ii.}1]\label{claim15}
If $\delta = -1$, we reach a contradiction.
\end{claim*}

\begin{subproof}[Proof of Claim~$(\mathrm{B.ii.}1)$]
In this case we already have a $(-1)$-tower resolution of $C$. This resolution must be $\pi$, since there is no more base-point on $C_{k+m}$ and $C_n$ intersects $E_n$. But we observe that the curves $E_1,\ldots,E_{k+m-1},C_{k+m}$ are not connected and thus cannot be the contracted locus of $\eta$. Hence $\delta = -1$ is not possible.
\end{subproof}

\begin{claim*}[\mathrm{B.ii.}2]\label{claim16}
If $\delta = 0$, we reach a contradiction.
\end{claim*}

\begin{subproof}[Proof of Claim~$(\mathrm{B.ii.}2)$]
The base-point $p_{k+m+1}$ is the unique intersection point between $C_{k+m}$ and $E_{k+m}$. After this blow-up, we have a $(-1)$-tower resolution of $C$, which must be $\pi$, for the same reason as in the case $\delta = -1$. The configuration of curves is shown in Figure~$\ref{fig:ii2}$.

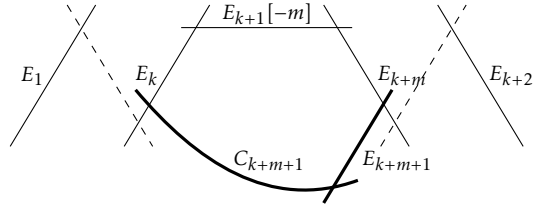
\begin{figure}[ht]
\centering
\begin{tikzpicture}[scale=0.75]
\draw (0,0) --(1.5,2.5);
\node at (0.4,1.25) {\scriptsize $E_{1}$};
\draw[dashed] (1,2.5) --(2.5,0);
\draw (2,0) --(3.5,2.5);
\node at (2.4,1.25) {\scriptsize $E_{k}$};

\draw (3,2.1) --(6,2.1);
\node at (4.5,2.35) {\scriptsize $E_{k+1}[-m]$};

\draw (5.5,2.5) --(7,0);
\node at (6.88,1.25){\scriptsize $E_{k+m}$};
\draw[dashed] (6.5,0) --(8,2.5);
\draw (7.5,2.5) --(9,0);
\node at (8.8,1.25){\scriptsize $E_{k+2}$};

\draw[very thick] (6.7,1) --(5.5,-1);
\node at (6.8,-0.25){\scriptsize $E_{k+m+1}$};

\draw[very thick] (2.2,1) to [out=-50, in=200] (6.1,-0.6);
\node at (4.55,-0.25) {\scriptsize $C_{k+m+1}$};

\end{tikzpicture}
\caption{Case $l=1$, $\delta = 0$.}\label{fig:ii2}
\end{figure}

The map $\eta$ contracts first $C_{k+m+1}$ and then the chain $E_k,\ldots,E_1$. After these contractions the self-intersection of the image of $E_{k+1}$ is $-m+k$, but must also be $-1$ and hence $k=m-1$. From $\delta = 0$ we then obtain the equation $d^2 = m(m^2-1)$. Since $m$ and $m^2-1$ are coprime, they are both squares, as $d>0$. But if $m \geq 2$ is a square, then $m^2-1$ is not a square. Hence the only integer solutions to the equation are $(d,m) = (0,-1), (0,0), (0,1)$, and thus $\delta = 0$ is also not possible.
\end{subproof}

\begin{claim*}[\mathrm{B.ii.}3]
If $\delta \geq 1$, we reach a contradiction.
\end{claim*}

\begin{subproof}[Proof of Claim~$(\mathrm{B.ii.}3)$]
For $i=l+m,\ldots, l+m+\delta$, the base-point $p_{i+1}$ is the unique intersection point between $C_i$ and $E_i$. After these blow-ups we have a $(-1)$-tower resolution of $C$, which has to be $\pi$ for the same reason as in the previous cases. The configuration of curves is shown in Figure~$\ref{fig:ii5}$.
\begin{figure}[ht]
\centering
\begin{tikzpicture}[scale=0.75]
\draw (0,0) --(1.5,2.5);
\node at (0.4,1.25) {\scriptsize $E_{1}$};
\draw[dashed] (1,2.5) --(2.5,0);
\draw (2,0) --(3.5,2.5);
\node at (2.4,1.25) {\scriptsize $E_{k}$};

\draw (3,2.1) --(6,2.1);
\node at (4.5,2.35) {\scriptsize $E_{k+1}[-m]$};

\draw (5.5,2.5) --(7,0);
\node at (5.75,1.25){\scriptsize $E_{k+m}$};
\draw (6.5,0) --(8,2.5);
\node at (8.05,1.25){\scriptsize $E_{k+m-1}$};
\draw[dashed] (7.5,2.5) --(9,0);
\draw (8.5,0) --(10,2.5);
\node at (9.8,1.25){\scriptsize $E_{k+2}$};

\draw (6.7,1) --(5.5,-1);
\node at (6.8,-0.25){\scriptsize $E_{k+m+1}$};
\draw[dashed] (6,-1) --(4.8,1);
\draw[very thick] (5.3,1) --(4.1,-1);
\node at (4.17,0.85){\scriptsize $E_{k+m+\delta+1}$};

\draw[very thick] (2.2,1) to [out=-60, in=200] (4.9,-0.5);
\node at (3.78,0.3) {\scriptsize $C_{k+m+\delta+1}$};

\end{tikzpicture}
\caption{Case $l=1$, $\delta \geq 1$.}\label{fig:ii5}
\end{figure}
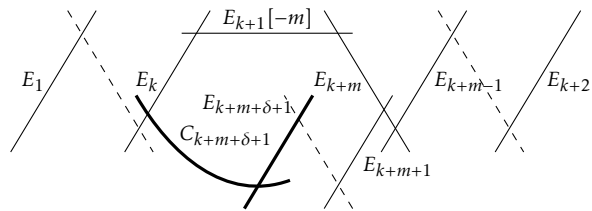

Since $\delta \geq 1$, the curve $E_{l+m+1}$ has self-intersection $-2$. But we know that $\eta$ contracts $E_{k+m}$ before the $(-2)$-curves $E_{l+m-1}$ and $E_{l+m+1}$, which leads to a contradiction.
\end{subproof}

This concludes (ii) of part (B) and hence finishes the proof of Proposition~$\ref{Prop:1jump}$.
\end{proof}

\begin{remark}\label{remark:5555554}
{\rm Below we see the configuration of exceptional curves of a resolution of a non-extendable isomorphism between two curves of degree $13$ with multiplicity sequence $(5_{(6)},4)$. All the unlabeled curves have self-intersection $-2$. Starting with either of the $(-1)$-curves, one can successively contract all curves in this configuration, except the other $(-1)$-curve. The image of this curve in $\p^2$, denoted $C$, then has self-intersection $169 = 13^2$. It remains to be verified whether such curves exist and whether new counterexamples to Conjecture~$\ref{conjecture:yoshihara}$ may arise in this way. We remark that $C \setminus \Sing(C) \simeq \A^1 \setminus \{0\}$ and thus $C$ is different from the unicuspidal examples of degree $13$ constructed in \cite{Cos12}.

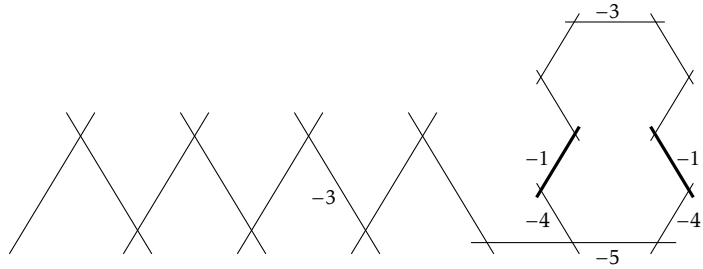
\begin{center}
\begin{tikzpicture}[scale=0.75]
\draw (-2,0)--(-0.5,2.5);
\draw (-1,2.5) --(0.5,0);
\draw (0,0) -- (1.5,2.5);
\draw (1,2.5) --(2.5,0);
\draw (2,0) -- (3.5,2.5);
\draw (3,2.5) -- (4.5,0);
\node at (3.52,1) {\scriptsize $-3$};
\draw (4,0) -- (5.5,2.5);
\draw (5,2.5) -- (6.5,0);

\draw (6.1,0.2) -- (9.7,0.2);
\node at (8.5,-0.05) {\scriptsize $-5$};

\draw (9.25,0) -- (10,1.25);
\node at (7.27,0.6) {\scriptsize $-4$};
\draw[very thick] (10,1)--(9.25,2.25);
\node at (7.27,1.7) {\scriptsize $-1$};
\draw (9.25,2)--(10,3.25);
\draw (10,3)--(9.25,4.25);

\draw (8,0)--(7.25,1.25);
\node at (9.92,0.6) {\scriptsize $-4$};
\draw[very thick] (7.25,1)--(8,2.25);
\node at (9.92,1.7) {\scriptsize $-1$};
\draw (8,2)--(7.25,3.25);
\draw (7.25,3)--(8,4.25);

\draw (7.75,4.1)--(9.5,4.1);
\node at (8.5,4.3) {\scriptsize $-3$};

\end{tikzpicture}
\end{center}
}
\end{remark}

\begin{corollary}\label{cor:1jumpcase}
Let $C \subset \p^2$ be an irreducible curve with one of the multiplicity sequences $(3,2_{(3)})$, $(3_{(2)},2_{(4)})$, $(3_{(3)},2)$, $(3_{(4)},2_{(3)})$, $(4,3_{(3)})$, $(4,3_{(5)})$, $(4_{(2)},3_{(3)})$, or $(4_{(3)},3)$. Then either $C$ is unicuspidal or any open embedding $\p^2 \setminus C \hookrightarrow \p^2$ extends to an automorphism of $\p^2$.
\end{corollary}

\begin{proof}
This is a direct consequence of Proposition~$\ref{Prop:1jump}$.
\end{proof}

\begin{remark}
{\rm Note that in Corollary~$\ref{cor:1jumpcase}$, only curves with the multiplicity sequences $(3_{(3)},2)$ and $(4_{(3)},3)$ can be unicuspidal.}
\end{remark}

\begin{proposition}\label{Prop:alleven}
Let $C \subset \p^2$ be a rational curve of degree $d$ and multiplicity sequence $(m_1,\ldots,m_k)$ such that all multiplicities are even and there exists $l<k$ such that $m_{l+1} = \ldots = m_k = 2$ and $m_{j} < m_{j+1}+\ldots+m_k$ for all $j \leq l$. Let $\varphi \colon \p^2 \setminus C \hookrightarrow \p^2$ be an open embedding that does not extend to an automorphism of $\p^2$. Then $C$ is unicuspidal.
\end{proposition}

\begin{proof}
Suppose that $C$ is not unicuspidal. By Proposition~$\ref{Prop:nojump}$, we can assume that the multiplicity sequence of $C$ is non-constant. By Lemma~$\ref{Lem:tower}$, there exists $\pi \colon X = X_n \xrightarrow{\pi_n} \ldots \xrightarrow{\pi_2} X_1 \xrightarrow{\pi_1} X_0 = \p^2$ a $(-1)$-tower resolution  of $C$ with base-points $p_1,\ldots,p_n$ and exceptional curves $E_1,\ldots,E_n$, and a $(-1)$-tower resolution $\eta \colon X \to \p^2$ of some curve $D \subset \p^2$ such that $\varphi \circ \pi = \eta$. Then $E_1 \cup \ldots \cup E_{n-1} \cup C_n$ is the exceptional locus of $\eta$, being the support of an SNC-divisor that has a tree structure. The composition $\pi_1 \circ \ldots \circ \pi_k$ is the minimal resolution of singularities of $C$. For $i = 1,\ldots,k$, we obtain the following intersection numbers, by Lemma~$\ref{Lem:proximate mult}$:
$$C_k \cdot E_i = m_i - \sum_{p_j \succ p_i} m_j.$$
In particular, $C_k \cdot E_k = 2$. Since $m_{j} < m_{j+1}+\ldots+m_k$ for all $j \leq l$, it follows that, for $i=1,\ldots,l$, the curves $E_i$ and $E_k$ do not intersect in $X_k$ and hence also not in $X$. Since all $m_i$ are even, it follows that the intersection numbers $C_k \cdot E_i$ are even. It follows moreover that the intersection numbers $C_n \cdot E_i$ are also even for $i=1,\ldots,l$, since $E_k$ and $E_i$ do not intersect in $X_k$. The curve $E_1\cup\ldots \cup E_{n-1} \cup C_n$ is SNC and therefore $E_i$ and $C_n$ do not intersect at all, for $i=1,\ldots,l$. Since the multiplicities $m_{l+1},\ldots,m_k$ are all equal to $2$, it follows that $C_k$ does not intersect any of the curves $E_1,\ldots,E_{k-1}$, but only $E_k$. Since $C$ is not unicuspidal, the curves $C_k$ and $E_k$ intersect in two distinct points. We denote by $\delta$ the self-intersection of $C_k$, which is given by $\delta = d^2 - \sum_{i=1}^k m_i^2$. Since $C$ has a $(-1)$-tower resolution, we have $\delta \geq -1$.

\begin{claim*}[1]
If $\delta = -1$, we reach a contradiction.
\end{claim*}

\begin{subproof}[Proof of Claim~$(1)$]
We already have a $(-1)$-tower resolution of $C$ (see Figure~$\ref{fig:even1}$). Since $C_k$ and $E_k$ intersect in two points and there is no more base-point on $C_k$, there is no more base-point at all. But we observe that $C_k$ and $E_1 \cup \ldots \cup E_{k-1}$ are not connected. This is not possible and hence $\delta$ must be $\geq 0$.

\begin{figure}[ht]
\centering
\begin{tikzpicture}[scale=0.75]
\draw[dashed] (2,0) -- (3.5,2.5);

\draw[very thick] (3,2.1) -- (6,2.1);
\node at (5.65,2.35) {\scriptsize $E_{k}$};

\draw[very thick] (3.2,0.5) to [out=70, in=180] (4.5,2.6);
\draw[very thick] (4.5,2.6) to [out=0, in=110] (5.8,0.5);

\node at (3.75,1.25) {\scriptsize $C_{k}$};

\end{tikzpicture}
\caption{Case $\delta = -1$.}\label{fig:even1}
\end{figure}
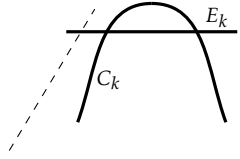
\end{subproof}

\begin{claim*}[2]
If $\delta = 0$, we reach a contradiction.
\end{claim*}

\begin{subproof}[Proof of Claim~$(2)$]
The genus-degree formula yields
$$d^2 - 3d + 2 = \sum_{i=1}^k m_i(m_i-1).$$
Using $\delta = 0$, we get $3d-2 = \sum_{i=1}^k m_i$. This identity implies that $d$ is even. We can thus find the equations
\begin{align*}
\left(\frac{d}{2}\right)^2 & = \sum_{i=1}^k \left(\frac{m_i}{2}\right)^2 ,\\
3\left(\frac{d}{2}\right)+1 &= \sum_{i=1}^k \frac{m_i}{2}.
\end{align*}
Adding these identities yields
\begin{align*}
\frac{d}{2}\left(\frac{d}{2}+3\right) + 1 &= \sum_{i=1}^k \frac{m_i}{2}\left(\frac{m_i}{2}+1\right).
\end{align*}
The left-hand side of this equation is odd, whereas the right-hand side is even. This is a contradiction and thus $\delta = 0$ is not possible.
\end{subproof}

\begin{claim*}[3]
If $\delta = 1$, we reach a contradiction.
\end{claim*}

\begin{subproof}[Proof of Claim~$(3)$]
The base-point $p_{k+1}$ is one of the intersection points between $C_k$ and $E_k$. The curve $C_{k+1}$ has then self-intersection $0$ in $X_{k+1}$ and thus the base-point $p_{k+2}$ is the unique intersection point between $C_{k+1}$ and $E_{k+1}$. The configuration of curves in $X_{k+2}$ is shown in Figure~$\ref{fig:even2}$. In the surface $X$, the curve $E_k$ has self-intersection $-2$. This implies that $\eta$ first contracts $C_n$ and then $E_{k},\ldots,E_1$, in this order. By assumption, the multiplicity sequence of $C$ is non-constant. This implies that there exists a curve $E_j$ with $j<k$ that intersects $3$ other exceptional curves. But this implies that the image of $E_{k+1}$, after contracting $C_n,E_k,\ldots,E_1$, is singular and hence cannot be contracted. We thus reach a contradiction and conclude that $\delta \neq 1$.

\begin{figure}[ht]
\centering
\begin{tikzpicture}[scale=0.75]
\draw (0,0) -- (1.5,2.5);
\draw (1,0) -- (2.5,2.5);
\draw (1,2.5) -- (2.5,0);
\draw[dashed] (2,0) -- (3.5,2.5);
\node at (2.25,0.9) {\scriptsize $E_{j}$};

\draw (3,2.1) -- (6,2.1);
\node at (4.5,2.35) {\scriptsize $E_{k}$};

\draw[very thick] (5.5,2.5) -- (7,0);
\node at (5.75,1.25) {\scriptsize $E_{k+1}$};
\draw (6,-0.83) -- (8,2.5);
\node at (7.8,1.25){\scriptsize $E_{k+2}$};

\draw[very thick] (4,2.3) to [out=-90, in=180] (6.6,-0.3);
\node at (4.8,-0.02) {\scriptsize $C_{k+2}$};

\end{tikzpicture}
\caption{Case $\delta = 1$.}\label{fig:even2}
\end{figure}
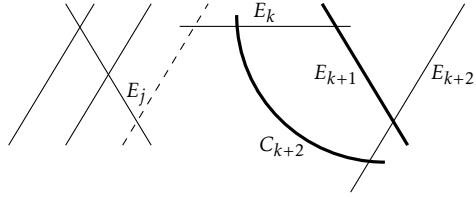
\end{subproof}

\begin{claim*}[4]
If $\delta \geq 2$, we reach a contradiction.
\end{claim*}

\begin{subproof}[Proof of Claim~$(4)$]
Again, the base-point $p_{k+1}$ is one of the intersection points between $C_k$ and $E_k$. Since $\pi$ is a $(-1)$-tower resolution of $C$, it follows that for $i=k+1,\ldots,k+\delta$, the base-point $p_{i+1}$ is the unique intersection point between $C_k$ and $E_k$ (see Figure~$\ref{fig:even3}$). This implies that in $X$, the curve $E_{k+1}$ has self-intersection $-2$. We observe that $E_k$ also intersects the $(-2)$-curve $E_{k-1}$ in $X$. Since $\eta$ contracts $E_k$ before $E_{k-1}$ and $E_{k+1}$, this leads to a contradiction.

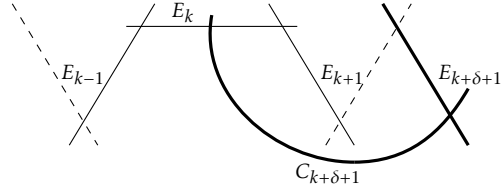
\begin{figure}[ht]
\centering
\begin{tikzpicture}[scale=0.75]
\draw[dashed] (1,2.5) -- (2.5,0);
\draw (2,0) --(3.5,2.5);
\node at (2.25,1.25){\scriptsize $E_{k-1}$};

\draw (3,2.1) --(6,2.1);
\node at (4,2.35) {\scriptsize $E_{k}$};

\draw (5.5,2.5) --(7,0);
\node at (6.8,1.25){\scriptsize $E_{k+1}$};
\draw[dashed] (6.5,0) --(8,2.5);
\draw[very thick] (7.5,2.5) --(9,0);
\node at (9.05,1.25){\scriptsize $E_{k+\delta+1}$};

\draw[very thick] (4.5,2.3) to [out=-100, in=180] (7,-0.3);
\draw[very thick] (7,-0.3) to [out=0, in=240] (9,1);
\node at (6.55,-0.5) {\scriptsize $C_{k+\delta+1}$};

\end{tikzpicture}
\caption{Case $\delta \geq 2$.}\label{fig:even3}
\end{figure}
\end{subproof}
This concludes the proof of Proposition~$\ref{Prop:alleven}$.
\end{proof}

\begin{corollary}\label{cor:proximatecase}
Let $C \subset \p^2$ be an irreducible curve with one of the multiplicity sequences $(4,2_{(4)})$, $(4_{(3)},2_{(3)})$, or $(6,2_{(6)})$. If $C$ is not unicuspidal, then any open embedding $\p^2 \setminus C \hookrightarrow \p^2$ extends to an automorphism of $\p^2$.
\end{corollary}

\begin{proof}
This is a direct consequence of Proposition~$\ref{Prop:alleven}$.
\end{proof}

\subsection{A special sextic curve and the proof of Theorem~$\ref{Thm:degree8}$}

\begin{proposition}\label{prop:deg6equivalent}
Let $C \subset \p^2$ be a curve of degree $6$ and multiplicity sequence $(3,2_{(7)})$ such that there exists $\varphi \colon \p^2 \setminus C \to \p^2 \setminus D$ an isomorphism, where $D \subset \p^2$ is a curve. Then $C$ and $D$ are projectively equivalent.
\end{proposition}

\begin{proof}
If $\varphi$ extends to an automorphism of $\p^2$ the claim is trivial, so we assume this is not the case. Then by Lemma~$\ref{Lem:tower}$, there exists a $(-1)$-tower resolution $\pi \colon X \to \p^2$ of $C$ and a $(-1)$-tower resolution $\eta \colon X \to \p^2$ of $D$ such that $\eta = \varphi \circ \pi$. The curve $C$ has $8$ singular points $p_1,\ldots,p_8$, where $p_{i+1}$ lies in the first neighborhood of $p_i$ for $i=1,\ldots,7$. The map $\pi$ is a $(-1)$-tower resolution of $C$ and thus blows up the points $p_1,\ldots,p_8$. We denote by $E_i$ the exceptional curve of the blow-up of $p_i$, for $i=1,\ldots,8$. After blowing up these $8$ points, the strict transform $\hat{C}$ of $C$ has self-intersection $6^2 - 3^2 - 7 \cdot 2^2 = -1$. We observe that $\hat{C}$ and $E_8$ intersect with multiplicity $2$. Since no other base-point of $\pi$ lies on $\hat{C}$, it follows that also the strict transforms of $\hat{C}$ and $E_8$ intersect with multiplicity $2$ in $X$. But this means that $E_8$ is not contracted by $\eta$. It follows that $E_8$ is the last exceptional curve of $\pi$ and $\eta(E_8) = D$.

By B\'ezout's theorem the points $p_1, p_2, p_3$ are not collinear and hence there exists a conic $Q_1 \subset \p^2$ that passes through $p_1,\ldots,p_5$. Again by B\'ezout's theorem, it follows that $C$ and $Q_1$ intersect transversely in some proper point of $\p^2$ that is different from $p_1$. It then follows that the strict transform $\hat{Q}_1$ of $Q_1$ in $X$ transversely intersects $E_5$ and $\hat{C}$. By symmetry there also exists a conic $Q_2 \subset \p^2$ whose strict transform $\hat{Q}_2$ by $\eta$ intersects $E_3$ and $\hat{D}$ transversely. The configuration of curves in $X$ is shown below.

 \bigskip
 
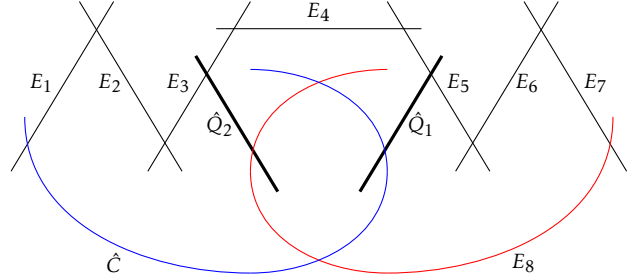
\begin{center}
\begin{tikzpicture}[scale=0.9]
\draw (0,0)--(1.5,2.5);
\draw (1,2.5) -- (2.5,0);
\draw (2,0) --(3.5,2.5);
\node at (2.45,1.3){\scriptsize $E_3$};
\node at (0.45,1.3){\scriptsize $E_1$};
\node at (1.45,1.3){\scriptsize $E_2$};

\draw (3,2.1) --(6,2.1);
\node at (4.5,2.35) {\scriptsize $E_4$};

\draw (5.5,2.5) --(7,0);
\node at (6.55,1.3){\scriptsize $E_5$};
\draw (6.5,0) --(8,2.5);
\draw (7.5,2.5) --(9,0);
\node at (8.55,1.3){\scriptsize $E_7$};
\node at (7.55,1.3){\scriptsize $E_6$};

\node at (3.05,0.7){\scriptsize $\hat{Q}_2$};
\node at (6,0.7){\scriptsize $\hat{Q}_1$};

\draw[very thick] (2.7,1.7)--(3.9,-0.3);
\draw[very thick] (6.3,1.7) --(5.1,-0.3);

\draw[red] (3.5,0) to [out=90, in=180] (5.5,1.5);
\draw[red] (3.5,0) to [out=-90, in=180] (5.5,-1.5);
\draw[red] (5.5,-1.5) to [out=0, in=-90] (8.8,0.8);

\draw[blue] (5.5,0) to [out=90, in=0] (3.5,1.5);
\draw[blue] (5.5,0) to [out=-90, in=0] (3.5,-1.5);
\draw[blue] (3.5,-1.5) to [out=180, in=-90] (0.2,0.8);

\node at (7.5,-1.35) {\scriptsize $E_8$};
\node at (1.5,-1.35) {\scriptsize $\hat{C}$};

\end{tikzpicture}
\end{center}

 \bigskip

To see that $\hat{Q}_1$ and $\hat{Q}_2$ do not intersect in $X$, we observe that $\pi$ sends $\hat{Q}_2$ to a rational quartic curve with multiplicity sequence $(2_{(3)})$ and singular points $p_1,p_2,p_3$. It then follows that
$$\hat{Q}_1\cdot \hat{Q}_2 =  Q_2 \cdot \pi(\hat{Q}_2)  - 2 - 2 - 2 - 1 - 1 = 0.$$
Moreover, the curves $\hat{Q}_1$ and $\hat{Q}_2$ both have self-intersection $-1$ in $X$. We can thus construct a morphism $\rho$ by contracting the curves $\hat{Q}_2, E_3, E_2, E_1$ and $\hat{Q}_1, E_5, E_6, E_7$. The rank of the Picard group of $X$ is $9$, and hence the rank of the Picard group of the image of $\rho$ is $1$. It thus follows that $\rho$ is a morphism $X \to \p^2$. The images of $\hat{C}, E_4$ and $E_8$ all have self-intersection $4$ and are thus smooth conics in $\p^2$. The curves $\rho(E_4)$ and $\rho(\hat{C})$ intersect in two distinct points $p,q \in \p^2$, with multiplicity $1$ in $p$ and multiplicity $3$ in $q$. The curves $\rho(E_4)$ and $\rho(E_8)$ also intersect in $p$ and $q$, but with multiplicity $3$ in $p$ and multiplicity $1$ in $q$. The configuration of the $3$ conics is shown below.

 \bigskip
 
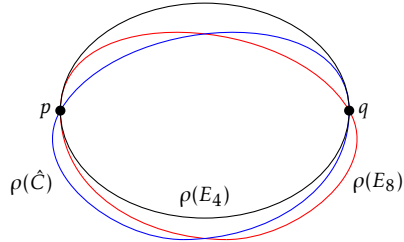
\begin{center}
\begin{tikzpicture}[scale=0.95]

\draw[red] (-2,0) to [out=-90, in=180] (0.3,-1.8);
\draw[red] (-2,0) to [out=90, in=120] (2,0);
\draw[red] (0.3,-1.8) to [out=0, in=-60] (2,0);

\draw[blue] (2,0) to [out=90, in=60] (-2,0);
\draw[blue] (-2,0) to [out=-120, in=180] (-0.3,-1.8);
\draw[blue] (-0.3,-1.8) to [out=0, in=-90] (2,0);

\draw (0,0) ellipse (2 and 1.5);

\fill (-2,0) circle (2pt);
\fill (2,0) circle (2pt);

\node at (0,-1.2){\scriptsize $\rho(E_4)$};
\node at (-2.4,-1){\scriptsize $\rho(\hat{C})$};
\node at (2.4,-1){\scriptsize $\rho(E_8)$};

\node at (-2.2,0){\scriptsize $p$};
\node at (2.2,0){\scriptsize $q$};

\end{tikzpicture}
\end{center}

 \bigskip
 
\noindent Up to a linear change of coordinates, we can assume that the smooth conic $\rho(E_4)$ has equation $xz+y^2 = 0$ and the points $p$ and $q$ are $[1:0:0]$ and $[0:0:1]$ respectively. Conics that pass through the points $[1:0:0]$ and $[0:0:1]$ are of the form
$$ay^2 + bxy + cxz + dyz = 0$$
where $a,b,c,d \in \k$. A smooth conic with this equation intersects $xz+y^2=0$ with multiplicity $3$ in $[1:0:0]$ if and only if $a=c \neq 0$, $b=0$ and $d \neq 0$. Thus there exists some $\lambda \in \k^*$ such that $\rho(\hat{C})$ has equation $xz + y^2 + \lambda yz = 0$. Analogously, there exists $\mu \in \k^*$ such that $\rho(E_8)$ has equation $xz + y^2 + \mu yz = 0$.

We then find $\theta \in \Aut(\p^2)$ that sends a point $[x:y:z]$ to $[\frac{\lambda}{\mu}z:y:\frac{\mu}{\lambda}x]$. Thus $\theta$ preserves the conic $xz+y^2=0$ and exchanges $\rho(\hat{C})$ and $\rho(E_8)$. It follows that $\hat{\theta} \coloneqq \rho^{-1}\circ \theta \circ \rho$ is an automorphism of $X$ that exchanges $\hat{C}$ and $E_8$ and sends $E_i$ to $E_{8-i}$ for $i = 2,\ldots,7$. But then $\eta \circ \hat{\theta} \circ \pi^{-1}$ is an automorphism of $\p^2$ that sends $C$ to $D$, and hence $C$ and $D$ are projectively equivalent.
\end{proof}

Before we are able to prove Theorem~$\ref{Thm:degree8}$, we need to look at one more special case.

\begin{lemma}\label{lemma:522222}
Let $C \subset \p^2$ be a curve of degree $7$ and multiplicity sequence $(5,2_{(5)})$. Then every open embedding $\p^2 \setminus C \hookrightarrow \p^2$ extends to an automorphism of $\p^2$.
\end{lemma}

\begin{proof}
Suppose that there exists an open embedding $\varphi \colon \p^2 \setminus C \hookrightarrow \p^2$ that does not extend to an automorphism of $\p^2$. Then by Lemma~$\ref{Lem:tower}$, there exists $\pi \colon X = X_n \xrightarrow{\pi_n} \ldots \xrightarrow{\pi_2} X_1 \xrightarrow{\pi_1} X_0 = \p^2$ a $(-1)$-tower resolution of $C$ with base-points $p_1,\ldots,p_n$ and exceptional curves $E_1,\ldots,E_n$, and a $(-1)$-tower resolution $\eta \colon X \to \p^2$ of some curve $D \subset \p^2$ such that $\varphi \circ \pi = \eta$. Then $E_1 \cup \ldots \cup E_{n-1} \cup C_n$ is the exceptional locus of $\eta$, being the support of an SNC-divisor that has a tree structure. By Lemma~$\ref{Lem:proximate mult}$, we obtain the intersection number
$$C_n \cdot E_1 = m_i - \sum_{p_j \succ p_1} m_j.$$
Thus either $C_n \cdot E_1 =3$ or $C_n \cdot E_1 =1$. Since $C_n$ can intersect $E_1$ only transversely in at most one point, we conclude that $C_n \cdot E_1 =1$ and that $p_3$ is proximate to $p_1$. For the first $6$ blow-ups of $\pi$, we then obtain the configuration of curves illustrated below.

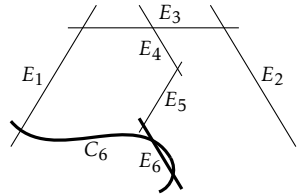
\begin{center}
\begin{tikzpicture}[scale=0.75]

\draw (2,0) --(3.5,2.5);
\node at (2.4,1.25) {\scriptsize $E_{1}$};
\draw (3,2.1) --(6,2.1);
\node at (4.8,2.35) {\scriptsize $E_3$};
\draw (5.5,2.5) --(7,0);
\node at (6.6,1.25) {\scriptsize $E_2$};

\draw (4.25,2.5) --(5,1.25);
\node at (4.45,1.7) {\scriptsize $E_4$};
\draw (5,1.5) --(4.25,0.25);
\node at (4.9,0.8) {\scriptsize $E_5$};
\draw[very thick] (4.25,0.5) --(5,-0.75);
\node at (4.45,-0.28) {\scriptsize $E_6$};

\draw[very thick] (2,0.45) to [out=-50, in=120] (4.8,-0.2);
\draw[very thick] (4.8,-0.2) to [out=300, in=20] (4.6,-0.8);
\node at (3.5,-0.07) {\scriptsize $C_6$};

\end{tikzpicture}
\end{center}
The curves $E_2$ and $E_4$ have self-intersection $-2$ in $X$ since the resolution $\pi$ is obtained by blowing up more points on $E_6$. Moreover, the map $\eta$ contracts $E_3$ before $E_2$ and $E_4$, but this leads to a contradiction.
\end{proof}

We are now ready to give the proof of the second main result.

\begin{proof}[Proof of Theorem~$\ref{Thm:degree8}$]
We assume that $C$ is not a line, conic, or a nodal cubic. We can also assume that $C$ is rational and has a unique proper singular point with one of the multiplicity sequences in Table~$\ref{table:sequences}$, by Corollary~$\ref{Cor:table}$. Otherwise, $\varphi$ extends to an automorphism of $\p^2$. If $C$ is unicuspidal, then $C$ and $D$ are projectively equivalent by Corollary~$\ref{cor:unicuspidal8}$. If $C$ is not unicuspidal, then $\varphi$ extends to an automorphism of $\p^2$ by Corollary~$\ref{cor:jumpcase}$, Corollary~$\ref{cor:equalcase}$, Corollary~$\ref{cor:1jumpcase}$, Corollary~$\ref{cor:proximatecase}$, and Lemma~$\ref{lemma:522222}$, except when $C$ is of degree $6$ with multiplicity sequence $(3,2_{(7)})$ or $C$ is of degree $8$ with multiplicity sequence $(3_{(7)})$. If $C$ has multiplicity sequence $(3,2_{(7)})$, the claim follows from Proposition~$\ref{prop:deg6equivalent}$. If $C$ has multiplicity sequence $(3_{(7)})$, then $C \setminus \Sing(C)$ is isomorphic to $\A^1 \setminus \{0\}$, by Proposition~$\ref{Prop:nojump}$.
\end{proof}

\begin{remark}
{\rm For all known examples of irreducible curves $C \subset \p^2$ that have non-extendable open embeddings $\p^2 \setminus C \hookrightarrow \p^2$, we have that $C \setminus \Sing(C) \simeq \p^1 \setminus \{p_1,\ldots,p_k\}$, where $k \in \{1,2,3,9\}$. There are only very few known non-unicuspidal examples. Do there exist examples for any $k \in \N$?}
\end{remark}

\subsection{A counterexample of degree 8}\label{sec:deg8}
It follows from Theorem~$\ref{Thm:degree8}$ that if two irreducible curves $C, D \subset \p^2$ of degree $\leq 8$ are counterexamples to Conjecture~$\ref{conjecture:yoshihara}$, then $C$ and $D$ are of degree $8$ and have multiplicity sequence $(3_{(7)})$. In this section, we show that such counterexamples do indeed exist. First we need the following auxiliary construction.

 \begin{lemma}\label{lemma:conicconfiguration}
 We denote the conic
 $$\Lambda \colon xy + xz + yz = 0$$
 and for $\lambda \in \k \setminus \{0,-1\}$ the conics
\begin{align*}
\Gamma_\lambda &\colon x^2 - (1+\lambda)xy - \lambda xz - (1+\lambda)yz = 0, \\
\Delta_\lambda &\colon z^2 - \left(1+\frac{1}{\lambda}\right)xy - \frac{1}{\lambda}xz - \left(1+\frac{1}{\lambda}\right)yz = 0.
\end{align*}
Then the curves $\Lambda$, $\Gamma_\lambda$ and $\Delta_\lambda$ intersect in $[0:1:0]$ with multiplicity $3$ for each pair. Moreover, the curves
\begin{itemize}
\item $\Lambda$ and $\Gamma_\lambda$ intersect in $[0:0:1]$,
\item $\Lambda$ and $\Delta_\lambda$ intersect in $[1:0:0]$,
\item $\Gamma_\lambda$ and $\Delta_\lambda$ intersect in $[\lambda:0:1]$,
\end{itemize}
and in no other point apart from $[0:1:0]$. The configuration of these conics is shown below.

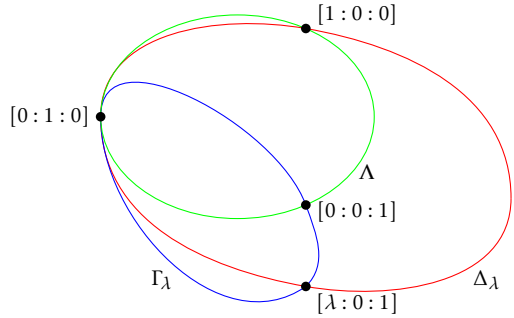
\begin{center}
\begin{tikzpicture}[scale=0.9]

\draw[red] (-2,0) to [out=90, in=170] (1,1.3);
\draw[red] (-2,0) to [out=-90, in=170] (1,-2.5);
\draw[red] (1,1.3) [out=-10,in=90] to (4,-1.2);
\draw[red] (4,-1.2) [out=-90, in=-10] to (1,-2.5);

\draw[blue] (-2,0) to [out=90, in=110] (1,-1.3);
\draw[blue] (1,-1.3) to [out=-70, in=45] (1,-2.5);
\draw[blue] (-2,0) to [out=-90,in=-145] (1,-2.5);

\draw[green] (0,0) ellipse (2 and 1.5);

\fill (1,-2.5) circle (2pt);
\fill (1,-1.3) circle (2pt);
\fill (1,1.3) circle (2pt);
\fill (-2,0) circle (2pt);

\node at (1.9,-0.8){\scriptsize $\Lambda$};
\node at (3.65,-2.4){\scriptsize $\Delta_\lambda$};
\node at (-1.1,-2.4){\scriptsize $\Gamma_\lambda$};

\node at (-2.75,0){\scriptsize $[0:1:0]$};
\node at (1.75,1.5){\scriptsize $[1:0:0]$};
\node at (1.75,-1.4){\scriptsize $[0:0:1]$};
\node at (1.75,-2.75){\scriptsize $[\lambda:0:1]$};

\end{tikzpicture}
\end{center}

\noindent Furthermore, there exists an automorphsim of $\p^2$ that preserves $\Lambda$ and exchanges $\Gamma_\lambda$ and $\Delta_\lambda$ if and only if $\lambda = 1$.

\end{lemma}

\begin{proof}
The curves $\Lambda$, $\Gamma_\lambda$ and $\Delta_\lambda$ are given by explicit equations and it is a straightforward computation to determine the intersection points and multiplicities.

To prove the last claim, suppose that $\theta \in \Aut(\p^2) = \PGL_3(\k)$ preserves $\Lambda$ and exchanges $\Gamma_\lambda$ and $\Delta_\lambda$. Then $\theta$ fixes $[0:1:0]$ and exchanges $[1:0:0]$ and $[0:0:1]$. These conditions imply that $\theta$ is of the form $[x:y:z] \mapsto [\alpha z : y : \beta x]$, for some $\alpha, \beta \in \k^*$. The image of $\Lambda$ under $\theta$ then has equation $\beta xy + \alpha \beta xz + \alpha yz = 0$. Since $\Lambda$ is preserved, it follows that $\alpha = \beta = \alpha \beta$ and hence $\alpha = \beta = 1$. The map $\theta$ also fixes the intersection point $[\lambda : 0 : 1]$ between $\Gamma_\lambda$ and $\Delta_\lambda$. Since $\theta([\lambda:0:1]) = [1:0:\lambda]$, it follows that $\lambda = 1$. For the converse, suppose that $\lambda = 1$. Then the automorphism $[x:y:z] \mapsto [z:y:x]$ preserves $\Lambda$ and exchanges $\Gamma_1$ and $\Delta_1$.
\end{proof}

\begin{proof}[Proof of Theorem~$\ref{theorem:counterexample}$]
With the same notations as in Lemma~$\ref{lemma:conicconfiguration}$, we choose some $\lambda \in \k \setminus \{0,\pm 1\}$ and conics $\Lambda$, $\Gamma = \Gamma_\lambda$, $\Delta = \Delta_\lambda$. We denote moreover by $L_y$ the line $y=0$ and by $L_\lambda$ the line through $[0:1:0]$ and $[\lambda : 0 : 1]$. The line $L_\lambda$ has equation $x - \lambda z = 0$ and intersects $\Lambda$ in the points $[0:1:0]$ and $[1+\lambda:-1:1+\frac{1}{\lambda}]$. The configuration of these curves in shown below.

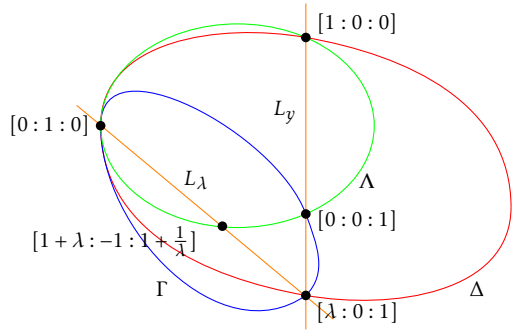
\begin{center}
\begin{tikzpicture}[scale=0.9]

\draw[red] (-2,0) to [out=90, in=170] (1,1.3);
\draw[red] (-2,0) to [out=-90, in=170] (1,-2.5);
\draw[red] (1,1.3) [out=-10,in=90] to (4,-1.2);
\draw[red] (4,-1.2) [out=-90, in=-10] to (1,-2.5);

\draw[blue] (-2,0) to [out=90, in=110] (1,-1.3);
\draw[blue] (1,-1.3) to [out=-70, in=45] (1,-2.5);
\draw[blue] (-2,0) to [out=-90,in=-145] (1,-2.5);

\draw[green] (0,0) ellipse (2 and 1.5);

\draw[orange] (1,1.8)--(1,-3);
\draw[orange] (-2.35,0.3)--(1.4,-2.85);

\fill (1,-2.5) circle (2pt);
\fill (1,-1.3) circle (2pt);
\fill (1,1.3) circle (2pt);
\fill (-2,0) circle (2pt);
\fill (-0.22,-1.48) circle (2pt);

\node at (1.9,-0.8){\scriptsize $\Lambda$};
\node at (3.5,-2.4){\scriptsize $\Delta$};
\node at (-1.1,-2.4){\scriptsize $\Gamma$};
\node at (0.7,0.2){\scriptsize $L_y$};
\node at (-0.6,-0.8){\scriptsize $L_\lambda$};

\node at (-2.75,0){\scriptsize $[0:1:0]$};
\node at (1.75,1.5){\scriptsize $[1:0:0]$};
\node at (1.75,-1.4){\scriptsize $[0:0:1]$};
\node at (1.75,-2.75){\scriptsize $[\lambda:0:1]$};
\node at (-1.8,-1.7){\scriptsize $[1+\lambda:-1:1+\frac{1}{\lambda}]$};

\end{tikzpicture}
\end{center}

We then blow up the points $[1:0:0]$, $[0:0:1]$ and $[\lambda:0:1]$, with exceptional curves $E_1$, $E_2$, and $E_3$ respectively. The configuration after these blow-ups is shown below. By abuse of notation, we use the same names for the strict transforms of all curves. Curves with self-intersection $-1$ are drawn with thick lines and all other self-intersection numbers are indicated, except if they are $-2$.

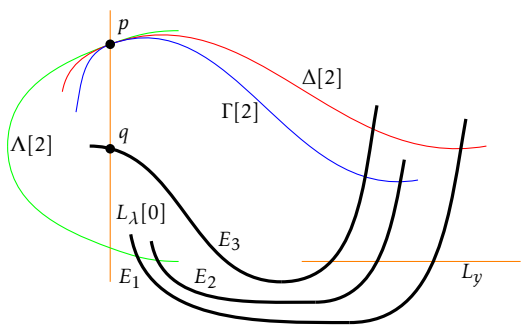
\begin{center}
\begin{tikzpicture}[scale=0.9]

\draw[orange] (1,2)--(1,-2);
\draw[green] (-0.5,0) to [out=90,in=200] (1,1.5);
\draw[green] (-0.5,0) to [out=-90,in=160] (1,-1.5);
\draw[green] (1,1.5) to [out=20,in=180] (2,1.7);
\draw[green] (1,-1.5) to [out=-20,in=180] (2,-1.7);
\node at (-0.15,0){\scriptsize $\Lambda[2]$};
\node at (1.45,-1){\scriptsize $L_\lambda[0]$};
\node at (6.3,-1.9){\scriptsize $L_y$};

\draw[orange] (3.8,-1.7)--(7,-1.7);

\draw[red] (0.3,0.8) to [out=80,in=200] (1,1.5);
\draw[red] (1,1.5) to [out=20,in=190] (6.5,0);
\node at (4.1,1){\scriptsize $\Delta[2]$};

\draw[blue] (0.5,0.5) to [out=80,in=200] (1,1.5);
\draw[blue] (1,1.5) to [out=20,in=190] (5.5,-0.5);
\node at (2.9,0.5){\scriptsize $\Gamma[2]$};

\draw[very thick] (0.7,0) to [out=0,in=180] (3.5,-2);
\draw[very thick] (3.5,-2) to [out=0,in=-100] (4.9,0.6);

\draw[very thick] (1.6,-1.4) to [out=-80,in=180] (4,-2.3);
\draw[very thick] (4,-2.3) to [out=0,in=-100] (5.3,-0.2);

\draw[very thick] (1.3,-1.3) to [out=-80,in=180] (4.5,-2.6);
\draw[very thick] (4.5,-2.6) to [out=0,in=-100] (6.2,0.4);

\node at (2.75,-1.4){\scriptsize $E_3$};
\node at (2.4,-1.95){\scriptsize $E_2$};
\node at (1.3,-1.95){\scriptsize $E_1$};

\node at (1.2,1.75){\scriptsize $p$};
\node at (1.2,0.15){\scriptsize $q$};

\fill (1,1.5) circle (2pt);
\fill (1,-0.04) circle (2pt);

\end{tikzpicture}
\end{center}

Next, we blow up the intersection point $q$ between $L_\lambda$ and $E_3$, with exceptional curve $E_4$. The curves $\Gamma$, $\Delta$ and $\Lambda$ each intersect with multiplicity $3$ in the point $p$. We then blow up $p$ and two points proximate to $p$ (with exceptional curves $E_5$, $E_6$, $E_7$) so that the strict transforms of $\Gamma$, $\Delta$ and $\Lambda$ are disjoint. We thus obtain the following configuration of curves.

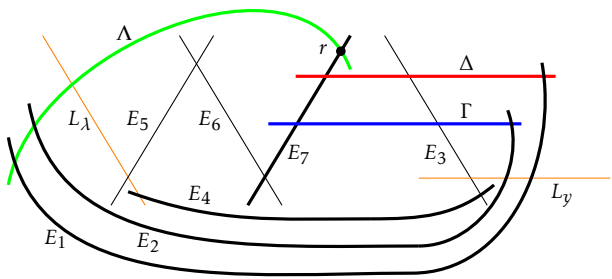
\begin{center}
\begin{tikzpicture}[scale=0.9]

\draw[orange] (-1,2.5) --(0.5,0);
\draw (0,0) -- (1.5,2.5);
\draw (1,2.5) --(2.5,0);
\draw[very thick] (2,0) -- (3.5,2.5);
\draw (4,2.5) --(5.5,0);

\draw[very thick,red] (2.7,1.9)--(6.5,1.9);
\node at (5.2,2.1){\scriptsize $\Delta$};
\draw[very thick,blue] (2.3,1.2)--(6,1.2);
\node at (5.2,1.4){\scriptsize $\Gamma$};

\draw[orange] (4.5,0.4)--(7.3,0.4);
\node at (-0.45,1.25){\scriptsize $L_\lambda$};

\draw[very thick,green] (-1.5,0.3) to [out=80,in=110] (3.5,2);
\node at (0.2,2.55){\scriptsize $\Lambda$};
\node at (6.6,0.15){\scriptsize $L_y$};

\draw[very thick] (-1.5,1) to [out=-80,in=180] (4.4,-1);
\draw[very thick] (4.4,-1) to [out=0,in=-80] (6.3,2.1);

\draw[very thick] (-1.2,1.5) to [out=-80,in=180] (4.5,-0.6);
\draw[very thick] (4.5,-0.6) to [out=0,in=-70] (5.8,1.4);

\draw[very thick] (0.25,0.2) to [out=-20,in=180] (3.8,-0.2);
\draw[very thick] (3.8,-0.2) to [out=0,in=220] (5.6,0.3);

\node at (-0.8,-0.5){\scriptsize $E_1$};
\node at (0.55,-0.55){\scriptsize $E_2$};
\node at (4.75,0.75){\scriptsize $E_3$};
\node at (1.3,0.15){\scriptsize $E_4$};
\node at (0.4,1.25){\scriptsize $E_5$};
\node at (1.45,1.25){\scriptsize $E_6$};
\node at (2.75,0.75){\scriptsize $E_7$};

\node at (3.1,2.3){\scriptsize $r$};

\fill (3.36,2.27) circle (2pt);

\end{tikzpicture}
\end{center}

Finally, we blow up the intersection point $r$ between $\Lambda$ and $E_7$ and two points proximate to $r$, with exceptional curves $E_8$, $E_9$, $E_{10}$, and obtain the configuration shown below. We denote the surface obtained after these blow-ups by $X$ and denote the composition of all 10 blow-ups by $\rho \colon X \to \p^2$. The curves $E_1$, $E_2$, $E_4$, $E_{10}$ are dashed and unlabeled because they will not be used for what follows.

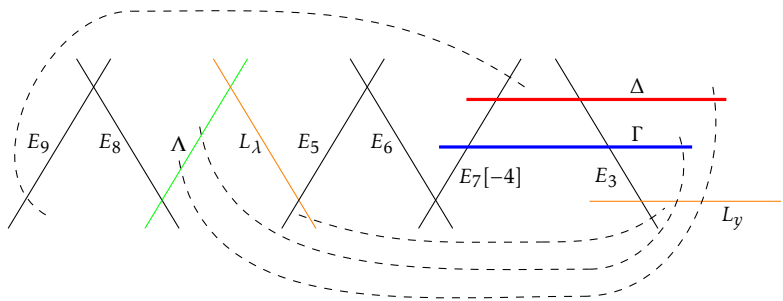
\begin{center}
\begin{tikzpicture}[scale=0.9]
\draw (-4,0)--(-2.5,2.5);
\draw (-3,2.5) --(-1.5,0);
\draw[green] (-2,0)--(-0.5,2.5);
\draw[orange] (-1,2.5) --(0.5,0);
\draw (0,0) -- (1.5,2.5);
\draw (1,2.5) --(2.5,0);
\draw (2,0) -- (3.5,2.5);
\draw (4,2.5) --(5.5,0);
\node at (-0.45,1.25){\scriptsize $L_\lambda$};
\node at (-1.5,1.25){\scriptsize $\Lambda$};
\node at (6.6,0.15){\scriptsize $L_y$};

\draw[very thick, red] (2.7,1.9)--(6.5,1.9);
\node at (5.2,2.1){\scriptsize $\Delta$};
\draw[very thick, blue] (2.3,1.2)--(6,1.2);
\node at (5.2,1.4){\scriptsize $\Gamma$};

\draw[orange] (4.5,0.4)--(7.3,0.4);

\draw[dashed] (-3.45,0.2) to [out=150,in=180] (-2.2,3.2);
\draw[dashed] (-2.2,3.2) to [out=0,in=145] (3.6,2.05);

\draw[dashed] (-1.5,1) to [out=-80,in=180] (4.4,-1);
\draw[dashed] (4.4,-1) to [out=0,in=-80] (6.3,2.1);

\draw[dashed] (-1.2,1.5) to [out=-80,in=180] (4.5,-0.6);
\draw[dashed] (4.5,-0.6) to [out=0,in=-70] (5.8,1.4);

\draw[dashed] (0.25,0.2) to [out=-20,in=180] (3.8,-0.2);
\draw[dashed] (3.8,-0.2) to [out=0,in=220] (5.6,0.3);

\node at (4.75,0.75){\scriptsize $E_3$};

\node at (0.4,1.25){\scriptsize $E_5$};
\node at (1.47,1.25){\scriptsize $E_6$};
\node at (3.05,0.75){\scriptsize $E_7[-4]$};
\node at (-2.5,1.25){\scriptsize $E_8$};
\node at (-3.55,1.25){\scriptsize $E_9$};

\end{tikzpicture}
\end{center}

The rank of the Picard group of $X$ is $11$, since this surface is obtained from $\p^2$ by $10$ blow-ups.
We can now find a morphism $\pi \colon X \to \p^2$, by contracting the $10$ curves $\Delta$, $E_3$, $L_y$, $E_7$, $E_6$, $E_5$, $L_\lambda$, $\Lambda$, $E_8$, $E_9$, in this order. The image $C \coloneqq \pi(\Gamma)$ is then a curve of degree $8$ in $\p^2$ with multiplicity sequence $(3_{(7)})$. Likewise, we find a morphism $\eta \colon X \to \p^2$, where we first contract $\Gamma$ instead of $\Delta$. The image $D \coloneqq \eta(\Delta)$ is then also a curve of degree $8$ with multiplicity sequence $(3_{(7)})$. The complements $\p^2 \setminus C$ and $\p^2 \setminus D$ are both isomorphic to the complement of the union of the curves $\Gamma$, $\Delta$, $E_3$, $L_y$, $E_7$, $E_6$, $E_5$, $L_\lambda$, $\Lambda$, $E_8$, $E_9$ in $X$.

Suppose now that $C$ and $D$ are projectively equivalent, i.e. there exists $\theta \in \PGL_3(\k)$ with $\theta(C) = D$. We observe that the base-points of $\pi$ are completely determined by $C$, since $\pi$ is the minimal SNC-resolution of $C$ followed by the blow-up of the unique intersection point between $E_3$ and $E_7$. Likewise, the base-points of $\eta$ are determined by $D$. It follows that $\hat{\theta} \coloneqq \eta^{-1}\circ \theta \circ \pi$ defines an automorphism of $X$ that exchanges $\Gamma$ and $\Delta$ and preserves the other exceptional curves. But then $\hat{\theta}$ induces an automorphism of $\p^2$ (via $\rho$) that exchanges the conics $\Gamma, \Delta \subset \p^2$ and preserves $\Lambda$, $L_y$ and $L_\lambda$. But this is not possible by Lemma~$\ref{lemma:conicconfiguration}$, since we have chosen $\lambda \neq 1$. We thus reach a contradiction and conclude that $C$ and $D$ are not projectively equivalent.
\end{proof}

\begin{remark}
{\rm The construction in the proof of Theorem~$\ref{theorem:counterexample}$ also works if the base-field $\k$ is not algebraically closed, except if $\k$ has only $2$ or $3$ elements. In these cases we cannot choose $\lambda \in \k \setminus \{0,\pm1\} = \varnothing$.}
\end{remark}

\bibliographymark{References}

\providecommand{\bysame}{\leavevmode\hbox to3em{\hrulefill}\thinspace}
\providecommand{\arXiv}[2][]{\href{https://arxiv.org/abs/#2}{arXiv:#1#2}}
\providecommand{\MR}{\relax\ifhmode\unskip\space\fi MR }
\providecommand{\MRhref}[2]{%
  \href{http://www.ams.org/mathscinet-getitem?mr=#1}{#2}
}
\providecommand{\href}[2]{#2}

\end{document}